%% file: main.tex
\documentclass[twoside]{article}

\usepackage[accepted]{aistats2019}
\usepackage[utf8]{inputenc} 
\usepackage[T1]{fontenc}    
\usepackage[pagebackref=true,colorlinks,bookmarks=false]{hyperref}
\usepackage{graphicx}
\usepackage{subfigure}
\usepackage{booktabs} 
\usepackage{microtype}      
\usepackage{url}            
\usepackage{amsfonts,eufrak}       
\usepackage{algorithm}
\usepackage{algorithmic}
\usepackage{amsmath}
\usepackage{graphicx}
\usepackage{multirow}
\usepackage{bm}
\usepackage{color}
\usepackage{amsthm}
\usepackage{tabularx}
\usepackage{footnote}
\usepackage{threeparttable}
\usepackage{tablefootnote}
\makesavenoteenv{tabular}
\makesavenoteenv{table}
\usepackage{caption}
\usepackage{fancyhdr}
\usepackage{amssymb}
\usepackage{pifont}
\usepackage{xr}
\usepackage{cite}
\usepackage{natbib}

\newcommand{\comment}[1]{}

\newcommand{\et}{{~\emph{et al.}~}}
\newcommand{\eg}{{~\emph{e.g.}~}}
\newcommand{\ie}{{\emph{i.e.}}}

\input{MACPdef-ams_manu.tex}
\input{MACPdef_manu.tex}

\begin{document}

%

%

\twocolumn[

\aistatstitle{Faster First-Order Methods for Stochastic Non-Convex Optimization on Riemannian Manifolds}

\aistatsauthor{ Pan Zhou$^{*}$ \And Xiao--Tong Yuan$^{\dagger}$ \And  Jiashi Feng$^{*}$}

\aistatsaddress{ $\ ^{*}$National University of Singapore $\quad$   $\ ^{\dagger}$Nanjing University of Information Science $\&$ Technology\\
 pzhou@u.nus.edu $\qquad\qquad\qquad\ \ $ xtyuan@nuist.edu.cn $\qquad\qquad\qquad\ \ $ elefjia@nus.edu.sg} ]

\begin{abstract}
SPIDER (Stochastic Path Integrated Differential EstimatoR) is an efficient gradient estimation technique developed for non-convex stochastic optimization. Although having been shown to attain nearly optimal computational complexity bounds, the SPIDER-type methods are limited to linear metric spaces.  In this paper, we introduce the Riemannian SPIDER (R-SPIDER) method as a novel nonlinear-metric extension of SPIDER for efficient non-convex optimization on Riemannian manifolds. We prove that for finite-sum problems with $n$ components, R-SPIDER converges to an $\epsilon$-accuracy stationary point within $\mathcal{O}\big(\min\big(n+\frac{\sqrt{n}}{\epsilon^2},\frac{1}{\epsilon^3}\big)\big)$ stochastic gradient evaluations, which is sharper in magnitude than the prior Riemannian first-order methods. For online optimization, R-SPIDER is shown to converge with $\mathcal{O}\big(\frac{1}{\epsilon^3}\big)$ complexity which is, to the best of our knowledge, the first non-asymptotic result for online Riemannian optimization. Especially, for gradient dominated functions, we further develop a variant of R-SPIDER and prove its linear convergence rate. Numerical results demonstrate the computational efficiency of the proposed methods.
\end{abstract}

\section{Introduction}
We consider the following \emph{finite-sum} and \emph{online} non-convex problems on a Riemannian manifold $\M$:
\begin{equation}\label{generalproblem}
\min_{\xm\in\M} f(\xm):=
\begin{cases}
\frac{1}{n} \sum_{i=1}^{n} f_i(\xm) \!\!\!\!&\text{(finite-sum)} \\
 \EE [f(\xm;\pi)] \!\!\!\!&\text{(online)}
\end{cases},
\end{equation}
where $f: \M \mapsto \mathbb{R}$ is a smooth non-convex loss function. For the finite-sum problem, each individual loss $f_{i}(\xm)$ is associated with the $i$-th sample, while in online setting, the stochastic component $f(\xm;\pi)$ is indexed by a random variable $\pi$. Such a formulation encapsulates several important finite-sum problems and their corresponding online counterparts, including principle component analysis (PCA)~\cite{wold1987principal}, low-rank matrix/tensor completion/recovery~\cite{tan2014riemannian,vandereycken2013low,mishra2014r3mc,kasai2016low}, dictionary learning~\cite{cherian2017riemannian,sun2017complete}, Gaussian mixture models~\cite{hosseini2015matrix} and low-rank multivariate  regression~\cite{meyer2011linear}, to name a few.

\begin{table*}[tp]
		\begin{threeparttable}[b]
			\caption{Comparison of IFO complexity for different Riemannian first-order stochastic optimization  algorithms on the noncovnex problem~\eqref{generalproblem} under finite-sum and online settings. The $\epsilon$-accuracy solution is measured by the
expected gradient norm $\EE\left[\|\nabla f(\xm)\|\right]\leq \epsilon$. Here $L$, $\sigma$ and $\zeta$ respectively denote the gradient Lipschitz constant, the gradient variance and the curvature parameter of the Riemannian manifold (see Section~\ref{Preliminaries}).}
			\setlength{\tabcolsep}{20pt} 
\label{comparisontable}
\begin{center}
{ \footnotesize {
\begin{tabularx}{\textwidth}{c|c|cc}\toprule
							&\multirow{2}{*}{}&\multicolumn{2}{c}{ Non-convex Problem}  \\
							&\multirow{2}{*}{}&\multicolumn{1}{c}{general non-convex }&\multicolumn{1}{c}{ $\tau$-gradient dominated} \\
							\midrule
\multirow{5}{*}{Finite-sum} &							R-SRG~\cite{kasai2018riemannian}& $\Oc{n+\frac{L^2}{\epsilon^4}}$&   $\Oc{(n+\tau^2L^2)\log\lb\frac{1}{\epsilon}\rb}$\\

&R-SVRG~\cite{zhang2016riemannian}& $\mathcal{O}\Big(n+\frac{\zeta n^{\frac{2}{3}}}{\epsilon^2}\Big)$& $\Oc{(n+\tau L \zeta^{\frac{1}{2}}n^{\frac{2}{3}})\log\lb\frac{1}{\epsilon}\rb}$ \\
						
& this work& $\mathcal{O}\left(\min\left(n+\frac{L  \sqrt{n}}{\epsilon^2},\frac{L\sigma}{\epsilon^3}\right)\right)$ &  $\mathcal{O}\left( \min\left(\left(n+ \tau L \sqrt{n}\right)\log\left(\frac{1}{\epsilon} \right),\frac{\tau L \sigma}{\epsilon}\right) \right)$  \\
							\midrule
						
	\multirow{1}{*}{Online}&				this work& $\Oc{ \frac{ L\sigma}{\epsilon^3} }$&  $\Oc{ \frac{\tau L\sigma}{\epsilon} }$ \\
							\bottomrule
						\end{tabularx}
				}}
			\end{center}
		\end{threeparttable}
\end{table*}

One classic approach for solving problem~\eqref{generalproblem} (or its convex counterpart) is to take it as a constrained optimization problem in ambient Euclidean space and find the minimizers via projected (stochastic) gradient descent~\cite{oja1992principal,da1998geodesic,badeau2005fast}. This kind of methods, however, tend to suffer from high computational cost as projection onto certain manifolds (e.g., positive-definite matrices) could be expensive in large-scale learning problems~\cite{zhang2016riemannian}.

As an appealing alternative, the Riemannian optimization methods have recently gained wide attention in machine learning~\citep{zhang2018estimate,zhang2016first,bonnabel2013stochastic,kasai2018riemannian, zhang2016riemannian,kasai2016riemannian,kasai2018riemanniana}. In contrast to the Euclidean-projection based methods, the Riemannian methods directly move the iteration along the geodesic path towards the optimal solution, and thus can better respect the geometric structure of the problem in hand. Specifically, the Riemannian gradient methods have the following recursive form:
\begin{equation}\label{updatingform}
\xmi{k+1}=\Exp{\xmi{k}}{-\eta_{k} \gmi{k}},
\end{equation}
where $\gmi{k}$ is the gradient estimate of the full Riemannian gradient $\nabla f(\xmi{k})$, $\eta_k$ denotes the learning rate, and the exponential mapping $\Exp{\xm}{\ym}$, as defined in Section~\ref{Preliminaries}, maps $\ym$ in the tangent space at $\xm$ to $\Exp{\xm}{\ym}$ on the manifold $\M$ along a proper geodesic curve. For instance, Riemannian gradient descent (R-GD) uses the full Riemannian gradient $\gmi{k}=\nabla f(\xmi{k})$ in Eqn.~\eqref{updatingform} and has been shown to have sublinear rate of convergence in geodesically convex problems~\cite{zhang2016first}. To boost efficiency, Liu\et\cite{liu2017accelerated} and Zhang\et\cite{zhang2018estimate} further introduced the Nesterov acceleration techniques~\cite{nesterov2013introductory} into R-GD with convergence rate significantly improved for geodesically convex functions.

To avoid the time-consuming full gradient computation required in R-GD, Riemannian stochastic optimization algorithms~\cite{bonnabel2013stochastic,kasai2016riemannian, zhang2016riemannian,kasai2018riemanniana,kasai2018riemannian} leverage the decomposable (finite-sum) structure of problem~\eqref{generalproblem}. For instance, Bonnabel\et\cite{bonnabel2013stochastic} proposed R-SGD that  only evaluates gradient of one (or a mini-batch) randomly selected sample for variable update per iteration. Though with good iteration efficiency, R-SGD converges  slowly  as it uses decaying learning rate for convergence guarantee due to its gradient variance. To tackle this issue, Riemannian stochastic variance-reduced gradient (R-SVRG) algorithms~\cite{zhang2016riemannian,kasai2018riemanniana} adapt SVRG~\cite{SVRG} to problem~\eqref{generalproblem}. Benefiting from the variance-reduced technique, R-SVRG  converges more stably and efficiently than R-SGD. More recently, inspired by the variance-reduced stochastic recursive gradient approach~\cite{nguyen2017sarah,nguyen2017stochastic}, the Riemannian stochastic recursive gradient (R-SRG) algorithm~\cite{kasai2018riemannian} establishes a recursive equation to estimate the full Riemannian gradient so that the computational efficiency can be further improved (see Table~\ref{comparisontable}).

SPIDER (Stochastic Path Integrated Differential EstimatoR)~\cite{fang2018spider} is a recursive estimation method developed for tracking the history full gradients with
significantly reduced computational cost. By combining SPIDER with normalized gradient methods, nearly optimal iteration complexity bounds can be attained for non-convex optimization in Euclidean space~\cite{fang2018spider}. Though appealing in vector space problems, it has not been explored for non-convex optimization in nonlinear metric spaces such as Riemannian manifold.

In this paper, we introduce the \textit{Riemannian Stochastic Path Integrated Differential EstimatoR} (R-SPIDER) as a simple yet efficient extension of the SPIDER from Euclidean space to Riemannian manifolds. Specifically, for a proper positive integer $p$, at each time instance $k$ with $\mod(k,p)\equiv0$, R-SPIDER first samples a large data batch $\SSm_1$ and estimates the initial full Riemannian gradient $\nabla f(\xmi{k})$ as $\vmti{k}=\nabla f_{\SSm_1}(\xmi{k})=\frac{1}{|\SSm_1|} \sum_{i\in\SSm_1} f_i(\xmi{k})$. Then at each of the next $p-1$ iterations, it samples a smaller mini-batch $\SSm_2$ and estimates/tracks $\nabla f(\xmi{k})$:
\begin{equation}\label{upadtetionform}
\vmti{k}= \fsi{2}(\xmi{k})-\Para{\xmi{k-1}}{\xmi{k}}{\fsi{2}(\xmi{k-1})-\vmti{k-1}},
\end{equation}
where the parallel transport $\Para{\xm}{\zm}{\ym}$ (as defined in Section~\ref{Preliminaries}) transports $\ym$ from the tangent space at $\xm$ to that at the point $\zm$. Here the parallel transport operation is necessary since $\fsi{2}(\xmi{k-1})$ and $\fsi{2}(\xmi{k})$ are  located in different tangent spaces. 
Given the gradient estimate $\vmti{k}$, the variable is updated via normalized gradient descent $
\xmi{k+1}= \mbox{Exp}_{\xmi{k}}\big(-\etai{k} \frac{\vmti{k}}{\|\vmti{k}\|}\big).
$
Note that R-SRG~\cite{kasai2018riemannian} applies a similar recursion form as in~\eqref{upadtetionform} for full Riemannian gradient estimation, and the core difference between their method and ours lies in that R-SPIDER is equipped with gradient normalization which is missing in R-SRG. Then by carefully setting the learning rate $\eta$ and mini-batch sizes of $\SSm_1$ and $\SSm_2$, R-SPIDER only needs to sample a necessary number of data points for accurately estimating Riemannian gradient and sufficiently decreasing the objective at each iteration. In this way, R-SPIDER achieves sharper bounds of  incremental first order oracle (IFO, see Definition~\ref{def:IFO}) complexity than R-SRG and other state-of-the-art Riemannian non-convex optimization methods.

Table~\ref{comparisontable} summarizes our main results on the computational complexity of R-SPIDER for non-convex problems, along with those for the above mentioned Riemannian gradient algorithms. The following are some highlighted advantages of our results over the state-of-the-arts.

For the finite-sum setting of problem~\eqref{generalproblem} with general non-convex functions, the IFO complexity of R-SPIDER to achieve $\EE\left[\|\nabla f(\xm)\|\right]\leq \epsilon$ is $\mathcal{O}\big(\min\big(n+\frac{L  \sqrt{n}}{\epsilon^2},\frac{L\sigma}{\epsilon^3}\big)\big)$ which matches the lower IFO complexity bound in Euclidean space~\cite{fang2018spider}. By comparison, the IFO complexity bounds of R-SRG and R-SVRG are $\mathcal{O}\big(n+\frac{L^2}{\epsilon^4}\big)$ and $\mathcal{O}\big(n+\frac{\zeta n^{\frac{2}{3}}}{\epsilon^2}\big)$, respectively. It can be verified that R-SPIDER improves over R-SRG by a factor of $\mathcal{O}\big(\frac{1}{\epsilon}\big)$ and R-SVRG by a factor $\mathcal{O}\big(n^{1/6}\big)$ regardless of the relation between $n$ and $\epsilon$.

When $f(\xm)$ is a $\tau$-gradient dominated function with finite-sum structure, R-SPIDER enjoys the IFO complexity of $\mathcal{O}\big( \min\big(\big(n+ \tau L \sqrt{n}\big)\log\big(\frac{1}{\epsilon} \big),\frac{\tau L \sigma}{\epsilon}\big) \big)$ which is again lower than the $\mathcal{O}\big( \big(n+\tau L\zeta^{\frac{1}{2}}n^{\frac{2}{3}}\big) \log\big(\frac{1}{\epsilon}\big)\big)$ bound for R-SVRG by a factor of $\Oc{n^{1/6}}$. Note that our IFO complexity is not dependent on the curvature parameter $\zeta (\geq 1)$ of the manifold $\M$, because our analysis does not involve the geodesic trigonometry inequality on a manifold. To compare with R-SRG with complexity bound $\mathcal{O}\big( \big(n+\tau^2 L^2\big) \log\big(\frac{1}{\epsilon}\big)\big)$, R-SPIDER is more efficient than R-SRG in large-sample-moderate-accuracy settings, e.g., in cases when $n$ dominates $1/\epsilon$.

For the online version of problem~~\eqref{generalproblem}, we  establish the IFO complexity bounds $\mathcal{O}\big(\frac{L\sigma}{\epsilon^3}\big)$ and $\Oc{ \frac{\tau L\sigma}{\epsilon}  }$ for generic non-convex and gradient dominated problems, respectively. To our best knowledge, these non-asymptotic convergence results are novel to non-convex online Riemannian optimization. Comparatively, Bonnabel\et\cite{bonnabel2013stochastic} only provided asymptotic convergence analysis of R-SGD: the iterating sequence generated by R-SGD converges to a critical point when the iteration number approaches infinity.

Finally, our analysis reveals as a byproduct that R-SPIDER provably benefits from mini-batching. Specifically, our theoretic results imply linear speedups in parallel computing setting for large mini-batch sizes. We are not aware of any similar linear speedup results in the prior Riemannian stochastic  algorithms.

\section{Preliminaries}\label{Preliminaries}
Throughout this paper, we assume that the Riemannian manifold $(\M,\g)$ is a real smooth manifold $\M$ equipped with a Riemannian metric $\g$. We denote the induced inner product $\la \ym, \zm\ra$ of any two vectors $\ym$ and  $\zm$ in the tangent space  $\Tgi{\xm}\M$ at the point $\xm$ as $\la \ym, \zm\ra=\g(\ym,\zm)$, and denote the norm $\|\ym\|$ as $\|\ym\|=\sqrt{\g(\ym,\ym)}$. Let $\nabla f_i(\xm)$ be the stochastic Riemannian gradient of  $f_i(\xm)$ and also be a unbiased estimate to the full Riemannian gradient $\nabla f(\xm)$, i.e. $\EE_i [\nabla f_i(\xm)]=\nabla f(\xm)$.

The exponential mapping $\Exp{\xm}{\ym}$ maps $\ym\in\Tgi{\xm}\M$ to $\zm\in \M$ such that there is a geodesic $\gamma(t)$ with $\gamma(0)=\xm$, $\gamma(1)=\zm$ and $\dot{\gamma}(0)=\frac{{d}}{{d}t} \gamma(t)=\ym$. Here the geodesic $\gamma(t)$ is a constant speed curve $\gamma:[0,1]\rightarrow \M$ which is locally distance minimized. If there exists a unique geodesic between any two points on $\M$, then the exponential map has an inverse mapping $\iExpp_{\xm}: \M\rightarrow \Tgi{\xm}\M$ and the geodesic is the unique shortest path with the geodesic distance $\dis{\xm}{\zm}=\|\iExp{\xm}{\zm}\|=\|\iExp{\zm}{\xm}\|$  between $\xm,\zm \in \M$.

To utilize the historical and current Riemannian gradients, we need to transport the historical gradients into the tangent space of the current point such that these gradients can be linearly combined in one tangent space. For this purpose, we need to define the parallel transport operator $\Paras{\xm}{\zm}: \Tgi{\xm}\M\rightarrow \Tgi{\zm}\M$ which maps $\ym\in\Tgi{\xm}\M$ to $\Paras{\xm}{\zm}(\ym)\in \Tgi{\zm}\M$ while preserving the inner product and norm, i.e., $\la \ym_1,\ym_2\ra=\la \Paras{\xm}{\zm}(\ym_1), \Paras{\xm}{\zm} (\ym_2)\ra$ and $\|\ym\|=\|\Paras{\xm}{\zm}(\ym)\|$ for $\forall \ym_1,\ym_2,\ym\in\Tgi{\xm}\M$.

We impose on the loss components $f_i(\xm)$ the assumption of geodesic gradient-Lipschitz-smoothness. Such a smoothness condition is conventionally assumed in analyzing Riemannian gradient algorithms~\cite{huang2015riemannian,huang2015broyden,kasai2018riemannian, zhang2016riemannian}.

\begin{assum}[Geodesically $L$-gradient-Lipschitz]\label{GL}
Each loss $f_i(\xm)$ is geodesically $L$-gradient Lipschitz such that $\EE_i \|\nabla f_i(\xm) -\Para{\ym}{\xm}{\nabla f_i(\ym)}\|^2 \leq L^2 \|\iExp{\xm}{\ym}\|^2$.
\end{assum}
It can be shown that if each $f_i(\xm)$ is geodesically $L$-gradient-Lipschitz, then for any $\xm,\ym\in\M$,
 		\begin{equation*}
		\begin{split}
		f(\ym)\leq f(\xm)+\la\nabla f(\xm),\iExp{\xm}{\ym} \ra +\frac{L}{2}\|\iExp{\xm}{\ym} \|^2.
		\end{split}
		\end{equation*}

We also need to impose the following boundness assumption on the variance of stochastic gradient.
\begin{assum}[Bounded Stochastic Gradient Variance]\label{boundedgradient} For any $\xm\in\M$ , the gradient variance of each loss $f_i(\xm)$ is bounded as $\EE_i\|\nabla f_i(\xm) -\nabla f(\xm)\|^2 \leq \sigma^2 .$
\end{assum}

We further introduce the following concept of $\tau$-gradient dominated function~\cite{polyak1963gradient,nesterov2006cubic} which will also be investigated in this paper.
\begin{defn}[$\tau$-Gradient Dominated Functions]\label{def:gdf}
$f(\xm)$ is said to be a $\tau$-gradient dominated function if it satisfies $f(\xm)-f(\xmi{*})\leq \tau \|\nabla f(\xm)\|^2$ for any $\xm\in \M$, where $\tau$ is a universal constant and $\xmi{*}=\argmin_{\xm\in\M} f(\xm)$ is the global minimizer of $f(\xm)$ on the manifold $\M$.
\end{defn}

The following defined incremental first order oracle (IFO) complexity is usually adopted as the computational complexity measurement for evaluating stochastic optimization algorithms~\cite{kasai2018riemannian, zhang2016riemannian,kasai2016riemannian,kasai2018riemanniana}.
\begin{defn}[IFO Complexity]\label{def:IFO}
  For $f(\xm)$ in problem~\eqref{generalproblem}, an IFO takes in an index $i \in [n]$ and a point $\xm$, and returns the pair $(f_i(\xm),\nabla f_i(\xm))$.
\end{defn}

\section{Riemannian SPIDER Algorithm}\label{secconvergence}
We first elaborate on the Riemannian SPIDER algorithm, and then analyze its convergence performance for general non-convex problems. For gradient dominated problems, we further develop a variant of R-SPIDER with a linear rate of convergence.

\subsection{Algorithm}\label{proposedalgorithm}
The R-SPIDER method is outlined in Algorithm~\ref{algmanifold}. At its core, R-SPIDER customizes SPIDER to recursively estimate/track the full Riemannian gradient in a computationally economic way. For each cycle of $p$ iterations, R-SPIDER first samples a large data batch $\SSm_1$ by with-replacement sampling and views the gradient estimate $\vmti{k} = \nabla f_{\SSm_1}(\xmi{k})=\frac{1}{|\SSm_1|} \sum_{i\in\SSm_1} f_i(\xmi{k})$ as the snapshot gradient. For the next forthcoming $p-1$ iterations, R-SPIDER only samples a smaller mini-batch $\SSm_2$ and estimates the full Riemannian gradient $\nabla f(\xmi{k})$ as $\vmti{k}=   \fsi{2}(\xmi{k})-\Para{\xmi{k-1}}{\xmi{k}}{\fsi{2}(\xmi{k-1})-\vmti{k-1}}$. Here the parallel transport operator $\Para{\xmi{k-1}}{\xmi{k}}{\cdot}$ is applied to ensure that the Riemannian gradients can be linearly combined in a common tangent space. If $\|\vmti{k}\|>0.5\epsilon$, then R-SPIDER performs normalized gradient descent to update $\xmi{k+1}= \mbox{Exp}_{\xmi{k}}\big(-\etai{k} \frac{\vmti{k}}{\|\vmti{k}\|}\big)$. Otherwise, the algorithm terminates and returns $\xmi{k}$.

The idea of recursive Riemannian gradient estimation has also been exploited by R-SRG~\citep{kasai2018riemannian}. Although sharing a similar spirit in full gradient approximation, R-SPIDER departs notably from R-SRG: at each iteration, R-SPIDER normalizes the gradient $\vmti{k}$ and thus is able to well control the distance $\dis{\xmi{k}}{\xmi{k+1}}$ between $\xmi{k}$ and $\xmi{k+1}$ by properly controlling the stepsize $\eta$, while R-SRG directly updates the variable without gradient normalization. It turns out that this normalization step is key to achieving  faster convergence speed for non-convex problem in R-SPIDER, since it helps reduce the variance of stochastic
gradient estimation by properly controlling the distance $\dis{\xmi{k}}{\xmi{k+1}}$ (see Lemma~\ref{lemma1}). As a consequence, at each iteration, R-SPIDER only needs to sample a necessary number of data points to estimate Riemannian gradient and decrease the objective sufficiently (see Theorems~\ref{objectivefinite} and~\ref{objective}). In this way,   R-SPIDER achieves lower overall computational complexity for solving  problem~\eqref{generalproblem}. 

\subsection{Computational complexity analysis}\label{convegrenceanalysis}
The vanilla SPIDER is known to achieve nearly optimal iteration complexity bounds for stochastic non-convex optimization in Euclidean space~\cite{fang2018spider}. We here show that R-SPIDER generalizes such an appealing property of SPIDER to Riemannian manifolds. We first present the following key lemma which guarantees sufficiently accurate Riemannian gradient estimation for R-SPIDER. We denote $\mathbb{I}_{\{\Es\}}$ as the indicator function: if the event $\Es$ is true, then $\mathbb{I}_{\{\Es\}}=1$; otherwise, $\mathbb{I}_{\{\Es\}}=0$.

\begin{lem}[Bounded Gradient Estimation Error]\label{lemma1}Suppose Assumptions~\ref{GL} and~\ref{boundedgradient} hold. Let $k_0=\lfloor k/p\rfloor $ and $\widetilde{k}_0=k_0 p$. The estimation error between the full Riemannian gradient $\nabla f(\xmi{k})$ and its estimate $\vmti{k}$ in Algorithm~\ref{algmanifold} is bounded as
		\begin{equation*}
\begin{split}
&\EE \left[\|\vmti{k}-\nabla f(\xmi{k})\|^2\ |\ \xmi{\widetilde{k}_0},\cdots,\xmi{\widetilde{k}_0+p-1} \right]\\
\leq\ & \mathbb{I}_{\left\{|\SSm_1|<n\right\}} \frac{\sigma^2}{|\SSm_1|} +  \frac{L^2}{|\SSm_2|}\sum_{i=\widetilde{k}_0}^{\widetilde{k}_0+p-1}\diss{\xmi{i}}{\xmi{i+1}},
\end{split}
	\end{equation*}
where $\dis{\xmi{i}}{\xmi{i+1}}$ is the distance between $\xmi{i}$ and $\xmi{i+1}$.
\end{lem}

{
	\begin{algorithm}[t]
		\caption{{R-SPIDER ($\xmi{0}$, $\epsilon$, $\eta$, $p$, $|\SSm_1|$, $|\SSm_2|$)}}
		\label{algmanifold}
		\begin{algorithmic}[1]
			\STATE {\bfseries Input:}\,initialization $\xmi{0}$,\,accuracy $\epsilon$,\,learning rate $\eta$, iteration interval $p$, mini-batch sizes $|\SSm_1|$ and $|\SSm_2|$.
			
			\FOR{ $k=0$ to $K-1$ }
			\IF{mod($k,p$)$=0$}
			\STATE Draw mini-batch $\SSm_1$ and compute $\vmti{k}= \nabla f_{\SSm_1}(\xmi{k})$;
			\ELSE
			\STATE Draw mini-batch $\SSm_{2}$  and compute $ \nabla f_{\SSm_{2}}(\xmi{k})$;\\
			\STATE $\vmti{k}=   \fsi{2}(\xmi{k})-\Para{\xmi{k-1}}{\xmi{k}}{\fsi{2}(\xmi{k-1})-\vmti{k-1}}$;
			\ENDIF
\STATE $\xmi{k+1}= \Exp{\xmi{k}}{-\etai{k} \frac{\vmti{k}}{\|\vmti{k}\|}}$;\\
			\ENDFOR
			
			\STATE {\bfseries Output:} $\xms$ which is chosen uniformly at  random from $\{\xmi{k}\}_{k=0}^{K-1}$.
		\end{algorithmic}
	\end{algorithm}
}

\begin{proof}
The key is to carefully handle the exponential mapping and parallel transport operators introduced for vector computation. See details in Appendix~\ref{proofofauxiliarylemma}.
\end{proof}
Lemma~\ref{lemma1} tells that by properly selecting the mini-batch sizes $|\SSm_1|$ and $|\SSm_2|$, the accuracy of gradient estimate $\vmti{k}$ can be controlled. Benefiting from the normalization step, we have $\dis{\xmi{k}}{\xmi{k+1}}=\|\iExp{\xmi{k}}{\xmi{k+1}}\|=\eta$. As a result, the gradient estimation error can be bounded as $\EE \big[\|\vmti{k}-\nabla f(\xmi{k})\|^2\ |\ \xmi{\widetilde{k}_0},\cdots,\xmi{\widetilde{k}_0+p-1} \big]
\leq \mathbb{I}_{\{|\SSm_1|<n\}} \frac{\sigma^2}{|\SSm_1|} +  \frac{pL^2\eta^2}{|\SSm_2|}$. Based on this result,  we are able to analyze the
rate-of-convergence of R-SPIDER.

\textbf{Finite-sum setting.} We first consider problem~\eqref{generalproblem} under finite-sum setting. By properly selecting parameters, we prove that at each iteration, the sequence $\{\xmi{k}\}$ produced by Algorithm~\ref{algmanifold} can lead to sufficient decrease of the objective loss $f(\xm)$  when $\|\vmi{k}\|$ is large. Based on this results, we further derive the iteration number of Algorithm~\ref{algmanifold} for computing an $\epsilon$-accuracy solution. The result is formally summarized in Theorem~\ref{objectivefinite}.

\begin{thm}\label{objectivefinite} Suppose Assumptions~\ref{GL} and~\ref{boundedgradient} hold.
 Let $s\!=\!\min\!\big(n,\frac{ 16\sigma^2}{\epsilon^2}\big)$, $p\!=\!n_0 s^{\frac{1}{2}} $, $\etai{k} = \min\left(\frac{\epsilon}{2Ln_0}, \frac{ \|\vmti{k}\|}{4L n_0}\right)$, $|\SSm_1|\!=\!s$, $|\SSm_2|\!=$ $\frac{4s^{\frac{1}{2}}}{n_0}$ and $n_0\in[1,4s^{\frac{1}{2}}]$. Then for finite-sum problem~\eqref{generalproblem}, the  sequence $\{\xmi{k}\}$ produced by Algorithm~\ref{algmanifold} satisfies
	\begin{equation*}
	\begin{split}
	\EE \left[f(\xmi{k+1}) - f(\xmi{k})\right] \leq -\frac{\epsilon}{64Ln_0} \left( 12\EE [\|\vmti{k}\|]  -7\epsilon\right).
	\end{split}
	\end{equation*}
Moreover, to achieve $\EE[\|\nabla f(\xms)\|]\leq \epsilon$,  Algorithm~\ref{algmanifold} will terminate at most $\left(\frac{14Ln_0\Delta}{\epsilon^2}\right)$ iterations  in expectation, where $\Delta=f(\xmi{0})- f(\xmi{*})$ with $\xmi{*}=\argmin_{\xm\in\M} f(\xm)$.
\end{thm}
\begin{proof}
The result comes readily from geodesically $L$-gradient-Lipschitz of $f$ and the variance bound in Lemma~\ref{lemma1}. See Appendix~\ref{proofofobjectivefinite} for a complete proof.
\end{proof}

Theorem~\ref{objectivefinite} shows that Algorithm~\ref{algmanifold} only needs to run at most $\left(\frac{14Ln_0\Delta}{\epsilon^2}\right)$ iteration to compute an $\epsilon$-accuracy solution $\xms$, \ie~$\EE[\|\nabla f(\xms)\|]\leq \epsilon$. This means the convergence rate of R-SPIDER is at the order of $\mathcal{O}\left(\frac{Ln_0\Delta}{\epsilon^2}\right)$. Besides, by one iteration loop of Algorithm~\ref{algmanifold}, the objective value $f(\xmi{k})$ monotonously decreases in expectation when $\EE[\|\vmi{k}\|]$ is large, \eg~$\EE[\|\vmi{k}\|]\geq \frac{7\epsilon}{12}$. By comparison, Kasai\et\cite{kasai2018riemannian} only proved the sublinear convergence rate of the gradient norm $\EE[\|\nabla f(\xm)\|^2]$ in R-SRG and did not reveal any  convergence behavior of the objective $f(\xm)$. Moreover, Theorem~\ref{objectivefinite} yields as a byproduct the benefits of mini-batching to R-SPIDER. Indeed, by controlling the parameter $n_0$ in R-SPIDER, the mini-batch size $|\SSm_2|$ at each iteration can range from 1 to $\min\big(4\sqrt{n},\frac{ 16\sigma}{\epsilon}\big)$. Also, it can be seen from Theorem~\ref{objectivefinite} that larger mini-batch size allows more aggressive step size $\etai{k}$ and thus leads to less necessary iterations to achieve an $\epsilon$-accuracy solution. More specifically, the convergence rate bound $\mathcal{O}\left(\frac{Ln_0\Delta}{\epsilon^2}\right)$ indicates that at least in theory, increasing the mini-batch sizes in R-SPIDER provides linear speedups in parallel computing environment. In contrast, these important benefits of mini-batching are not explicitly analyzed in the existing Riemannian stochastic gradient algorithms~\cite{kasai2018riemannian,zhang2016riemannian}.

Based on Theorem~\ref{objectivefinite}, we can derive the IFO complexity of R-SPIDER for non-convex problems in Corollary~\ref{complexityfinite}.
\begin{cor}\label{complexityfinite}
Using the same assumptions and parameters in Theorem~\ref{objectivefinite}, the IFO complexity of Algorithm~\ref{algmanifold} is $\mathcal{O}\big(\min\big(n+\frac{L  \Delta\sqrt{n}}{\epsilon^2},\frac{L\Delta\sigma}{\epsilon^3}\big)\big)$ for achieving $\EE[\|\nabla f(\xms)\|]\leq \epsilon$.
\end{cor}
\begin{proof}
The result is obtained directly from a cumulation of IFOs at each step of iteration. See Appendix~\ref{proofofcomplexityfinite}.
\end{proof}

From Corollary~\ref{complexityfinite}, the IFO complexity of R-SPIDER for non-convex finite-sum problems is at the order of $\Oc{\frac{1}{\epsilon^2}\min\left(\sqrt{n},\frac{1}{\epsilon}\right)}$. This result matches the state-of-the-art complexity bounds for general non-convex optimization problems in Euclidean space~\cite{fang2018spider,zhou2018stochastic}. Indeed, under Assumption~\ref{GL}, Fang\et\cite{fang2018spider} proved that the lower IFO complexity bound for finite-sum problem~\eqref{generalproblem} in Euclidean space is $\mathcal{O}\big(n+\frac{L\Delta\sqrt{n}}{\epsilon^2}\big)$ when the number $n$ of the component function obeys $n\leq \mathcal{O}\big(\frac{L^2\Delta^2}{\epsilon^4}\big)$. In the sense that Euclidean space is a special case of Riemannian manifold, our IFO complexity $\mathcal{O}\big(n+\frac{L\Delta\sqrt{n}}{\epsilon^2}\big)$  for finite-sum problem~\eqref{generalproblem} under Assumption~\ref{GL} is nearly optimal. If Assumption~\ref{boundedgradient} holds in addition, we can establish tighter IFO complexity  $\Oc{\frac{1}{\epsilon^2}\min\left(\sqrt{n},\frac{1}{\epsilon}\right)}$. This is because when the sample number $n$ satisfies
 $n\geq \frac{16\sigma^2}{\epsilon^2}$, by sampling $|\SSm_1|= \frac{16\sigma^2}{\epsilon^2} $ and $|\SSm_2|=\frac{16\sigma}{n_0\epsilon}$, the gradient estimation error already satisfies $\EE  [\|\vmti{k}-\nabla f(\xmi{k})\|^2]\leq \frac{\epsilon^2}{8}$. Accordingly, if $\frac{1}{K}\sum_{k=0}^{K-1}\EE\|\vmti{k}\|\leq 0.5\epsilon$ which is actually achieved after $K$ iterations, then $\EE [\|f(\xms)\|]=  \frac{1}{K}\sum_{k=0}^{K-1}\EE\|\nabla f(\xmi{k}) \| \leq   \frac{1}{K} \sum_{k=0}^{K-1} $ $\left[ \EE\| \nabla f(\xmi{k})-\vmti{k}\|+\EE\|\vmti{k}\| \right] \leq \epsilon$.
 So here it is only necessary to sample $|\SSm_1|=\frac{16\sigma^2}{\epsilon^2}$ data points  instead of the entire set of $n$ samples. 

Kasai\et\cite{kasai2018riemannian} proved that the IFO complexity of R-SRG is at the order of $\mathcal{O}\big(n+\frac{L^2}{\epsilon^4}\big)$ to obtain an $\epsilon$-accuracy solution. By comparison, we prove that R-SPIDER enjoys the complexity of $\Oc{\frac{1}{\epsilon^2}\min\left(\sqrt{n},\frac{1}{\epsilon}\right)}$, which is at least lower than R-SRG by a factor of $\frac{1}{\epsilon}$. This is because the normalization step in R-SPIDER allows us to well control the gradient estimation error and thus avoids sampling too many redundant samples at each iteration, resulting in sharper IFO complexity. Zhang\et\cite{zhang2016riemannian} showed that  R-SVRG has  the IFO complexity  $\mathcal{O}\big(n+\frac{\zeta^{1/2}n^{2/3}}{\epsilon^2}\big)$, where $\zeta\geq 1$ denotes the curvature parameter. Therefore, R-SPIDER improves over R-SVRG by a factor at least $n^{1/6}$ in IFO complexity. Note, here the curvature parameter $\zeta$ does not appear in our bounds, since we have avoided using the trigonometry inequality which characterizes the trigonometric geometric in Riemannian manifold~\cite{bonnabel2013stochastic,zhang2016riemannian,zhang2016first}.

The exponential mapping and parallel transport operators used in R-SPIDER are respectively classical instances of the more general concepts of retraction and vector transport~\cite{adler2002newton,absil2009optimization}. We note that under identical assumptions in~\cite{kasai2018riemannian}, the convergence rate and IFO complexity bounds for R-SPIDER generalize well to the setting where exponential mapping and parallel transport are replaced by retraction and vector transport operators. Some specific ways of constructing retraction and vector transport are available in~\cite{adler2002newton,absil2012projection,wen2013feasible}.

\textbf{Online setting.} Next we consider the online setting of problem~\eqref{generalproblem}. Similar to finite-sum setting, we prove in Theorem~\ref{objective} that the objective $f(\xm)$ can be sufficiently decreased when the gradient norm is not too small.
  \begin{thm}\label{objective} Suppose Assumptions~\ref{GL} and~\ref{boundedgradient} hold.
   Let $p\!=\!\frac{\sigma n_0}{\epsilon}$, $\etai{k} = \min\left(\frac{\epsilon}{2Ln_0}, \frac{ \|\vmti{k}\|}{4L n_0}\right)$, $|\SSm_1|\!=\!\frac{ 64\sigma^2}{\epsilon^2}$, $|\SSm_2|\!=\!\frac{4\sigma}{\epsilon n_0}$ and $n_0\in$ $ [1,4\sigma/\epsilon]$. Then for  problem~\eqref{generalproblem} under online setting, the  sequence $\{\xmi{k}\}$ produced by Algorithm~\ref{algmanifold} satisfies
	\begin{equation*}
	\begin{split}
	\EE \left[f(\xmi{k+1}) - f(\xmi{k})\right] \leq -\frac{\epsilon}{64Ln_0} \left( 12\EE [\|\vmti{k}\|]  -7\epsilon\right).
	\end{split}
	\end{equation*}
Moreover, to achieve $\EE [\|\nabla f(\xms)\|]\leq \epsilon$,  Algorithm~\ref{algmanifold} will terminate at most $\left(\frac{14Ln_0\Delta}{\epsilon^2}\right)$ iterations  in expectation, where $\Delta=f(\xmi{0})-f(\xmi{*})$ with $\xmi{*}=\argmin_{\xm\in\M} f(\xm)$.
\end{thm}
\begin{proof}
The proof mimics that of Theorem~\ref{objectivefinite} with proper adaptation to online setting. See Appendix~\ref{proofofobjection}.
\end{proof}
As a direct consequence of this result, the following corollary establishes the IFO complexity of R-SPIDER for the online optimization.

\begin{cor}\label{complexityonline}
Using the same assumptions and parameters in Theorem~\ref{objective}, the IFO complexity of Algorithm~\ref{algmanifold} is  $\mathcal{O}\left(\frac{L  \sigma \Delta}{\epsilon^3}\right)$ to achieve $\EE[\|\nabla f(\xm)\|]\leq \epsilon$.
\end{cor}
\begin{proof}
See Appendix~\ref{proofofcomplexityonline} for a proof of this result.
\end{proof}
Bonnabel\et\cite{bonnabel2013stochastic} have also analyzed R-SGD under online setting, but only with asymptotic convergence guarantee obtained. By comparison, we for the first time establish non-asymptotic complexity bounds for Riemannian online non-convex optimization.

{
	\begin{algorithm}[hb!t]
		\caption{{Riemannian Gradient Dominated SPIDER (R-GD-SPIDER)} }
		\label{algmanifold2}
		\begin{algorithmic}[1]
			\STATE {\bfseries Input:} initial point $\xmti{0}$, initial accuracy $\epsilon_0$,\,learning rate $\eta^0$, mini-batch sizes $|\SSm^0_1|$ and $|\SSm^0_2|$, iteration interval $p^0$, final accuracy $\epsilon$
			\FOR{ $t=1$ to $T$ }			
			\STATE
$\xmti{t}=\text{R-SPIDER}(\xmti{t-1},\epsilon_{t-1},\eta^t, \pii{t}, |\Si{1}{t}|, |\Si{2}{t}|)$.
\STATE Set $\epsilon_t\!=\!0.5\epsilon_{t-1}$, and $\eta^t, \pii{t}, |\Si{1}{t}|, |\Si{2}{t}|$ properly.
			\ENDFOR
			
			\STATE {\bfseries Output:} $\xmti{t}$
		\end{algorithmic}
	\end{algorithm}
	\vspace{-0.8em}
}

\subsection{On gradient dominated functions}\label{gradientdominatedfunctions}

We now turn to a special case of problem~\eqref{generalproblem} with gradient dominated loss function as defined in Definition~\ref{def:gdf}. For instance, the strongly geodesically convex (SGC) functions\footnote{A strongly geodesically convex function satisfies $f(\ym)\geq f(\xm)+\la \nabla f(\xm), \iExp{\xm}{\ym} \ra+\frac{\mu}{2}\|\iExp{\xm}{\ym}\|^2, \forall \xm,\ym\in\M$, for som $\mu>0$, which immediately implies $ f(\xm) - f(\xmi{*}) \leq  \frac{1}{2\mu}\| \nabla f(\xm)\|^2$ by Cauchy-Schwarz inequality.} are gradient dominated. Some non-strongly convex problems, \eg ill-conditioned linear prediction and logistic regression~\cite{karimi2016linear}, and   Riemannian  non-convex problems, \eg PCA~\cite{zhang2016riemannian}, also belong to gradient dominated functions. Please refer to~\cite{nesterov2006cubic,karimi2016linear} for more instances of gradient dominated functions. To better fit gradient dominated functions, we develop the Riemannian gradient dominated SPIDER (R-GD-SPIDER) as a multi-stage variant of R-SPIDER. A high-level description of R-GD-SPIDER is outlined in Algorithm~\ref{algmanifold2}. The basic idea is to use more aggressive learning rates in early stage of processing and gradually shrink the learning rate in later stage. With the help of such a simulated annealing process, R-GD-SPIDER exhibits linear convergence behavior for finite-sum problems, as formally stated in Theorem~\ref{totalcomplexityofconvex}. For the $t$-th iteration in Algorithm~\ref{algmanifold2}, R-SPIDER uses $|\SSm_{1}^t|$ and $|\SSm_{2,k}^t|$ samples to compute $\vmi{k}$ for $k=\lfloor k/p\rfloor \cdot p$ and  $k\neq \lfloor k/p\rfloor \cdot p$, respectively.

\begin{thm}\label{totalcomplexityofconvex}
Suppose that function $f(\xm)$ is $\tau$-gradient dominated, and Assumptions~\ref{GL} and~\ref{boundedgradient} hold. For finite-sum setting, at the $t$-th iteration, set $\epsilon_0=  \frac{\sqrt{\Delta}}{2\sqrt{\tau}}$,
$\epsilon_{t}=\frac{\epsilon_0}{2^t}$,
$s_t=\!\min\!\big(n,\frac{ 32\sigma^2}{\epsilon_{t-1}^2}\big)$, $p^t\!=\!n_0^t \sqrt{s_t}$, $\etai{k}^t =  \frac{\|\vmi{k}^t\|}{2Ln_0}$, 
$|\SSm_1^t|\!=\!s_t$, $|\SSm_{2,k}^t|\!=\!\min (\frac{8p^t\|\vmi{k-1}^t\|^2}{(n_0^t)^2\epsilon_{t-1}^2},n )$, where $n_0^t\in[1,\frac{8\sqrt{s_t}  \|\vmi{k-1}^t\|^2}{\epsilon_{t-1}^2}]$.\\
\textbf{(1)} The  sequence $\{\xmti{t}\}$ produced by Algorithm~\ref{algmanifold2} satisfies
\begin{equation*}
\EE \left[ f(\xmti{t})-f(\xmi{*})\right] \leq \frac{ \Delta}{4^t} \ \ \  \text{and}\ \ \  \EE [\|\nabla f(\xmti{t})\|] \leq  \frac{1}{2^t}\sqrt{\frac{\Delta}{\tau}},
\end{equation*}
where  $\Delta=f(\xmti{0})-f(\xmi{*})$ with $\xmi{*}=\argmin_{\xm\in\M} f(\xm)$. \\
\textbf{(2)} To achieve $\EE [\|\nabla f(\xmti{T})\|]\! \leq\! \epsilon$, in expectation the IFO complexity is $\mathcal{O}\left( \min\left(\left(n+ \tau L \sqrt{n}\right)\log\left(\frac{1}{\epsilon} \right),\frac{\tau L \sigma}{\epsilon}\right) \right).$	
\end{thm}
\begin{proof}
The part (1) follows immediately from the update rule of $\epsilon_t$. The part (2) can be proved by establishing the IFO bound $\min\left( n+\tau L \sqrt{n}  ,\frac{\tau L\sigma}{\epsilon_{t+1}}\right)$ for each stage $t$ and then putting them together. See Appendix~\ref{proofoffinitegradient}.
\end{proof}

\begin{figure*}[ht]
\begin{center}
\setlength{\tabcolsep}{0.8pt} 
\begin{tabular}{cccc}
\includegraphics[width=0.245\linewidth]{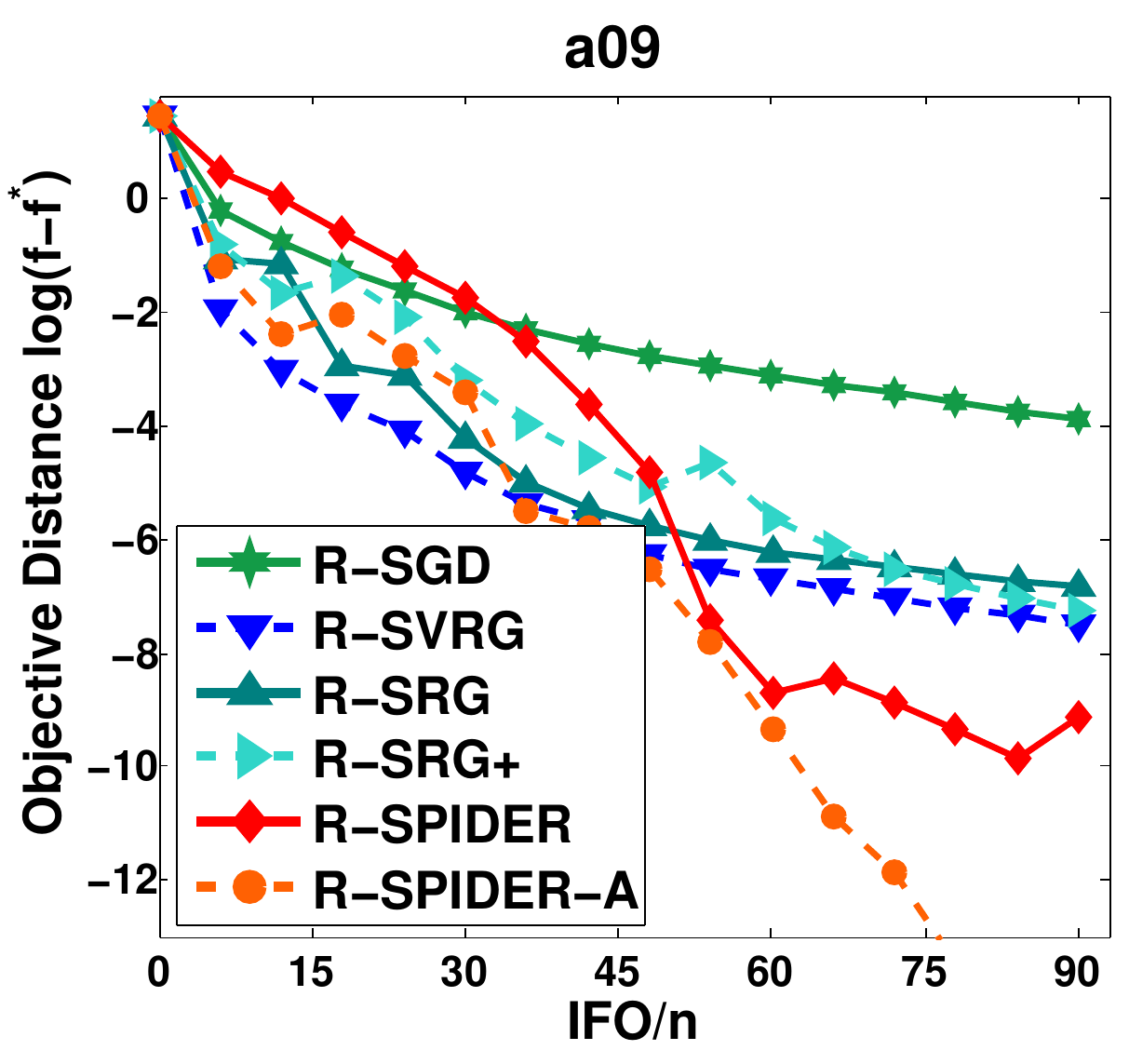}&
\includegraphics[width=0.245\linewidth]{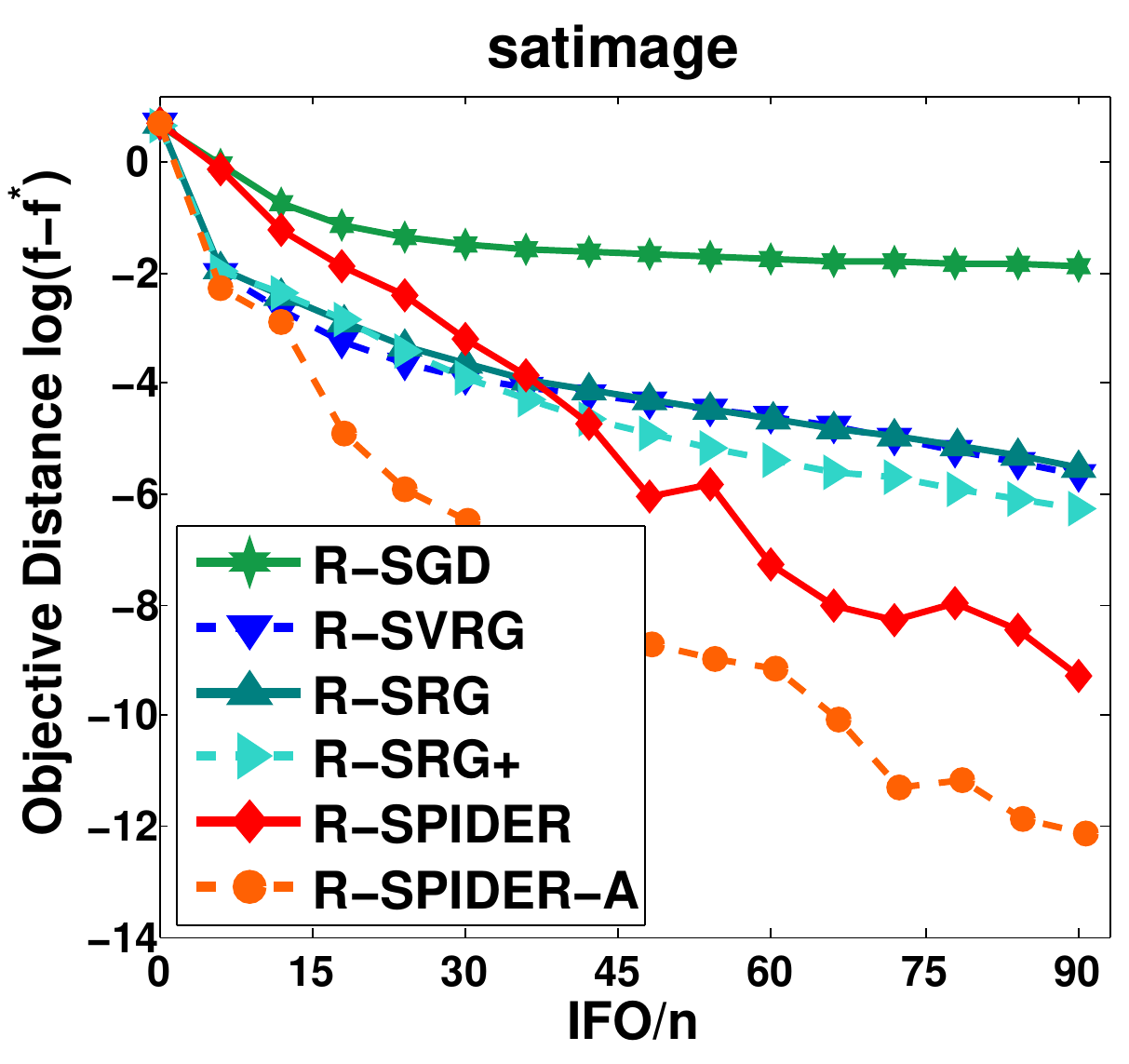}&
\includegraphics[width=0.245\linewidth]{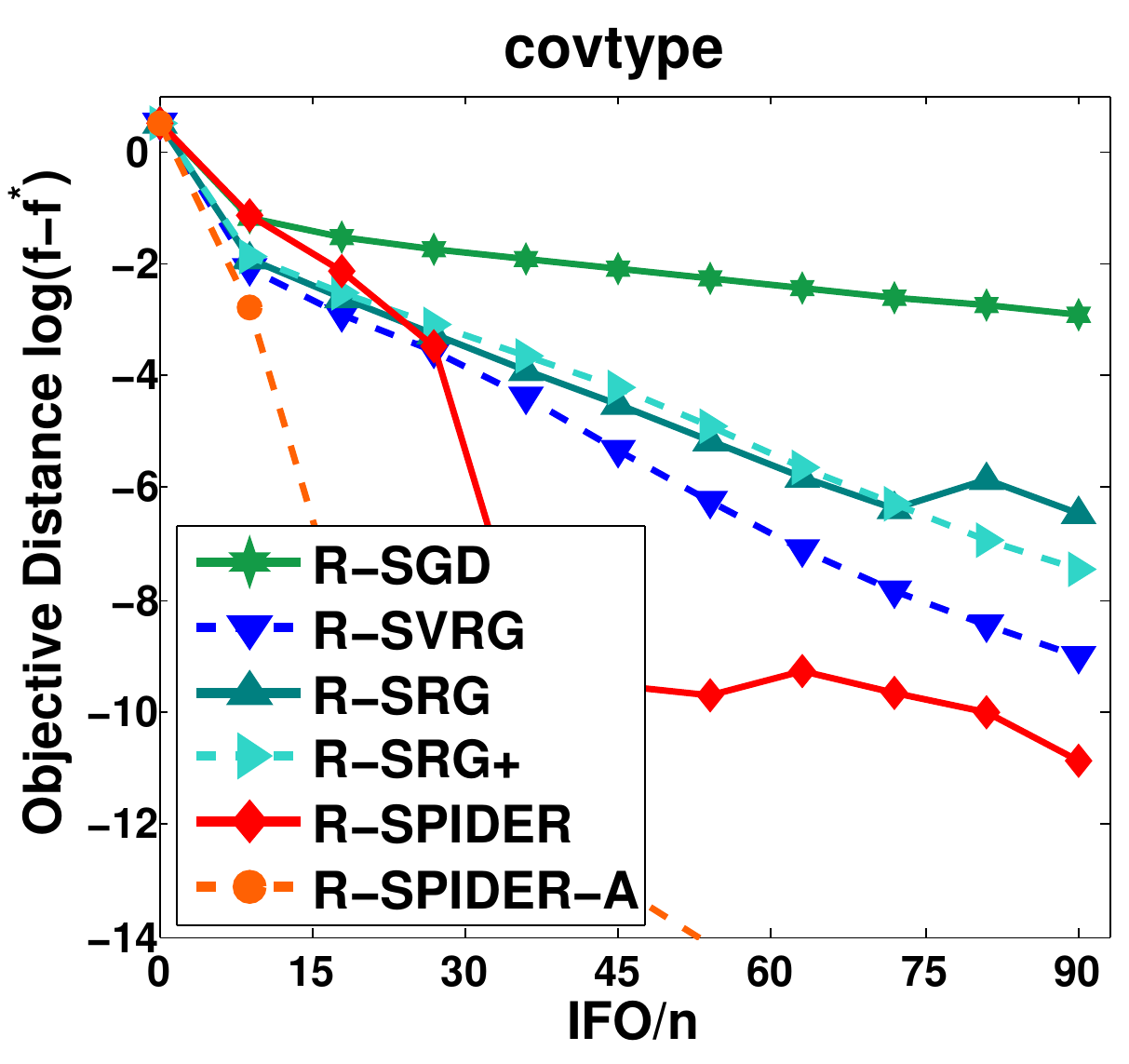}&
\includegraphics[width=0.245\linewidth]{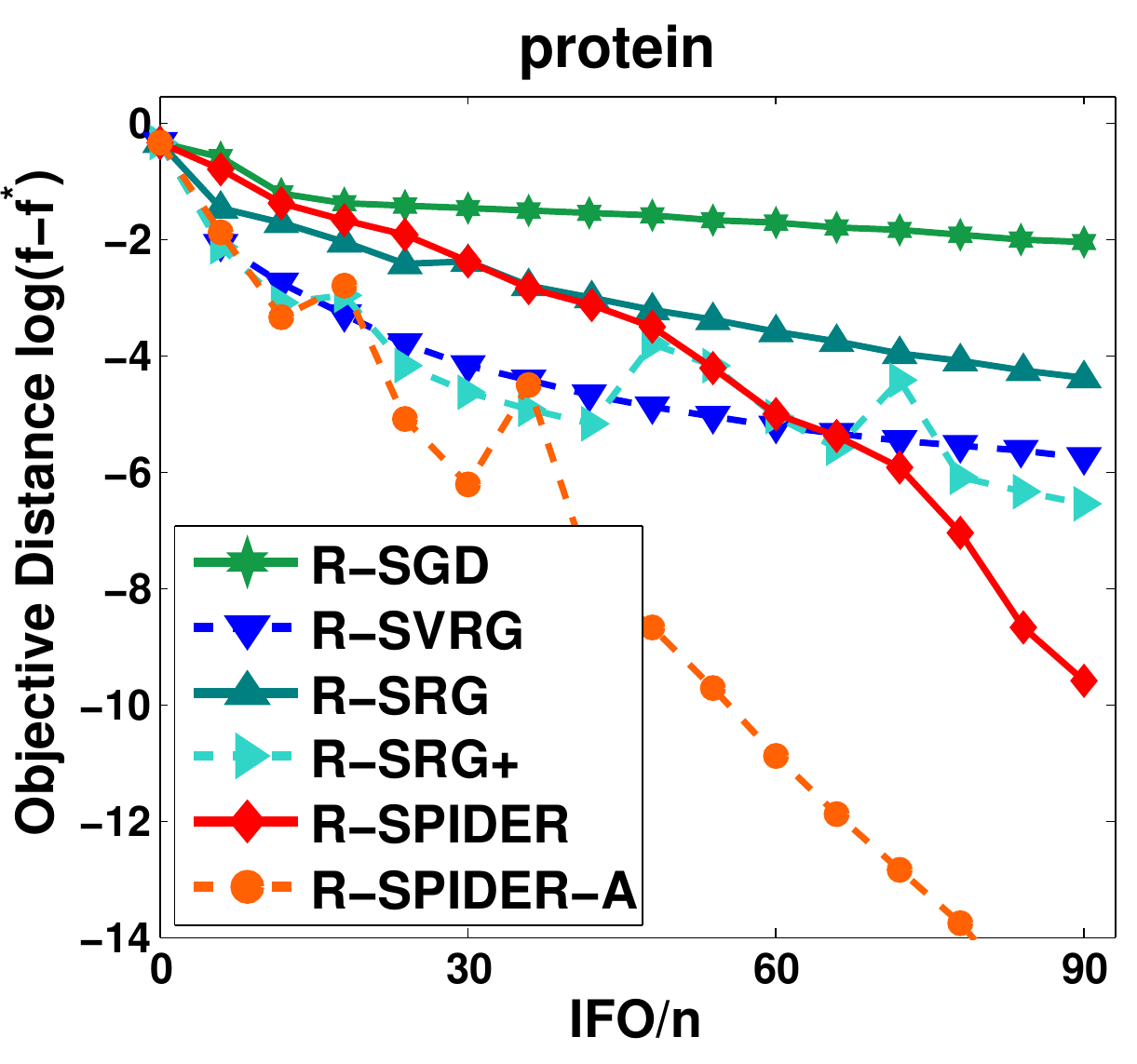}
\end{tabular}
\begin{tabular}{ccc}
\hspace{-0.6em}
\includegraphics[width=0.49\linewidth]{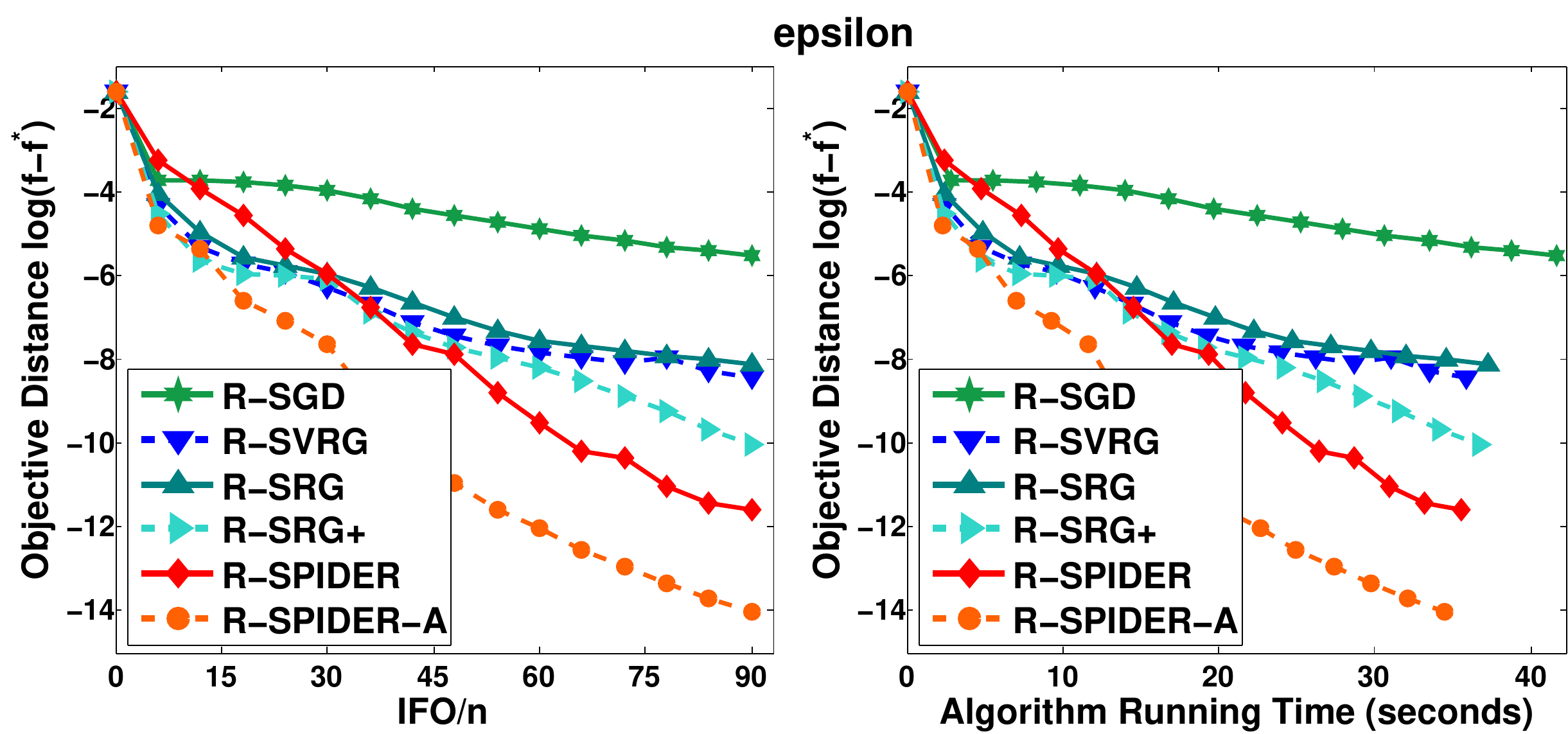}&&
\includegraphics[width=0.49\linewidth]{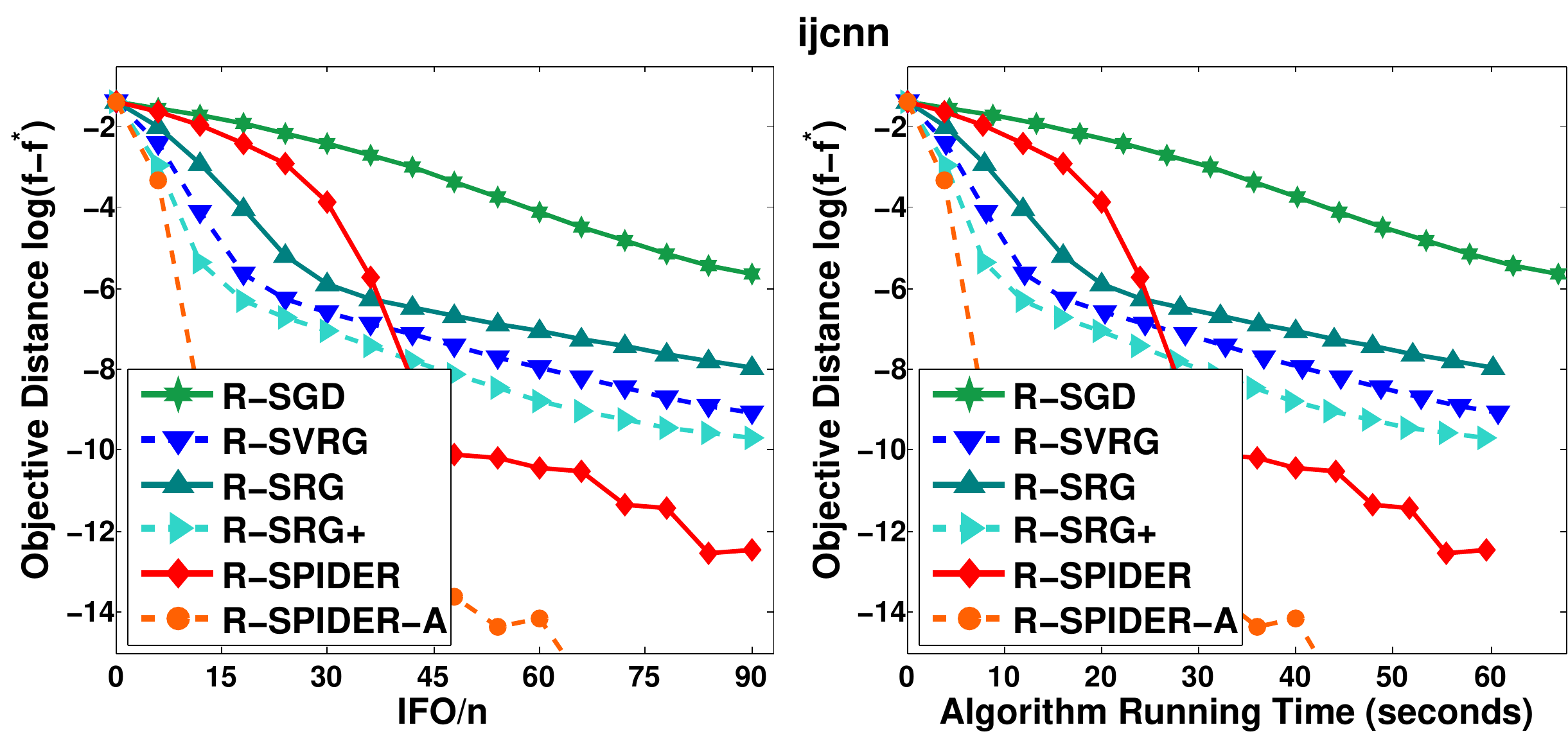}
\end{tabular}
\end{center}
\vspace{-1.2em}
\caption{Comparison among Riemannian stochastic gradient algorithms on  $k$-PCA problem. }
\label{comparisonasfdasfdwith}
\vspace{-1.0em}
\end{figure*}

The main message conveyed by Theorem~\ref{totalcomplexityofconvex} is that R-GD-SPIDER enjoys a global linear rate of convergence and its IFO complexity is at the order of  $\mathcal{O}\left( \min\left(\left(n+ \tau L \sqrt{n}\right)\log\left(\frac{1}{\epsilon} \right),\frac{\tau L \sigma}{\epsilon}\right) \right).$ For R-SVRG with $\tau$-gradient dominated functions, Zhang\et\cite{zhang2016riemannian} also established a linear convergence rate and an IFO complexity bound $\mathcal{O}\big((n+\tau L\zeta^{\frac{1}{2}}n^{\frac{2}{3}})\log\left(\frac{1}{\epsilon}\right)\big)$. As a comparison, our R-GD-SPIDER makes an improvement over R-SVRG in IFO complexity by a factor of $n^{\frac{1}{6}}$. For R-SRG~\cite{kasai2018riemannian}, the corresponding IFO complexity is $\Oc{(n+\tau^2 L^2) \log\left(\frac{1}{\epsilon}\right)}$. 
Therefore, in terms of IFO complexity, R-GD-SPIDER is superior to R-SRG  when the optimization accuracy $\epsilon$ is moderately small at a huge data size $n$.

Turning to the online setting, R-GD-SPIDER also converges linearly, as formally stated in Theorem~\ref{totalcomplexityofconvex2}.
\begin{thm}\label{totalcomplexityofconvex2}
Suppose that $f(\xm)$ is $\tau$-gradient dominated, and Assumptions~\ref{GL} and~\ref{boundedgradient} hold.  For online setting, at the $t$-th iteration, let $\epsilon_0= \frac{\sqrt{\Delta}}{2\sqrt{\tau}}$, $\epsilon_{t}=\frac{\epsilon_0}{2^t}$, $p_t\!=\!\frac{\sigma n_0^t}{\epsilon_{t-1}}$, $\etai{k}^t = \frac{\|\vmi{k}^{t}\|}{2Ln_0^t}$,
$|\SSm_1^t|\!=\!\frac{ 32\sigma^2}{\epsilon_{t-1}^2}$, $|\SSm_{2,k}^t|\!=\!\frac{8\sigma \|\vmi{k-1}^{t}\|^2}{\epsilon_{t-1}^3 n_0^t}$,
where  $n_0^t\in[1,\frac{8\sigma\|\vmi{k-1}^t\|^2}{ \epsilon_{t-1}^3}]$.\\
\textbf{(1)} The  sequence $\{\xmti{t}\}$ produced by Algorithm~\ref{algmanifold2} satisfies
\begin{equation*}
\EE \left[ f(\xmti{t})-f(\xmi{*})\right] \leq \frac{ \Delta}{4^t} \ \ \  \text{and}\ \ \  \EE [\|\nabla f(\xmti{t})\|] \leq  \frac{1}{2^t}\sqrt{\frac{\Delta}{\tau}},
\end{equation*}
where  $\Delta=f(\xmti{0})-f(\xmi{*})$ with $\xmi{*}=\argmin_{\xm\in\M} f(\xm)$. \\
\textbf{(2)} To achieve $\EE [\|\nabla f(\xmti{T})\|]\leq \epsilon$, in expectation the IFO complexity is $\mathcal{O}\left( \frac{\tau L \sigma}{\epsilon}\right).$	
\end{thm}

\begin{proof}
See Appendix~\ref{proofofgradientonline} for a proof of this result.
\end{proof}

Such a non-asymptotic convergence result is new to online Riemannian gradient dominated optimization.

\section{Experiments}\label{experiments}
In this section, we compare the proposed R-SPIDER with several state-of-the-art Riemannian stochastic gradient algorithms, including R-SGD~\cite{bonnabel2013stochastic}, R-SVRG~\cite{zhang2016riemannian,kasai2016riemannian}, R-SRG~\cite{kasai2018riemannian} and R-SRG+~\cite{kasai2018riemannian}. We evaluate all the considered algorithms on two learning tasks: the $k$-PCA problem and the low-rank matrix completion problem.
We run simulations on ten datasets, including six datasets from LibSVM, three face datasets (\textsf{YaleB}, \textsf{AR} and \textsf{PIE}) and one recommendation dataset (\textsf{MovieLens-1M}). The details of these datasets are described in Appendix~\ref{append:more_experiment}. For all the considered algorithms,
we tune their hyper-parameters optimally.

\textbf{A practical implementation of R-SPIDER.} To achieve the IFO complexity in Corollary~\ref{complexityfinite}, it is suggested to set the learning rate as $\eta=\frac{\epsilon}{4Ln_0}$ where $\epsilon$ is the desired optimization accuracy. However, since in the initial epochs the computed point is far from the optimum to problem~\eqref{generalproblem}, using a tiny learning rate could usually be conservative. In contrast, by using a more aggressive learning rate at the initial optimization stage, we can expect stable but faster convergence behavior. Here for R-SPIDER we design a decaying learning rate with formulation $\eta_k=\alpha^{\lfloor\frac{k}{p}\rfloor} \cdot \beta$ and call it ``R-SPIDER-A'', where $\alpha$ and $\beta$ are two constants. In our experiments, $\alpha$ is selected from $\{0.8, 0.85,0.9,0.95,0.99\}$ and $\beta$ from $\{5\times 10^{-2}, 10^{-2}, 5\times 10^{-3}, 10^{-3}\}$.

\begin{figure*}[ht]
\begin{center}
\setlength{\tabcolsep}{0.8pt} 
\begin{tabular}{cccc}
\includegraphics[width=0.245\linewidth]{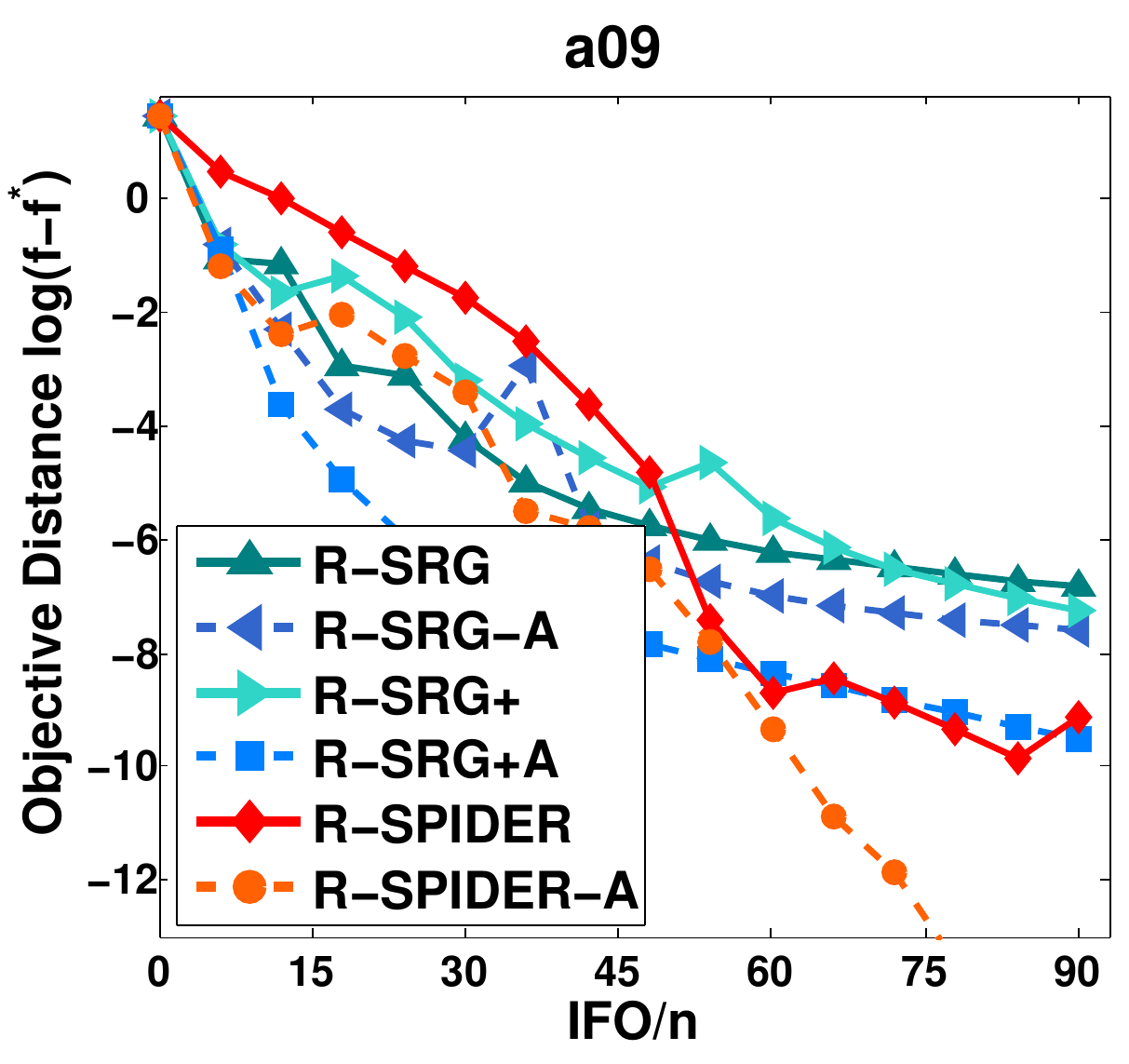}&
\includegraphics[width=0.245\linewidth]{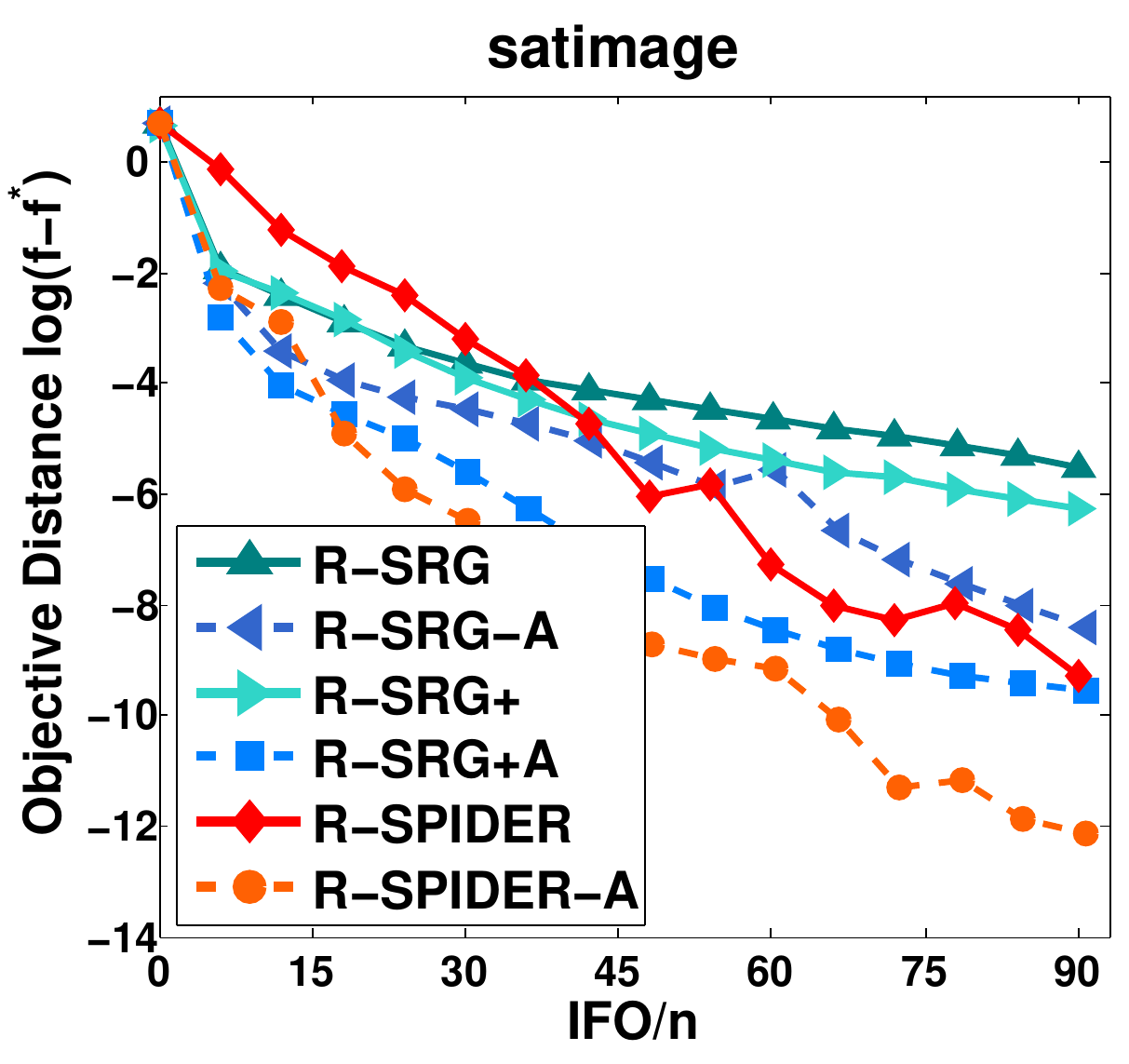} &
\includegraphics[width=0.245\linewidth]{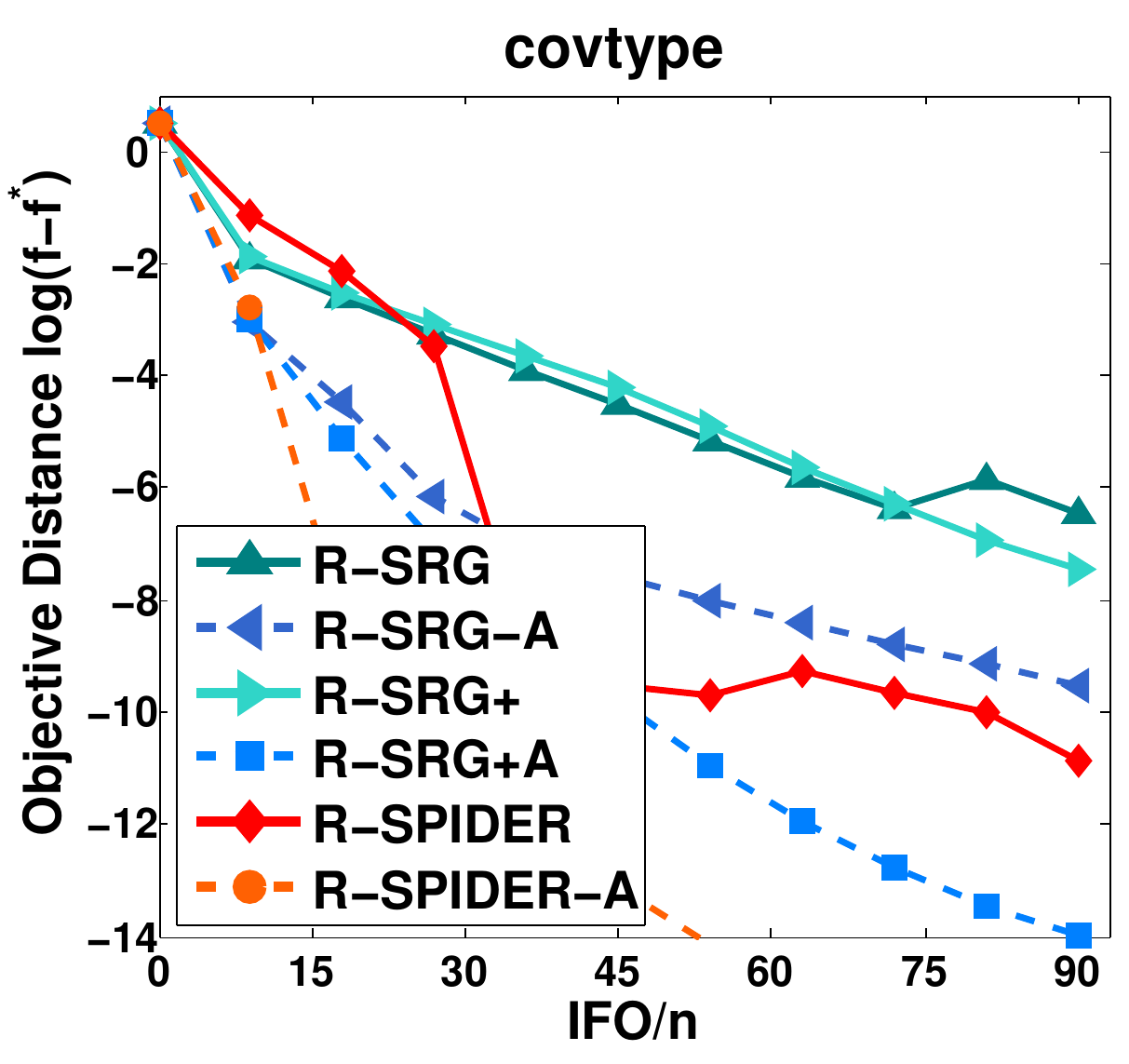}&
\includegraphics[width=0.245\linewidth]{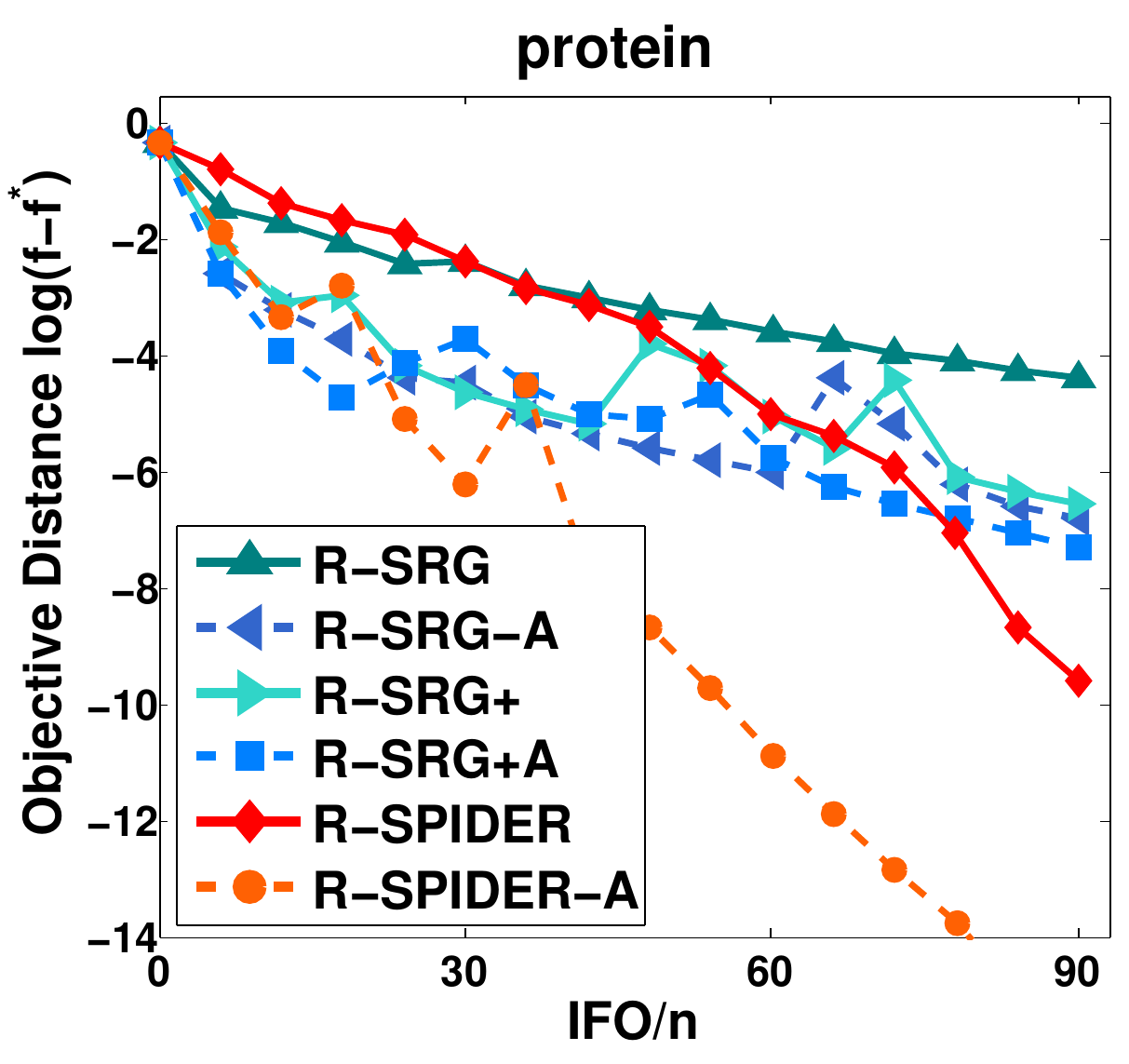}
\end{tabular}
\end{center}
\vspace{-1.4em}
\caption{Comparison between R-SPIDER and R-SRG with adaptive learning rates on $k$-PCA problem. }
\vspace{-0.8em}
\label{compariso46543fdwith}
\end{figure*}

\begin{figure*}[ht]
\begin{center}
\setlength{\tabcolsep}{0.8pt} 
\begin{tabular}{cccc}
\includegraphics[width=0.245\linewidth]{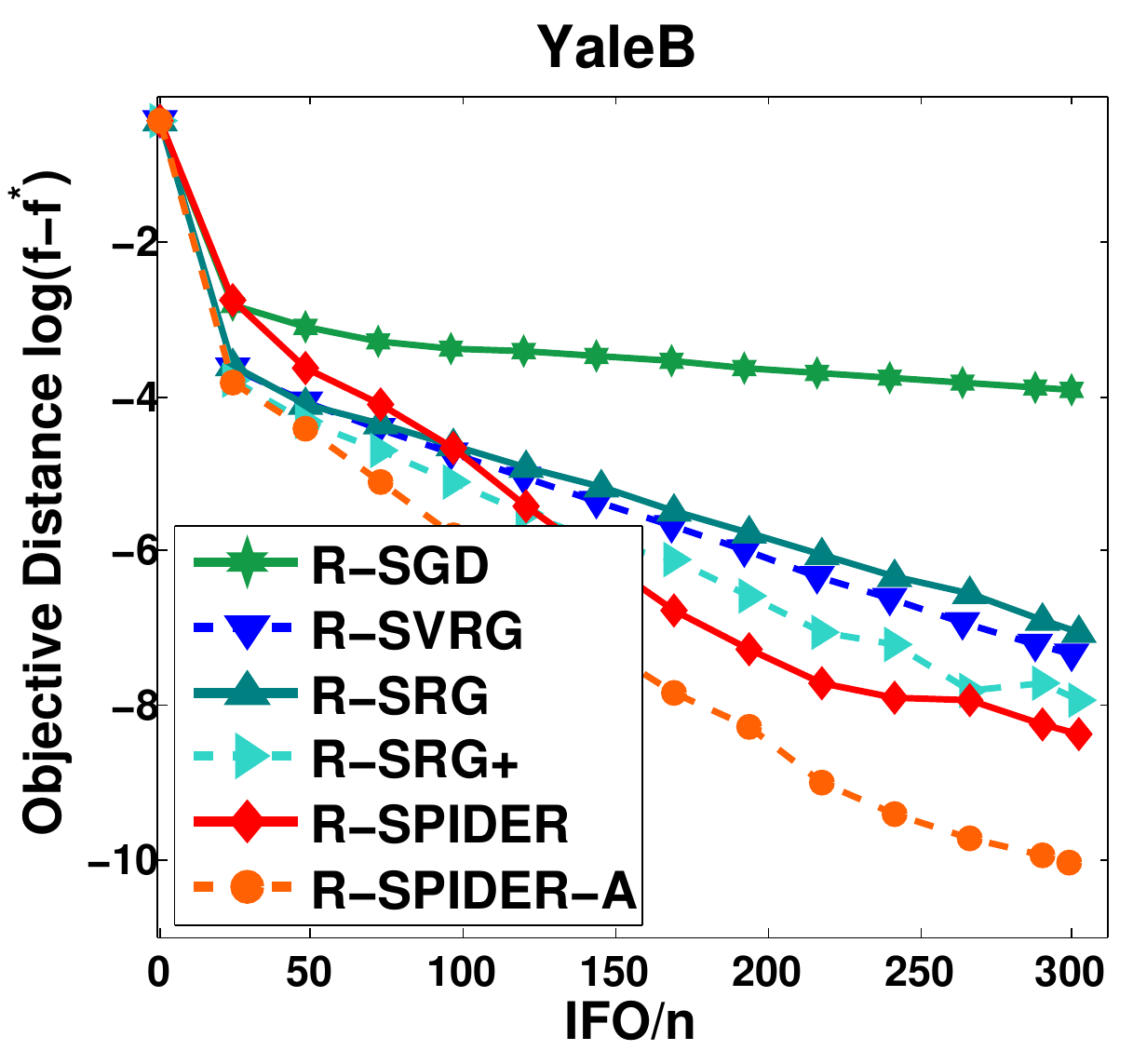}&
\includegraphics[width=0.245\linewidth]{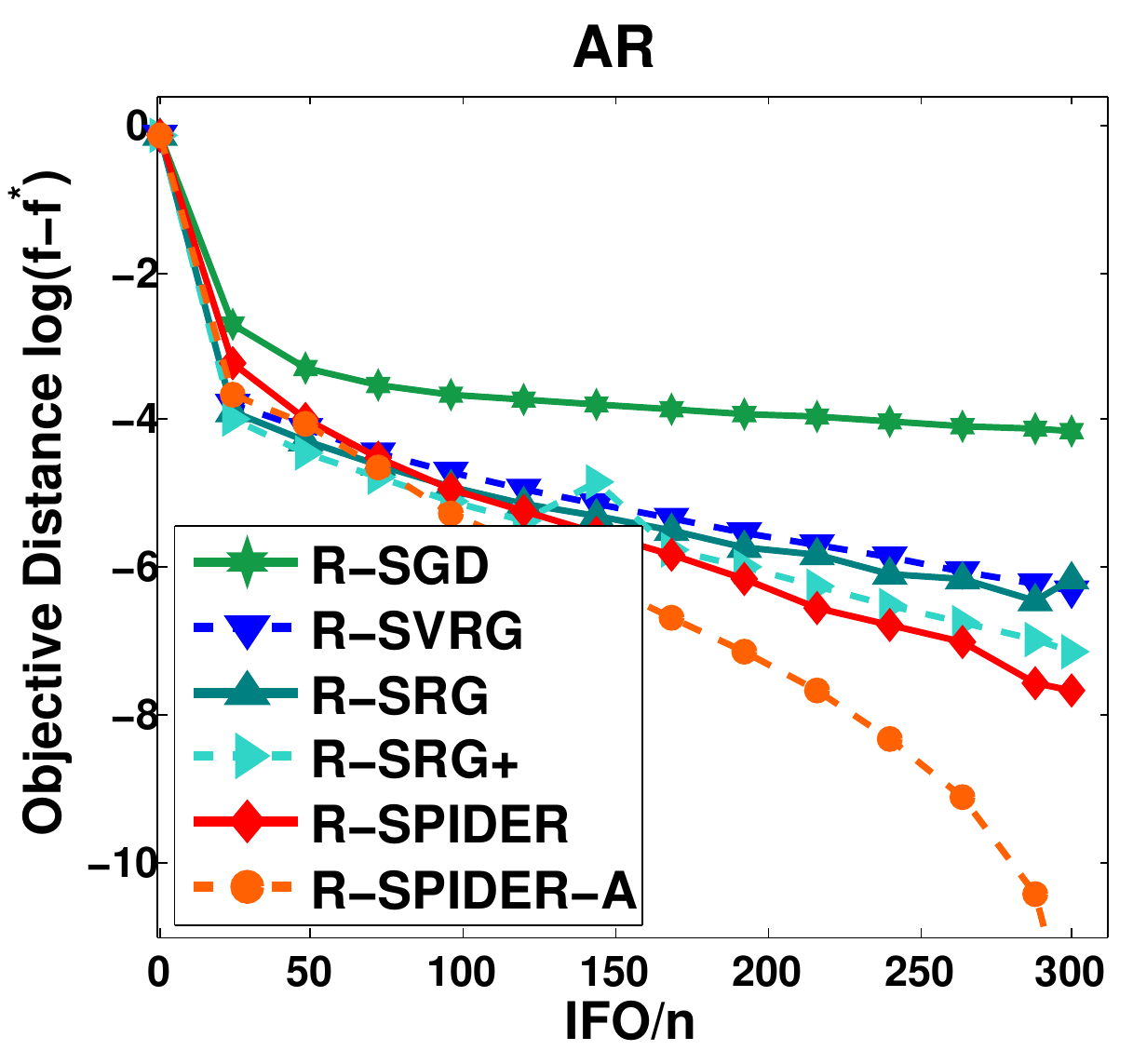}&
\includegraphics[width=0.245\linewidth]{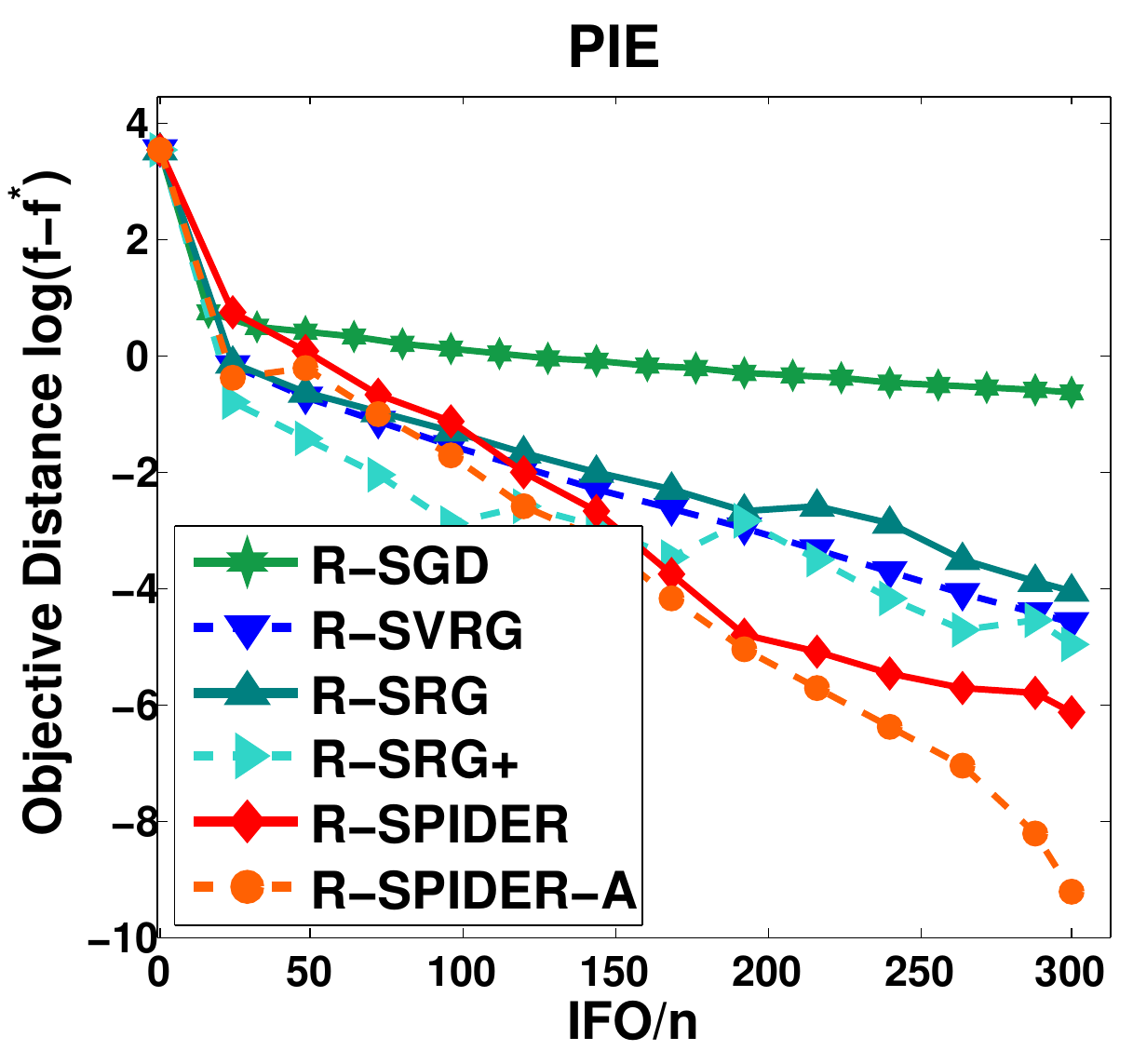}&
\includegraphics[width=0.245\linewidth]{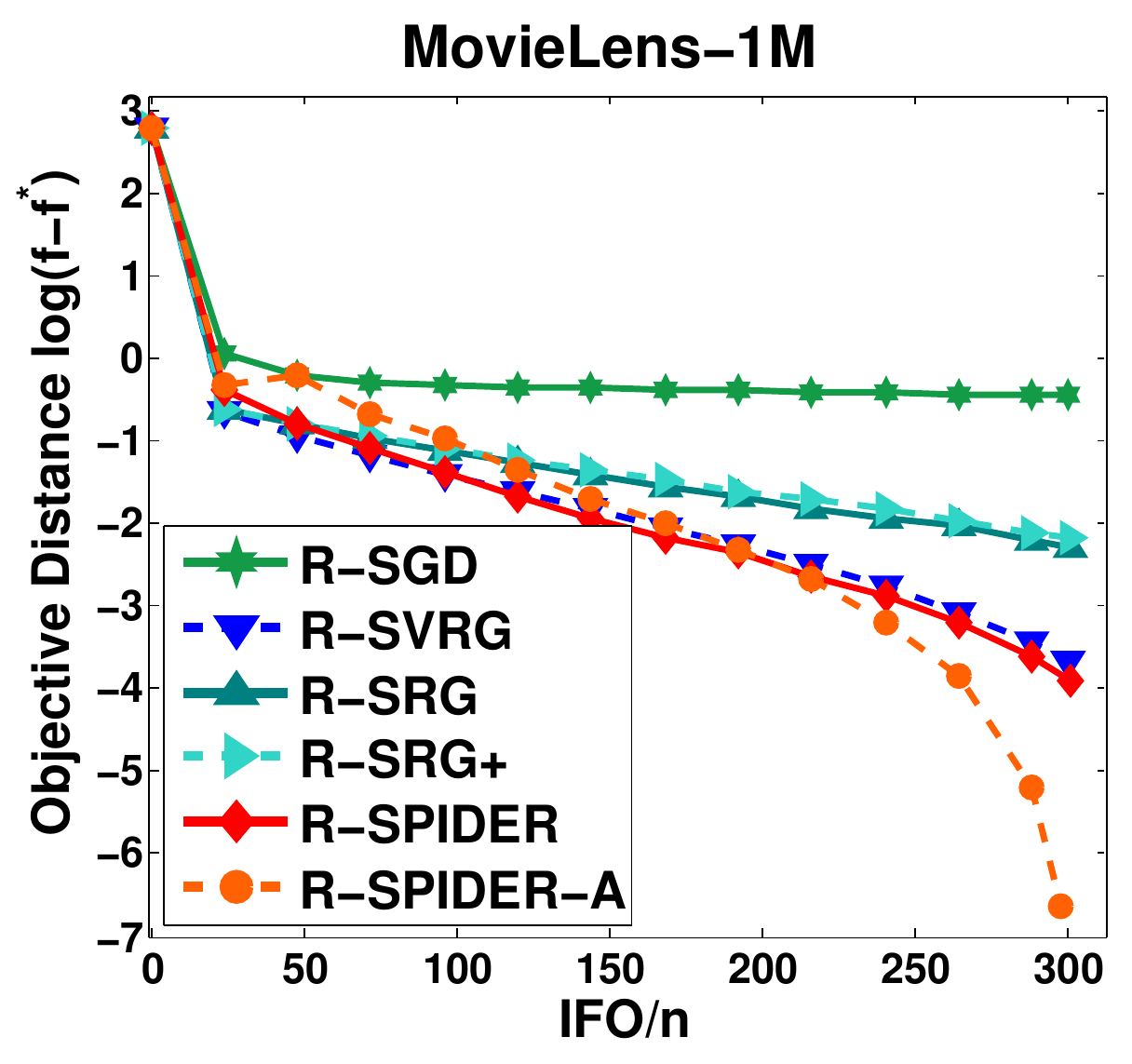}
\end{tabular}
\end{center}
\vspace{-1.4em}
\caption{Comparison among Riemannian stochastic gradient algorithms on  low-rank matrix completion problem. } \label{comparfsfasfwr}
\vspace{-1.1em}
\end{figure*}

\textbf{Evaluation on the $k$-PCA problem.} Given $n$ data points, $k$-PCA aims at computing their first $k$ leading eigenvectors, which is formulated as $  \min_{\Um\in\textsf{Gr}(k,d)} \frac{1}{n}\sum_{i=1}^{n} \ami{i}^{\top}\Um\Um^T\ami{i},$ 
where $\ami{i}\in\Rs{d}$ denotes the $i$-th sample vector and $\textsf{Gr}(k,d)\!=\!\{\Um\in \Rs{d\times k}\ | $ $\ \Um^\top\Um=\Imm\}$ denotes the Grassmann manifold. For this problem, we can directly obtain the ground truth $\Um^*$ by  using singular value decomposition (SVD), and then use $f(\Um^*)$ as optimal value $f^*$ for sub-optimality estimation in Figures~\ref{comparisonasfdasfdwith} and~\ref{compariso46543fdwith}. In this experiment, we compute the first ten leading eigenvectors.

From the experimental results in Figure~\ref{comparisonasfdasfdwith}, one can observe that our R-SPIDER-A converges significantly faster than other algorithms and R-SPIDER can also quickly converge to a relatively high accuracy, e.g. $10^{-8}$. In the initial epochs, R-SPIDER is comparable to other algorithms, showing relatively flat convergence behavior, mainly due to its very small learning rate and gradient normalization. Then along with more iterations, the computed solution becomes close to the optimum. Accordingly, the gradient begins to vanish and those considered algorithms without normalization tend to update the variable with small progress. In contrast, thanks to the normalization step, R-SPIDER moves more rapidly along the gradient descent direction and thus has sharper convergence curves. Meanwhile, R-SPIDER-A uses a relatively more aggressive learning rate in the initial epochs and decreases the learning rate along with more iterations. As a result, it exhibits the sharpest convergence behavior. On \textsf{epsilon} and \textsf{ijcnn} datasets (the bottom of Figure~\ref{comparisonasfdasfdwith}) we further plot the sub-optimality versus running-time curves. The main observations from this group of curves are consistent with those of IFO complexity, implying that the IFO complexity can comprehensively reflect the overall computational performance of a first-order Riemannian algorithm.  See Figure~\ref{comparfsfas23safsafwr} in Appendix~\ref{algorithmrunningttimecomparison} for more experimental results on running time comparison.

In Figure~\ref{compariso46543fdwith}, we compare R-SPIDER-A more closely with R-SRG-A and R-SRG+A which are respectively the counterparts of R-SRG and R-SRG+ with adaptive learning rate $\eta_k=\alpha(1+\alpha \lambda_{\alpha}\lfloor \frac{k}{p}\rfloor)$~\cite{kasai2018riemannian}. Here $\alpha$ and $\lambda_{\alpha}$ are tunable hyper-parameters. From the results, one can observe that the algorithm with adaptive learning rate usually outperforms the vanilla counterpart, which demonstrates the effectiveness of such an implementation trick. Moreover, R-SPIDER-A is   consistently superior to   R-SRG-A and R-SRG+A. See Figure~\ref{comparsdfwe45twetgfasfwr} in Appendix~\ref{moreexperimentada} for more results in this line.

\textbf{Evaluation on the low-rank matrix completion (LRMC) problem.}  Given a low-rank incomplete observation $\Am\in\Rs{d\times n}$, the LRMC problem aims at exactly recovering $\Am$. The mathematical formulation is $\min_{\Um\in\textsf{Gr}(k,d),\Gm\in\Rs{k\times n}}\ \|\PP_{\Omega}(\Am) - \PP_{\Omega}(\Um\Gm) \|^2$, where  the set $\Omega$ of locations corresponds to the observed entries, namely $(i,j)\in \Omega$ if $\Am_{ij}$ is observed. $\PP_{\Omega}$ is a linear operator that extracts entries in $\Omega$ and fills
the entries not in $\Omega$ with zeros. The LRMC problem can be expressed equivalently as
$\min_{\Um\in\textsf{Gr}(k,d),\Gm_i\in\Rs{k}}\ \frac{1}{n}\sum_{i=1}^{n} \|\PP_{\Omega_i}(\Am_i) - \PP_{\Omega_i}(\Um\Gm_i) \|^2.$
Since there is no ground truth for the optimum, we run Riemannian GD sufficiently long until the gradient satisfies $\|\nabla f(\xm)\|/\|\xm\|\leq 10^{-8}$, and then use the output as an approximate optimal value $f^*$ for sub-optimality estimation in Figure~\ref{comparfsfasfwr}. We test the considered algorithms on \textsf{YaleB}, \textsf{AR}, \textsf{PIE} and \textsf{MovieLens-1M}, considering these data approximately lie on a union of low-rank subspaces~\cite{candes2011robust,kasai2018riemannian}. For face images, we randomly sample $30\%$ pixels in each image as the observations and set $k=30$. For \textsf{MovieLens-1M}, we use its one million ratings for 3,952 movies from 6,040 users as the observations and set $k=100$.
 	
From Figure~\ref{comparfsfasfwr}, R-SPIDER-A and R-SPIDER show very similar behaviors as those in Figure~\ref{comparisonasfdasfdwith}. More specifically, R-SPIDER-A achieves fastest convergence rate, and R-SPIDER has similar convergence speed as other algorithms in the initial epochs and then runs faster along with more epochs. All these results confirm the superiority  of R-SPIDER and R-SPIDER-A.

\vspace{-0.5em}
\section{Conclusions}\label{conclusions}
\vspace{-0.3em}
We proposed R-SPIDER, which is an efficient Riemannian gradient method for non-convex stochastic optimization on Riemannian manifolds. Compared to  existing first-order Riemannian algorithms, R-SPIDER enjoys provably lower computational complexity bounds for finite-sum minimization. For online optimization, similar non-asymptotic bounds are established for R-SPIDER, which to our best knowledge has not been addressed in previous study. For the special case of gradient dominated functions, we further developed a variant of R-SPIDER with improved linear rate of convergence. Numerical results confirm the computational superiority of R-SPIDER over the state-of-the-arts.

\medskip
{\small
\bibliographystyle{unsrt}
\bibliography{referen}
}

\onecolumn
\aistatstitle{Faster First-Order Methods for Stochastic Non-Convex Optimization
on Riemannian Manifolds \\
(Supplementary File)
}
\author{Pan Zhou\\
National University of Singapore\\
{\tt\small panzhou3@gmail.com}
\and
Jiashi Feng\\
National University of Singapore\\
{\tt\small elefjia@nus.edu.sg}
}

\appendix

\begin{quote}
This supplementary document contains the technical proofs of convergence results and some additional numerical results of the manuscript entitled ``Faster First-Order Methods for Stochastic Non-convex Optimization on Riemannian Manifolds''. It is structured as follows. The proof of the key lemma, namely Lemma~\ref{lemma1} in Section~\ref{convegrenceanalysis}, is presented in Appendix~\ref{proofofauxiliarylemma}. Then  Appendix~\ref{proofofobjectivefinite} provides the proofs of the main results for general finite-sum non-convex problems in Section~\ref{convegrenceanalysis}, including Theorem~\ref{objectivefinite} and Corollary~\ref{complexityfinite}.   Next, Appendix~\ref{proofofobjection} gives the proof of the results for online setting, including Theorem~\ref{objective} and Corollary~\ref{complexityonline}. For gradient dominated results in Section~\ref{gradientdominatedfunctions}, including Theorems~\ref{totalcomplexityofconvex} and~\ref{totalcomplexityofconvex2}, are given in Appendix~\ref{proofoffinitegradient}. Finally, the detailed descriptions of datasets and more experimental results  are provided in Appendix~\ref{append:more_experiment}.
\end{quote}

\section{Proofs of Lemma~\ref{lemma1}}\label{proofofauxiliarylemma}
Before proving Lemma~\ref{lemma1}, we first present an useful lemma from~\cite{fang2018spider}. Let $Q(\xm)$ denote arbitrary determinstic vector and  $\xi_k(\xmi{0:k})$ denote the unbiased estimate $Q(\xmi{k})-Q(\xmi{k-1})$. Namely, $\EE[\xi_k(\xmi{0:k})| \xmi{0:k}]=Q(\xmi{k})-Q(\xmi{k-1})$. Then we aim to use the stochastic differential estimate to approximate $Q(\xmi{k})$ as follows:
\begin{equation*}
\tilde{Q} (\xmi{k})=\tilde{Q} (\xmi{0})+\sum_{i=1}^k \xi_i(\xmi{0:i}),
\end{equation*}
where $\tilde{Q} (\xmi{0})$ is the estimation of $Q(\xmi{0})$.
\begin{lem}\cite{fang2018spider}\label{lemma2}
	For any vector $h$, we have
	\begin{equation*}
	\EE \|\tilde{Q} (\xmi{k})-Q(\xmi{k})\|^2 \leq 	\EE \|\tilde{Q} (\xmi{0})-Q(\xmi{0})\|^2 +  \sum_{i=1}^{k}	\EE \|\xi_i(\xmi{0:i})-(Q(\xmi{i})-Q(\xmi{i-1}))\|^2.
	\end{equation*}
\end{lem}

Let $\A_i$ map any vector $\xm$ to a random vector esimate $\A_i(\xm)$ such that
\begin{equation}
\EE [\A_i(\xmi{k})-\A_i(\xmi{k-1}) | \xmi{0:k}] = \V_k -\V_{k-1},
\end{equation}
where $\V_k$ is defined below.
Assume $\A_{\SSm}=\frac{1}{|\SSm|} \sum_{i\in \SSm} \A_i$ where $\SSm$ denote the sampled data of sample number $|\SSm|$. Besides, $\A_i$ satisfies
\begin{equation*}
\EE_i\| \A_i(\xm) -\A_i(\ym)\|_2^2 \leq L^2 \|\iExp{\xm}{\ym}\|^2.
\end{equation*}
Then we define $\V_k=\A_{\SSm}(\xmi{k})-\A_{\SSm}(\xmi{k-1})+\V_{k-1}$ and $\V_0$ is the estimate of $\A(\xmi{0})$.  Based on Lemma~\ref{lemma2}, we can further conclude:
\begin{lem}\label{lemma3}
	Assume $\dis{\xmi{k-1}}{\xmi{k}}=\|\iExp{\xmi{k}}{\xmi{k-1}}\|= \rho_{k-1}$.	Then we have
	\begin{equation}
	\EE \|\V_{k} -\A(\xmi{k})\|^2 \leq  \EE \|\V_0-\A(\xmi{0})\|^2+ L^2 \sum_{i=1}^{t}\mathbb{I}_{\left\{|\SSm_i|<n\right\}} \frac{\rho_{i-1}^2}{|\SSm_i|}.
	\end{equation}
\end{lem}
\begin{proof}

	The proof here mimics that of Lemma 4 in~\cite{fang2018spider}.  For completeness, we provide the proof. Assume for the $k$-th sampling, the seleced sample set is denoted by $\SSm_k$. Then, we have
\begin{equation*}
\begin{split}
\EE \|\V_{k}-\A(\xmi{k})\|^2 = & \EE \|\A_{\SSm_k}(\xmi{k})-\A_{\SSm_k}(\xmi{k-1})+\V_{k-1}-\A(\xmi{k})\|^2\\
 = & \EE \|\A_{\SSm_k}(\xmi{k})-\A_{\SSm_k}(\xmi{k-1}) -\A(\xmi{k})+\A(\xmi{k-1})+\V_{k-1}-\A(\xmi{k-1})\|^2\\
 = & \EE \|\A_{\SSm_k}(\xmi{k})-\A_{\SSm_k}(\xmi{k-1})-\A(\xmi{k})+\A(\xmi{k-1})\|^2+\EE\|\V_{k-1}-\A(\xmi{k-1})\|^2\\
  &+ \EE \langle \A_{\SSm_k}(\xmi{k})-\A_{\SSm_k}(\xmi{k-1})-\A(\xmi{k})+\A(\xmi{k-1}), \V_{k-1}-\A(\xmi{k-1})\rangle\\
 \lee{172} & \EE \|\A_{\SSm_k}(\xmi{k})-\A_{\SSm_k}(\xmi{k-1}) -\A(\xmi{k})+\A(\xmi{k-1})\|^2+\EE\|\V_{k-1} -\A(\xmi{k-1})\|^2\\
= & \frac{1}{|\SSm_k|} \EE \|\A_{i}(\xmi{k})-\A_{i}(\xmi{k-1})-\A(\xmi{k})+\A(\xmi{k-1})\|^2+\EE\|\V_{k-1}-\A(\xmi{k-1})\|^2\\
\led{173} & \frac{1}{|\SSm_k|} \EE \|\A_{i}(\xmi{k})-\A_{i}(\xmi{k-1})\|^2+\EE\|\V_{k-1}-\A(\xmi{k-1})\|^2\\
\led{174} & \frac{L^2}{|\SSm_k|}\|\iExp{\xmi{k}}{\xmi{k-1}}\|^2+\EE\|\V_{k-1}-\A(\xmi{k-1})\|^2\\
\leq & \frac{L^2\rho_{k-1}^2}{|\SSm_k|}+\EE\|\V_{k-1}-\A(\xmi{k-1})\|^2,
\end{split}
\end{equation*}	
where \ding{172} holds since $\EE \langle \A_{\SSm_k}(\xmi{k})-\A_{\SSm_k}(\xmi{k-1})-\A(\xmi{k})+\A(\xmi{k-1}), \V_{k-1}-\A(\xmi{k-1})\rangle=\0$ in which the expectation is taken on the random set $\SSm_k$ ($\V_{k-1}-\A(\xmi{k-1})$ is constant); \ding{173} holds due to $\EE\|\xm-\EE(\xm)\|^2\leq \EE\|\xm\|^2$; \ding{174} holds since $f_i(\xm)$ is $L$-gradient Lipschitz, namely $\EE_i \|\nabla f_i(\xm) -\Tran{\ym}{\xm}{\nabla f_i(\ym)}\|_2^2 \leq L \|\iExp{\xm}{\ym}\|^2.$ Notice, when $|\SSm_k|\geq n$, in \ding{172}, we have $\EE \|\A_{\SSm_k}(\xmi{k})-\A_{\SSm_k}(\xmi{k-1}) -\A(\xmi{k})+\A(\xmi{k-1})\|^2+\EE\|\V_{k-1} -\A(\xmi{k-1})\|^2=0$. In this case, we can obtain $\EE \|\V_{k}-\A(\xmi{k})\|^2= \EE \|\V_{k-1}-\A(\xmi{k-1})\|^2$. Therefore, consider these two cases and sum up $k=0, 1,\cdots, t$, we have
	\begin{equation*}
\EE \|\V(\xmi{t})-\A(\xmi{t})\|^2 \leq \EE \|\V_0-\A(\xmi{0})\|^2+ L^2 \sum_{i=1}^{t}\mathbb{I}_{\left\{|\SSm_i|<n\right\}} \frac{\rho_{i-1}^2}{|\SSm_i|}.
\end{equation*}
The proof is completed.
\end{proof}

\begin{lem}\label{lemma22}Suppose Assumptions~\ref{GL} and~\ref{boundedgradient} hold. Let $k_0=\lfloor k/p\rfloor $ and $\widetilde{k}_0=k_0 p$. Assume that for $k=\lfloor k/p\rfloor \cdot p$, we select $|\SSm_1|$ samples to estimate $\vmi{k}$ and for $k\neq \lfloor k/p\rfloor \cdot p$, we select $|\SSm_{2,k}|$ samples to estimate $\vmi{k}$.  Then the estimation error between the full Riemannian gradient $\nabla f(\xmi{k})$ and its estimate $\vmti{k}$ in Algorithm~\ref{algmanifold} is bounded as
		\begin{equation*}
\begin{split}
&\EE \left[\|\vmti{k}-\nabla f(\xmi{k})\|^2\ |\ \xmi{\widetilde{k}_0},\cdots,\xmi{\widetilde{k}_0+p-1} \right]\leq \mathbb{I}_{\left\{|\SSm_1|<n\right\}} \frac{\sigma^2}{|\SSm_1|} +  L^2\sum_{i=\widetilde{k}_0}^{\widetilde{k}_0+p-1} \mathbb{I}_{\left\{|\SSm_{2,i+1}|<n\right\}} \frac{\diss{\xmi{i}}{\xmi{i+1}}}{|\SSm_{2,i+1}|},
\end{split}
	\end{equation*}
where $\dis{\xmi{i}}{\xmi{i+1}}$ is the distance between $\xmi{i}$ and $\xmi{i+1}$.
\end{lem}

\begin{proof}
		Here we construct an auxiliary sequence
	\begin{equation*}
	\begin{split}
	\vmhi{t}=&\sum_{i=1}^{t}\left( \Para{\xmi{i}}{\xmi{k}}{\fsi{2}(\xmi{i})}- \Para{\xmi{i-1}}{\xmi{k}}{\fsi{1}(\xmi{i-1})}\right)+ \Para{\xmi{0}}{\xmi{k}}{\fsi{1}(\xmi{0})}\\
	=& \Para{\xmi{t}}{\xmi{k}}{\fsi{2}(\xmi{t})}- \Para{\xmi{t-1}}{\xmi{k}}{\fsi{2}(\xmi{t-1})} +\vmhi{t-1},
	\end{split}
	\end{equation*}
	where $\xmi{k}$ is a given point and $\vmhi{0}=\Para{\xmi{0}}{\xmi{k}}{\fsi{1}(\xmi{0})}$. In this way, let $\A_{\SSm}(\xmi{t})=\Para{\xmi{t}}{\xmi{k}}{\fsii{2}{t}(\xmi{t})}$. Then we have $\vmhi{t}= \A_{\SSm}(\xmi{t})-\A_{\SSm}(\xmi{t-1})+\vmhi{t-1}$. 
	 Accordingly, we can obtain
	\begin{equation*}
	\begin{split}
	\EE_i\| \A_i(\xmi{t}) -\A_i(\xmi{t-1})\|_2^2=& \EE_i\left\|\Para{\xmi{t}}{\xmi{k}}{\fii{i}{}(\xmi{t})} -\Para{\xmi{t-1}}{\xmi{k}}{\fii{i}{}(\xmi{t-1})}\right\|^2\\
	=&\EE_i\left\|\Para{\xmi{t-1}}{\xmi{k}}{ \Para{\xmi{t}}{\xmi{t-1}}{\fii{i}{}(\xmi{t})}}-\Para{\xmi{t-1}}{\xmi{k}}{\fii{i}{}(\xmi{t-1})} \right\|^2\\
	=&\EE_i\left\|\Para{\xmi{t-1}}{\xmi{k}}{ \Para{\xmi{t}}{\xmi{t-1}}{\fii{i}{}(\xmi{t})}- \fii{i}{}(\xmi{t-1})} \right\|^2\\
	\lee{172}&\EE_i\left\|  \Para{\xmi{t}}{\xmi{t-1}}{\fii{i}{}(\xmi{t})}- \fii{i}{}(\xmi{t-1}) \right\|^2\\
	\leq & L^2 \|\iExp{\xmi{t-1}}{\xmi{t}}\|^2,
	\end{split}
	\end{equation*}
	where \ding{172} holds as the parallel transport $\Parass_{\xm}^{\ym}$ preserves the norm. 
	On the other hand, all $\A_i(\xmi{t})\ (t=0,\cdots, k)$ are located in the tangent space at the point $\xmi{k}$. Thus, Lemma~\ref{lemma3} is applicable to the sequence $\vmhi{t}$.
	
	Let $k_0=\lfloor K/p \rfloor$. For simplicity, we use $\V_0,\V_1,\cdots,\V_{k}$ to respectively denote $\V_{k_0},\V_{k_0+1},\cdots,\V_{k_0+k}$.
	 For $\V_0$, we have $\V_0=\Para{\xmi{k_0}}{\xmi{k}}{\nabla f_{\SSm_1}(\xmi{k_0})}$.  Then it yields
	\begin{equation*}
	\begin{split}
	\EE \|\V_0-\A(\xmi{0})\|^2= & \EE \|\Para{\xmi{k_0}}{\xmi{k}}{\nabla f_{\SSm_1}(\xmi{k_0})} - \Para{\xmi{k_0}}{\xmi{k}}{\nabla f(\xmi{k_0})} \|^2\\
	= & \EE \|\nabla f_{\SSm_1}(\xmi{k_0})-\nabla f(\xmi{k_0})\|^2\\
	=& \frac{1}{|\SSm_1|} \EE \|\nabla f_{i}(\xmi{k_0})-\nabla f(\xmi{k_0})\|^2\\
	\led{172} & \frac{\sigma^2}{|\SSm_1|},
	\end{split}
	\end{equation*}
	where \ding{172} holds since the gradient variance is bounded in Assumption~\ref{boundedgradient}.
	On the other hand, since $\xmi{k+1}= \Exp{\xmi{k}}{-\etai{k} \frac{\vmti{k}}{\|\vmti{k}\|}}$, we have
	\begin{equation*} \diss{\xmi{k+1}}{\xmi{k}}=\left\|\iExp{\xmi{k}}{\xmi{k+1}}\right\|.
	\end{equation*}
Therefore, we have
	\begin{equation*}
	\EE \|\vmhi{t}-\Para{\xmi{t}}{\xmi{k}}{\nabla f(\xmi{t})}\|^2 \leq  L^2 \sum_{i=0}^{t-1}\mathbb{I}_{\left\{|\SSm_{2,i+1}|<n\right\}} \frac{\diss{\xmi{i}}{\xmi{i+1}}}{|\SSm_{2,i+1}|} +  \frac{\sigma^2}{|\SSm_1|}.
	\end{equation*}
	By setting $t=k$ and noting $t\leq p$ for each epoch, we establish
	\begin{equation*}
	\EE \|\vmti{k}-\nabla f(\xmi{k})\|^2 =\EE \|\vmhi{k}-\Para{\xmi{k}}{\xmi{k}}{\nabla f(\xmi{k})}\|^2  \leq \frac{\sigma^2}{|\SSm_1|} +  L^2\sum_{i=k_0}^{k_0+p-1} \mathbb{I}_{\left\{|\SSm_{2,i+1}|<n\right\}} \frac{\diss{\xmi{i}}{\xmi{i+1}}}{|\SSm_{2,i+1}|}.
	\end{equation*}
	Notice, when we sample all $n$ samples, we have $\EE \|\V_0-\A(x_0)\|^2=0$ and thus
$$\EE \|\vmti{k}-\nabla f(\xmi{k})\|^2 \leq  L^2 \sum_{i=k_0}^{k_0+p-1}\mathbb{I}_{\left\{|\SSm_{2,i+1}|<n\right\}} \frac{\diss{\xmi{i}}{\xmi{i+1}}}{|\SSm_{2,i+1}|}.$$
So by combining the two case together, we can obtain the result in Lemma~\ref{lemma1}. The proof is completed.
\end{proof}

Now we are ready to prove Lemma~\ref{lemma1}.
\begin{proof}[\textbf{Proof of Lemma~\ref{lemma1}}]
To prove Lemma~\ref{lemma1}, we can directly set $|\SSm_{2,k}|$ in Lemma~\ref{lemma22} as $|\SSm_2|$ in Lemma~\ref{lemma1} and obtain the results in Lemma~\ref{lemma1}. The proof is completed.
\end{proof}

\section{Proof of the Results in Section~\ref{convegrenceanalysis}}
\subsection{Proof of Theorem~\ref{objectivefinite}}\label{proofofobjectivefinite}

\begin{proof}
For brevity, let $\etati{k}=\frac{\etai{k}}{\|\vmti{k}\|}$.
Then by using the $L$-gradient Lipschitz, we have
\begin{equation}\label{equtionloss}
\begin{split}
	f(\xmi{k+1})\leq& f(\xmi{k})+\langle \nabla f(\xmi{k}),\iExp{\xmi{k}}{\xmi{k+1}} \rangle+\frac{L}{2}\|\iExp{\xmi{k}}{\xmi{k+1}}\|^2\\
	\leq& f(\xmi{k})-\etati{k}\langle \nabla f(\xmi{k}), \vmti{k}\rangle+\frac{\etati{k}^2L}{2}\|\vmti{k}\|^2\\
	\leq& f(\xmi{k})-\etati{k}\langle \nabla f(\xmi{k})-\vmti{k}, \vmti{k}\rangle-\etati{k}\left(1-\frac{\etati{k}L}{2}\right)\|\vmti{k}\|^2\\
	\leq& f(\xmi{k})-\etati{k}\langle \nabla f(\xmi{k})-\vmti{k}, \vmti{k}\rangle-\etati{k}\left(1-\frac{\etati{k}L}{2}\right)\|\vmti{k}\|^2\\
	\leq& f(\xmi{k})+\frac{\etati{k}}{2}\| \nabla f(\xmi{k})-\vmti{k}\|^2  -\frac{\etati{k}}{2}\left(1-\etati{k}L\right)\|\vmti{k}\|^2.
\end{split}
\end{equation}

Since we have  $\xmi{k+1}= \iExp{\xmi{k}}{-\etai{k} \frac{\vmti{k}}{\|\vmti{k}\|}}$, we can obtain
\begin{equation} \label{distanceslabel}
\dis{\xmi{k+1}}{\xmi{k}}=\etai{k} = \min\left(\frac{\epsilon}{2Ln_0}, \frac{ \|\vmti{k}\|}{4L n_0}\right) \leq \frac{\epsilon}{2Ln_0}.
	\end{equation}

Now we consider the two cases: (1) $k$ is not an integer multiple of $p$; (2) $k$ is an integer multiple of $p$.  We can consider case (1) as follows. 	If $s=n$, then by Lemma~\ref{lemma1} and Eqn.~\eqref{distanceslabel}, we have
	\begin{equation*}
	\EE \|\vmti{k}-\nabla f(\xmi{k})\|^2 \leq \frac{L^2}{|\SSm_2|}\sum_{i=k_0}^{k_0+p-1}\diss{\xmi{i}}{\xmi{i+1}} \leq \frac{pL^2}{|\SSm_2|} \frac{\epsilon^2}{4L^2n_0^2}  = n_0 s^{\frac{1}{2}} L^2  \frac{n_0}{4 s^{\frac{1}{2}}} \frac{\epsilon^2 }{4 L^2n_0^2} =\frac{1}{16}\epsilon^2.
	\end{equation*}

If $s=\frac{16\sigma^2}{\epsilon^2}$, then  Lemma~\ref{lemma1} gives
\begin{equation}\label{smallvariance}
	\EE \|\vmti{k}-\nabla f(\xmi{k})\|^2 \leq \frac{pL^2}{|\SSm_2|} \frac{\epsilon^2 }{4 L^2n_0^2} +  \frac{\sigma^2}{|\SSm_1|}= n_0 s^{\frac{1}{2}} L^2 \frac{\epsilon^2 }{4 L^2n_0^2} \frac{  n_0}{4 s^{\frac{1}{2}}}+\frac{\epsilon^2}{16}=\frac{1}{8}\epsilon^2.
	\end{equation}
For case (2), namely when $k$ is an integer multiple of $p$,  we have $
	\EE \|\vmti{k}-\nabla f(\xmi{k})\|^2 \leq \frac{pL^2\etai{k}^2}{|\SSm_2|} +  \frac{\sigma^2}{|\SSm_1|}= 0 +\frac{\epsilon^2}{16}\leq \frac{1}{8}\epsilon^2. $
	At the same time, since $\etai{k} = \min\left(\frac{\epsilon}{2Ln_0}, \frac{ \|\vmti{k}\|}{4L n_0}\right)$, we have $\etati{k}= \frac{\etai{k}}{\|\vmti{k}\|} = \min\left( \frac{ \epsilon}{2Ln_0\|\vmti{k}\| }, \frac{1}{4L n_0}\right) \leq \frac{1}{4Ln_0}$ and
	\begin{equation*}
	\begin{split} \etati{k}(1-\etati{k}L)\|\vmti{k}\|^2 \geq \frac{3\etati{k}}{4} \|\vmti{k}\|^2 =  \frac{3}{8}\min\left( \frac{ \epsilon}{Ln_0\|\vmti{k}\| }, \frac{1}{2L n_0}\right) \|\vmti{k}\|^2    =  \frac{3\epsilon^2}{16Ln_0}\min\left(\frac{2 \|\vmti{k}\|}{\epsilon},\frac{ \|\vmti{k}\|^2}{\epsilon^2}\right) \ged{172} \frac{3\epsilon(2\|\vmti{k}\|-\epsilon)}{16Ln_0},
	\end{split}
	\end{equation*}
where \ding{172} uses $x^2\geq 2|x|-1$ for $\forall x$. So by taking expectation, we have
\begin{equation*}
\begin{split}
	\EE \left[f(\xmi{k+1})- f(\xmi{k})\right] \leq \frac{1}{2} \frac{1}{4Ln_0} \frac{\epsilon^2}{8} - \frac{1}{2} \frac{3\epsilon(2\|\vmti{k}\|-\epsilon)}{16Ln_0} =  -\frac{\epsilon}{64Ln_0} \left( 12\EE \|\vmti{k}\|  -7\epsilon\right).
\end{split}
\end{equation*}
In this way, we have
\begin{equation*}
	\begin{split}
\frac{1}{K}\sum_{k=0}^{K-1}\EE \|\vmti{k}\|   \leq \frac{7\epsilon}{64} +\frac{16Ln_0}{3K \epsilon}	\EE \left[f(\xmi{0})- f(\xmi{K})\right] \leq \frac{7\epsilon}{64} + \frac{16Ln_0 \Delta}{3 K \epsilon},
	\end{split}
	\end{equation*}
where we use $	\EE \left[f(\xmi{0})- f(\xmi{K})\right]  \leq \EE \left[f(\xmi{0}) -f(\xmi{*})\right] \leq   f(\xmi{0}) -f(\xmi{*})\leq \Delta$. It means that after running at most $K=\frac{14Ln_0\Delta}{\epsilon^{2}}$ iterations, the algorithm will terminate, since
\begin{equation*}
\begin{split}
	\EE\| \nabla f(\widetilde{\xm})\| = &\frac{1}{K} \sum_{k=0}^{K-1} \EE\| \nabla f(\xmi{k})\| \leq   \frac{1}{K} \sum_{k=0}^{K-1} \left[ \EE\| \nabla f(\xmi{k})-\vmti{k}\|+\EE\|\vmti{k}\| \right]
	\led{172}   \frac{1}{K} \sum_{k=0}^{K-1}
  \sqrt{ \EE\| \nabla f(\xmi{k})-\vmti{k}\|^2} + \frac{\epsilon}{2}
	\led{173}  \epsilon,
\end{split}
\end{equation*}
where \ding{172} uses the Jensen's inequality; \ding{173} holds since $\EE\| \nabla f(\xm)-\vmti{k}\|^2\leq \frac{\epsilon^2}{8}$ in Eqn.~\eqref{smallvariance}. The proof is completed.	The proof is completed.
\end{proof}

\subsection{Proof of Corollary~\ref{complexityfinite}}\label{proofofcomplexityfinite}
\begin{proof}
According to Theorem~\ref{objectivefinite}, we know that after running at most $K=\frac{14Ln_0\Delta}{\epsilon^{2}}$ iterations, the algorithm will terminate.
	In this way, we can compute the stochastic gradient complexity as
	\begin{equation*}
	\begin{split}
	\mathcal{O}\left(\frac{K}{p}|\SSm_1|+K|\SSm_2|\right)=   \mathcal{O}\left( \frac{Ln_0\Delta}{\epsilon^{2}} \left( s \frac{1}{n_0 s^{1/2}} +\frac{s^{1/2}}{2 n_0}\right)\right)  =    \mathcal{O}\left(\min\left(n+\frac{L  \Delta\sqrt{n}}{\epsilon^2},\frac{L\Delta\sigma}{\epsilon^3}\right)\right).
	\end{split}
	\end{equation*}
The proof is completed.
\end{proof}

\subsection{Proof of Theorem~\ref{objective}}\label{proofofobjection}

\begin{proof}
	For brevity, let $\etati{k}=\frac{\etai{k}}{\|\vmti{k}\|}$. From Eqn.~\eqref{equtionloss}, we can obtain the following inequality:
	\begin{equation}
	\begin{split}
	f(\xmi{k+1})\leq f(\xmi{k})+\frac{\etati{k}}{2}\| \nabla f(\xmi{k})-\vmti{k}\|^2  -\frac{\etati{k}}{2}\left(1-\etati{k}L\right)\|\vmti{k}\|^2.
	\end{split}
	\end{equation}
Now we consider the two cases: (1) $k$ is not an integer multiple of $p$; (2) $k$ is an integer multiple of $p$. We can consider case (1) as follows.
 By setting $p=\frac{\sigma n_0}{\epsilon}$, $\etai{k} = \min\left(\frac{\epsilon}{2Ln_0}, \frac{ \|\vmti{k}\|}{4L n_0}\right)$, $|\SSm_1|=\frac{16\sigma^2}{\epsilon^2}$, $\SSm_2=\frac{\sigma}{2\epsilon n_0}$, where $n_0\in[1,2\sigma/\epsilon]$, Lemma~\ref{lemma1} gives
\begin{equation}\label{safdsaf}
\EE \|\vmti{k}-\nabla f(\xmi{k})\|^2 \leq \frac{pL^2\etai{k}^2}{|\SSm_2|} +  \frac{\sigma^2}{|\SSm_1|}=\frac{\sigma n_0}{\epsilon} L^2 \frac{\epsilon^2 }{4 L^2n_0^2} \frac{\epsilon n_0}{4\sigma}+\frac{\epsilon^2}{16} = \frac{1}{8}\epsilon^2.
\end{equation}
For case (2), namely when $k$ is an integer multiple of $p$, we have $
	\EE \|\vmti{k}-\nabla f(\xmi{k})\|^2 \leq \frac{pL^2\eta^2}{|\SSm_2|} +  \frac{\sigma^2}{|\SSm_1|}= 0 +\frac{\epsilon^2}{16}\leq \frac{1}{8}\epsilon^2. $
Then similar to proof in Sec.~\ref{proofofobjectivefinite},  since $\etai{k} = \min\left(\frac{\epsilon}{2Ln_0}, \frac{ \|\vmti{k}\|}{4L n_0}\right)$, we have $\etati{k}= \frac{\etai{k}}{\|\vmti{k}\|} = \min\left( \frac{ \epsilon}{2Ln_0\|\vmti{k}\| }, \frac{1}{4L n_0}\right) \leq \frac{1}{4Ln_0}$ and
	\begin{equation*}
	\begin{split} \etati{k}(1-\etati{k}L)\|\vmti{k}\|^2 \geq \frac{3\etati{k}}{4} \|\vmti{k}\|^2 =  \frac{3}{8}\min\left( \frac{ \epsilon}{Ln_0\|\vmti{k}\| }, \frac{1}{2L n_0}\right) \|\vmti{k}\|^2    =  \frac{3\epsilon^2}{16Ln_0}\min\left(\frac{2 \|\vmti{k}\|}{\epsilon},\frac{ \|\vmti{k}\|^2}{\epsilon^2}\right) \ged{172} \frac{3\epsilon(2\|\vmti{k}\|-\epsilon)}{16Ln_0},
	\end{split}
	\end{equation*}
where \ding{172} uses $x^2\geq 2|x|-1$ for $\forall x$.

So by taking expectation, we have
\begin{equation*}
\begin{split}
	\EE \left[f(\xmi{k+1})- f(\xmi{k})\right] \leq \frac{1}{2} \frac{1}{4Ln_0} \frac{\epsilon^2}{8} - \frac{1}{2} \frac{3\epsilon(2\|\vmti{k}\|-\epsilon)}{16Ln_0} =  -\frac{\epsilon}{64Ln_0} \left( 12\EE \|\vmti{k}\|  -7\epsilon\right).
\end{split}
\end{equation*}
In this way, we have
\begin{equation*}
	\begin{split}
\frac{1}{K}\sum_{k=0}^{K-1}\EE \|\vmti{k}\|   \leq \frac{7\epsilon}{64} +\frac{16Ln_0}{3K \epsilon}	\EE \left[f(\xmi{0})- f(\xmi{K})\right] \leq \frac{7\epsilon}{64} + \frac{16Ln_0 \Delta}{3 K \epsilon},
	\end{split}
	\end{equation*}
where we use $	\EE \left[f(\xmi{0})- f(\xmi{K})\right]  \leq \EE \left[f(\xmi{0}) -f(\xmi{*})\right] \leq   f(\xmi{0}) -f(\xmi{*})\leq \Delta$. It means that after running at most $K=\frac{14Ln_0\Delta}{\epsilon^{2}}$ iterations, the algorithm will terminate, since
\begin{equation*}
\begin{split}
	\EE\| \nabla f(\widetilde{\xm})\| = &\frac{1}{K} \sum_{k=0}^{K-1} \EE\| \nabla f(\xmi{k})\| \leq   \frac{1}{K} \sum_{k=0}^{K-1} \left[ \EE\| \nabla f(\xmi{k})-\vmti{k}\|+\EE\|\vmti{k}\| \right]
	\led{172}   \frac{1}{K} \sum_{k=0}^{K-1}
  \sqrt{ \EE\| \nabla f(\xmi{k})-\vmti{k}\|^2} + \frac{\epsilon}{2}
	\led{173}  \epsilon,
\end{split}
\end{equation*}
where \ding{172} uses the Jensen's inequality; \ding{173} holds since $\EE\| \nabla f(\xm)-\vmti{k}\|^2\leq \frac{\epsilon^2}{8}$ in Eqn.~\eqref{smallvariance}. The proof is completed.	The proof is completed.
\end{proof}

\subsection{Proof of Corollary~\ref{complexityonline}}\label{proofofcomplexityonline}
\begin{proof}
We adopt similar proof sketch of Corollary~\ref{complexityfinite}.
	According to Theorem~\ref{objective}, we know that after running at most $K=\frac{14Ln_0\Delta}{\epsilon^{2}}$ iterations, the algorithm will terminate. In this way, we can compute the stochastic gradient complexity as
	\begin{equation*}
	\begin{split}
	\mathcal{O}\left(\frac{K}{p}|\SSm_1|+K|\SSm_2|\right)=   \mathcal{O}\left( \frac{Ln_0\Delta}{\epsilon^{2}} \left(\frac{\sigma^2}{\epsilon^2} \frac{\epsilon}{\sigma n_0} +\frac{\sigma}{\epsilon n_0}\right)\right)  =   \mathcal{O}\left(\frac{L\sigma \Delta}{\epsilon^3}\right).
	\end{split}
	\end{equation*}
The proof is completed.
\end{proof}

\section{Proofs of the Results in Section~\ref{gradientdominatedfunctions}}\label{proofofgradientdominated}
Before proving Theorems~\ref{totalcomplexityofconvex} and \ref{totalcomplexityofconvex2}, we first prove Lemma~\ref{taugradientcomplexity} which is a key lemma to prove Theorems~\ref{totalcomplexityofconvex} and \ref{totalcomplexityofconvex2}.

\begin{lem}\label{taugradientcomplexity}
Assume function $f(\xm)$ is $\tau$-gradient dominated. Let $\Es$ denotes the event:
\begin{equation*}
\Es = \left\{ \EE \|\nabla f(\xms)\|^2 \leq  \epsilon^2 \quad  \text{and}\quad  \EE \left[ f(\xms)-f(\xmi{*})\right] \leq  \tau \epsilon^2.  \right\}
\end{equation*}
\\
(1)  For online-setting, we have  $p\!=\!\frac{\sigma n_0}{\epsilon}$, $\etai{k} = \frac{\|\vmi{k}\|}{2Ln_0}$, $|\SSm_1|\!=\!\frac{ 32\sigma^2}{\epsilon^2}$, $|\SSm_{2,k}|\!=\!\frac{8\sigma \|\vmi{k-1}\|^2}{\epsilon^3 n_0}$. To let the event $\Es$ happen, Algorithm~\ref{algmanifold} runs at most $K=\frac{64Ln_0\Delta}{\epsilon^{2}}$ iterations and the IFO complexity is
\begin{equation*}
\mathcal{O}\left( \frac{L\Delta\sigma}{\epsilon^3} \right), \quad \text{where}\ \widetilde{\Delta}=f(\xmi{0})-f(\xmi{*}).
\end{equation*}
(2) For finite-sum setting, we let $s\!=\!\min\!\big(n,\frac{ 32\sigma^2}{\epsilon^2}\big)$, $p\!=\!n_0 s^{\frac{1}{2}} $, $\etai{k} =  \frac{\|\vmi{k}\|}{2Ln_0}$, $|\SSm_1|\!=\!s$, $|\SSm_{2,k}|\!=$ $\min\left(\frac{8p\|\vmi{k-1}\|^2}{n_0^2\epsilon^2},n\right)$. To let the event $\Es$ happen, Algorithm~\ref{algmanifold} runs at most $K=\frac{64Ln_0\Delta}{\epsilon^{2}}$ iterations and the IFO complexity is
\begin{equation*}
    \mathcal{O}\left(\min\left(n+\frac{L  \Delta\sqrt{n}}{\epsilon^2},\frac{L\Delta\sigma}{\epsilon^3}\right)\right), \quad \text{where}\ \Delta=f(\xmi{0})-f(\xmi{*}).
\end{equation*}
\end{lem}

\begin{proof}
For brevity, let $\etati{k}=\frac{\etai{k}}{\|\vmi{k}\|}=\frac{1}{2Ln_0}$.
Then similar to Eqn.~\eqref{equtionloss}, by using the $L$-gradient Lipschitz, we have
\begin{equation*}
\begin{split}
	f(\xmi{k+1})\leq& f(\xmi{k})+\langle \nabla f(\xmi{k}),\iExp{\xmi{k}}{\xmi{k+1}} \rangle+\frac{L}{2}\|\iExp{\xmi{k}}{\xmi{k+1}}\|^2\\
	\leq& f(\xmi{k})-\etati{k}\langle \nabla f(\xmi{k}), \vmti{k}\rangle+\frac{\etati{k}^2L}{2}\|\vmti{k}\|^2\\
	\leq& f(\xmi{k})-\etati{k}\langle \nabla f(\xmi{k})-\vmti{k}, \vmti{k}\rangle-\etati{k}\left(1-\frac{\etati{k}L}{2}\right)\|\vmti{k}\|^2\\
	\leq& f(\xmi{k})-\etati{k}\langle \nabla f(\xmi{k})-\vmti{k}, \vmti{k}\rangle-\etati{k}\left(1-\frac{\etati{k}L}{2}\right)\|\vmti{k}\|^2\\
	\leq& f(\xmi{k})+\frac{\etati{k}}{2}\| \nabla f(\xmi{k})-\vmti{k}\|^2  -\frac{\etati{k}}{2}\left(1-\etati{k}L\right)\|\vmti{k}\|^2\\
\led{172} & f(\xmi{k})+\frac{1}{4Ln_0}\| \nabla f(\xmi{k})-\vmti{k}\|^2  -\frac{1}{8Ln_0}\|\vmti{k}\|^2,
\end{split}
\end{equation*}
where \ding{172} holds since $n_0\geq 1$. By summing up this equation from 0 to $K-1$ and taking expectation, we can obtain
\begin{equation*}
\begin{split}
\frac{1}{K}\sum_{k=0}^{K-1}\EE \|\vmti{k}\|^2   \leq  \frac{2}{K}\sum_{k=0}^{K-1}\EE \|\vmti{k} - \nabla f(\xmi{k})\|^2	+  \frac{8Ln_0}{K}\left[f(\xmi{0})- f(\xmi{K})\right] \led{172} \frac{2}{K}\sum_{k=0}^{K-1}\EE \|\vmti{k} - \nabla f(\xmi{k})\|^2	+  \frac{8Ln_0 \Delta}{K},
\end{split}
\end{equation*}
where \ding{172} uses $	\EE \left[f(\xmi{0})- f(\xmi{K})\right]  \leq \EE \left[f(\xmi{0}) -f(\xmi{*})\right] \leq   f(\xmi{0}) -f(\xmi{*})\leq \Delta$.

Now we use Lemma~\ref{lemma22} to bound each $\EE \|\vmti{k} - \nabla f(\xmi{k})\|^2$ for both online and finite-sum setting. For online-setting, we have  $p\!=\!\frac{\sigma n_0}{\epsilon}$, $\etai{k} = \frac{\|\vmi{k}\|}{2Ln_0}$, $|\SSm_1|\!=\!\frac{ 32\sigma^2}{\epsilon^2}$, $|\SSm_{2,k}|\!=\!\frac{8\sigma \|\vmi{k-1}\|^2}{\epsilon^3 n_0}$. From Lemma~\ref{lemma22}, we can establish
	\begin{equation*}
	\EE \|\vmti{k}-\nabla f(\xmi{k})\|^2 \leq \mathbb{I}_{\left\{|\SSm_{1}|<n\right\}} \frac{\sigma^2}{|\SSm_1|} +  L^2\!\!\sum_{i=k_0}^{k_0+p-1}\!\!\mathbb{I}_{\left\{|\SSm_{2,i+1}|<n\right\}} \frac{\diss{\xmi{i}}{\xmi{i+1}}}{|\SSm_{2,i+1}|} \leq \sigma^2 \frac{\epsilon^2}{32\sigma^2} +  L^2\!\!\sum_{i=k_0}^{k_0+p-1}  \frac{\|\vmi{i}\|^2}{4L^2n_0^2}  \frac{\epsilon^3  n_0 }{8\sigma \|\vmi{i}\|^2} \leq   \frac{\epsilon^2}{16},
	\end{equation*}
where we use $\diss{\xmi{k+1}}{\xmi{k}}= \left\|\iExp{\xmi{k}}{\xmi{k+1}}\right\|^2=\etai{k}^2$ since $\xmi{k+1}= \Exp{\xmi{k}}{-\etai{k} \frac{\vmti{k}}{\|\vmti{k}\|}}$.  For finite-sum setting, we let $s\!=\!\min\!\big(n,\frac{ 32\sigma^2}{\epsilon^2}\big)$, $p\!=\!n_0 s^{\frac{1}{2}} $, $\etai{k} =  \frac{\|\vmi{k}\|}{2Ln_0}$, $|\SSm_1|\!=\!s$, $|\SSm_{2,k}|\!=$ $\min\left(\frac{8p\|\vmi{k-1}\|^2}{n_0^2\epsilon^2},n\right)$. In this case, we also have
	\begin{equation*}
	\EE \|\vmti{k}-\nabla f(\xmi{k})\|^2 \leq \mathbb{I}_{\left\{|\SSm_{1}|<n\right\}} \frac{\sigma^2}{|\SSm_1|} +  L^2\sum_{i=k_0}^{k_0+p-1}\mathbb{I}_{\left\{|\SSm_{2,i+1}|<n\right\}} \frac{\diss{\xmi{i}}{\xmi{i+1}}}{|\SSm_{2,i+1}|} \leq \sigma^2 \frac{\epsilon^2}{32\sigma^2} +  L^2\sum_{i=k_0}^{k_0+p-1} \frac{\|\vmi{i}\|^2}{4L^2n_0^2} \frac{ n_0^2\epsilon^2}{ 8p\|\vmi{i}\|^2}  \leq   \frac{\epsilon^2}{16}.
	\end{equation*}

Meanwhile, we set $K=\frac{64Ln_0\Delta}{\epsilon^{2}}$, which gives
\begin{equation*}
\begin{split}
\frac{1}{K}\sum_{k=0}^{K-1}\EE \|\vmti{k}\|^2    \leq \frac{2}{K}\sum_{k=0}^{K-1}\EE \|\vmti{k} - \nabla f(\xmi{k})\|^2	+  \frac{8Ln_0 \Delta}{K} \leq \frac{\epsilon^2}{4}.
\end{split}
\end{equation*}
It means that after running at most $K=\frac{14Ln_0\Delta}{\epsilon^{2}}$ iterations, the algorithm will terminate, since
\begin{equation*}
\begin{split}
	\EE\| \nabla f(\widetilde{\xm})\|^2 = &\frac{1}{K} \sum_{k=0}^{K-1} \EE\| \nabla f(\xmi{k})\|^2 \leq   \frac{1}{K} \sum_{k=0}^{K-1} \left[ 2\EE\| \nabla f(\xmi{k})-\vmti{k}\|^2+2\EE\|\vmti{k}\|^2 \right]
	 \leq   \epsilon^2.
\end{split}
\end{equation*}

Then we use the definition of $\tau$-gradient dominated function, we have
\begin{equation*}
\begin{split}
	\EE[ f(\widetilde{\xm}) -f(\xmi{*})] = &\frac{1}{K} \sum_{k=0}^{K-1}\EE[ f(\xmi{k}) -f(\xmi{*})] \leq   \frac{\tau }{K} \sum_{k=0}^{K-1} \EE\| \nabla f(\xmi{k})\|^2 =\tau \EE\|\nabla f(\xms)\|^2
	 \leq  \tau \epsilon^2.
\end{split}
\end{equation*}

Now consider the IFO complexity for both online and finite-sum settings. For online setting, its IFO complexity is
	\begin{equation*}
	\begin{split}
	\mathcal{O}\left(\frac{K}{p}|\SSm_1|+\sum_{k=0}^{K-1}\EE |\SSm_{2,k}|\right)=   \mathcal{O}\left( \frac{L\Delta\sigma}{\epsilon^3} + \frac{\sigma}{n_0\epsilon^3}\sum_{k=0}^{K-1} \EE \|\vmi{k}\|^2 \right) \leq  \mathcal{O}\left( \frac{L\Delta\sigma}{\epsilon^3} + \frac{\sigma}{n_0\epsilon^3} K \cdot \frac{\epsilon^2}{4}\right) = \mathcal{O}\left( \frac{L\Delta\sigma}{\epsilon^3} \right).
	\end{split}
	\end{equation*}
similarly, we can compute the expectation IFO complexity for finite-sum setting:
	\begin{equation*}
	\begin{split}
	\mathcal{O}\left(\frac{K}{p}|\SSm_1|+\sum_{k=0}^{K-1}\EE |\SSm_{2,k}|\right)=      \mathcal{O}\left(\min\left(n+\frac{L  \Delta\sqrt{n}}{\epsilon^2},\frac{L\Delta\sigma}{\epsilon^3}\right)\right).
	\end{split}
	\end{equation*}
The proof is completed.
\end{proof}

\subsection{Proof of  Theorems~\ref{totalcomplexityofconvex} }\label{proofoffinitegradient}

Now we are ready to prove Theorem~\ref{totalcomplexityofconvex}.

\begin{proof}
We first consider the $t$ iteration in Algorithm~\ref{algmanifold2}.
By Lemma~\ref{taugradientcomplexity}, we obtain that by using $\epsilon_{t-1}$ with proper other parameters, the IFO complexity of Algorithm~\ref{algmanifold} for computing $\EE[\|\nabla f(\widetilde{\xm}_t)\|^2]\leq \epsilon_{t-1}^2$ is $$\mathcal{O}\big(\min\big(n+\frac{L  \Delta_t\sqrt{n}}{\epsilon_{t-1}^2},\frac{L\Delta_t\sigma}{\epsilon_{t-1}^3}\big)\big),$$ when the parameters satisfy $s_t=\!\min\!\big(n,\frac{ 32\sigma^2}{\epsilon_{t-1}^2}\big)$, $p^t\!=\!n_0^t s_t^{\frac{1}{2}}$, $\etai{k}^t =  \frac{\|\vmi{k}^t\|}{2Ln_0}$, $|\SSm_1^t|\!=\!s_t$, $|\SSm_{2,k}^t|\!=$ $\min (\frac{8p^t\|\vmi{k-1}^t\|^2}{(n_0^t)^2\epsilon_{t-1}^2},n )$  and $K^t=\frac{64Ln_0^t\Delta^t}{\epsilon_{t-1}^{2}}$. Then the initial point $\xmi{0}$ at the $t$ iteration is the output $\xmti{t-1}$ of the $(t-1)$-th iteration, which gives the distance { $\Delta_t=\EE [f(\xmi{0})-f(\xmi{*})]= \EE [f(\xmti{t-1})-f(\xmi{*})]\leq \tau \epsilon_{t-2}^2$} by using Lemma~\ref{taugradientcomplexity}. On the other hand, $\epsilon_t=\frac{\epsilon_0}{2^t}$. So the IFO complexity of the $t$-th iteration is
\begin{equation*}
\mathcal{O}\big(\min\big(n+\frac{L  \Delta_t\sqrt{n}}{\epsilon_{t-1}^2}, \frac{L\Delta_t\sigma}{\epsilon_{t-1}^3}\big)\big)=
\mathcal{O}\big(\min\big(n+\frac{L  \tau \epsilon_{t-2}^2 \sqrt{n}}{\epsilon_{t-1}^2}, \frac{L \sigma \tau \epsilon_{t-2}^2}{\epsilon_{t-1}^3}\big)\big)=
\mathcal{O}\left(\min\left(n+ \tau L\sqrt{n}, \frac{\tau L\sigma}{\epsilon_{t-1}}\right)\right).
\end{equation*}

So  to achieve  $\epsilon_T\leq \frac{\epsilon_0}{2^T}\leq \epsilon$, $T$ satisfies $T\geq \log\left(\frac{\epsilon_0}{\epsilon}\right)$.	 So for the $T$ iterations, the total complexity is
\begin{equation*}
\mathcal{O}\left(\min\left(\left(n+ \tau L\sqrt{n}\right)\log\left(\frac{1}{\epsilon}\right), \tau L\sigma\sum_{t=1}^{T} \frac{1}{\epsilon_{t-1}}\right)\right) = \mathcal{O}\left(\min\left(\left(n+ \tau L\sqrt{n}\right)\log\left(\frac{1}{\epsilon}\right), \frac{\tau L\sigma }{\epsilon}\right)\right).
\end{equation*}

Meanwhile, we can obtain
\begin{equation*}
\EE \|\nabla f(\xmti{t})\| \leq \sqrt{ \EE \|\nabla f(\xmti{t})\|^2} \leq  \epsilon_{t-1} = \frac{\epsilon_0}{2^{t-1}}= \frac{1}{2^t}\sqrt{\frac{\Delta}{\tau}} \quad  \text{and}\quad  \EE \left[ f(\xmti{t})-f(\xmi{*})\right] \leq  \tau \epsilon_{t-1}^2=\frac{\tau \epsilon_0^2}{4^{t-1}}=\frac{ \Delta}{4^t},
\end{equation*}
where we set $\epsilon_0=\frac{1}{2} \sqrt{\frac{\Delta}{\tau}}$.  The proof is completed.
\end{proof}

\subsection{Proof of Theorem~\ref{totalcomplexityofconvex2}}\label{proofofgradientonline}
\begin{proof}
The proof here is very similar to the strategy in Section~\ref{proofoffinitegradient} for proving Theorem~\ref{totalcomplexityofconvex}. The main idea is to use the result in Lemma~\ref{taugradientcomplexity}, to achieve
\begin{equation*}
\EE \|\nabla f(\xms)\|^2 \leq  \epsilon^2 \quad  \text{and}\quad  \EE \left[ f(\xms)-f(\xmi{*})\right] \leq  \tau \epsilon^2,
\end{equation*}
 the IFO complexity is
\begin{equation*}
\mathcal{O}\left( \frac{L\Delta\sigma}{\epsilon^3} \right), \quad \text{where}\ \widetilde{\Delta}=f(\xmi{0})-f(\xmi{*}).
\end{equation*}

Then following the proof in Section~\ref{proofoffinitegradient} for proving Theorem~\ref{totalcomplexityofconvex}, we can obtain the IFO complexity for achieving $\EE \|\nabla f(\xmti{t})\|^2\leq \epsilon_{t-1}^2 $:
	\begin{equation*}
	\begin{split}
	 \mathcal{O}\left( \frac{\tau L \sigma}{\epsilon}\right),
	\end{split}
	\end{equation*}
when the parameters obey  $p_t\!=\!\frac{\sigma n_0^t}{\epsilon_{t-1}}$, $\etai{k}^t = \frac{\|\vmi{k}^{t}\|}{2Ln_0^t}$, $|\SSm_1^t|\!=\!\frac{ 32\sigma^2}{\epsilon_{t-1}^2}$, $|\SSm_{2,k}^t|\!=\!\frac{8\sigma \|\vmi{k-1}^{t}\|^2}{\epsilon_{t-1}^3 n_0^t}$, and $K^t=\frac{64Ln_0^t\Delta^t}{\epsilon_{t-1}^{2}}$.

Meanwhile, we can obtain
\begin{equation*}
\EE \|\nabla f(\xmti{t})\| \leq \sqrt{ \EE \|\nabla f(\xmti{t})\|^2} \leq  \epsilon_{t-1} = \frac{\epsilon_0}{2^{t-1}}= \frac{1}{2^t}\sqrt{\frac{\Delta}{\tau}} \quad  \text{and}\quad  \EE \left[ f(\xmti{t})-f(\xmi{*})\right] \leq  \tau \epsilon_{t-1}^2=\frac{\tau \epsilon_0^2}{4^{t-1}}=\frac{ \Delta}{4^t},
\end{equation*}
where we set $\epsilon_0=\frac{1}{2} \sqrt{\frac{\Delta}{\tau}}$.  The proof is completed.
\end{proof}

\section{More Experimental Results}\label{append:more_experiment}

\subsection{Descriptions of Testing Datasets}
We first briefly introduce the ten testing datasets in the manuscript. Among them, there are six datasets, including   \textsf{a9a}, \textsf{satimage},  \textsf{covtype}, \textsf{protein}, \textsf{ijcnn1} and \textsf{epsilon}, that are provided in the LibSVM website\footnote[1]{https://www.csie.ntu.edu.tw/~cjlin/libsvmtools/datasets/}.
We also evaluate our algorithms on the three datasets: \textsf{YaleB}~\cite{georghiades2001few}, \textsf{AR}~\cite{AR} and \textsf{PIE}~\cite{sim2003cmu}, which are very commonly used face classification datasets.  Finally, we also test those algorithms on a movie recommendation dataset, namely \textsf{MovieLens-1M}\footnote[2]{https://grouplens.org/datasets/movielens/1m/}. Their detailed information is summarized in Table~\ref{Tabledatasets}. From it we can observe that these datasets are different from each other due to their feature dimension, training samples, and class numbers, \textit{etc}. 
\begin{table}[h]
\caption{Descriptions of the ten testing datasets.}
\vspace{-0.8em}
\setlength{\tabcolsep}{6.5pt} 
\renewcommand{\arraystretch}{0.94}
\label{Tabledatasets}
\centering
{ \footnotesize {
\begin{tabular}{lccc|lccc}
\toprule
 & $\#$class& $\#$sample& $\#$feature & & $\#$class& $\#$sample& $\#$feature\\  \midrule
 \textsf{a9a} & 2 & 32,561	&	123 & \textsf{epsilon} & 2 &40,000	&	2000 \\
\textsf{satimage} & 6	&4,435      &    36 & \textsf{YaleB} & 38 &  2,414& 2,016\\
\textsf{covtype} & 2 & 581,012	&54 & \textsf{AR} & 100& 2,600 &1,200 \\
\textsf{protein}& 3	& 14,895   &     357 & \textsf{PIE}&  64& 11,554 &1,024 \\
\textsf{ijcnn1} & 2 &49,990	&	22 & \textsf{MovieLens-1M}& ---& 6,040&3,706  \\
\bottomrule
  \hline
\end{tabular}
}}
\vspace{-0em}
\end{table}

\subsection{Comparison of Algorithm Running Time}\label{algorithmrunningttimecomparison}
In this subsection, we present more experimental results to show the algorithm running time comparison among the compared algorithms in the manuscript. The experimental results in Figure~\ref{comparisonasfdasfdwith} only provides the algorithm running time comparison of the \textsf{ijcnn} and \textsf{epsilon} datasets. Here we provide the comparison of all remaining datasets in Figure~\ref{comparfsfas23safsafwr} which respond to Figures~\ref{comparisonasfdasfdwith} and \ref{comparfsfasfwr} in the manuscript. From the curves of comparison of optimality  gap vs. algorithm running time, one can observe that    our R-SPIDER-A is the fastest method and   R-SPIDER can also quickly converge to a relatively high accuracy, \emph{e.g.} $10^{-8}$. We have discussed these results in the manuscript. Besides, all these results are consistent with the curves of the comparison of optimality  gap vs. IFO, since  the IFO complexity  can  comprehensively reflect the overall computational performance of a first-order Riemannian algorithm.

\begin{figure*}[ht]
\begin{center}
\setlength{\tabcolsep}{0.8pt} 
\begin{tabular}{cccc}
\includegraphics[width=0.2454\linewidth]{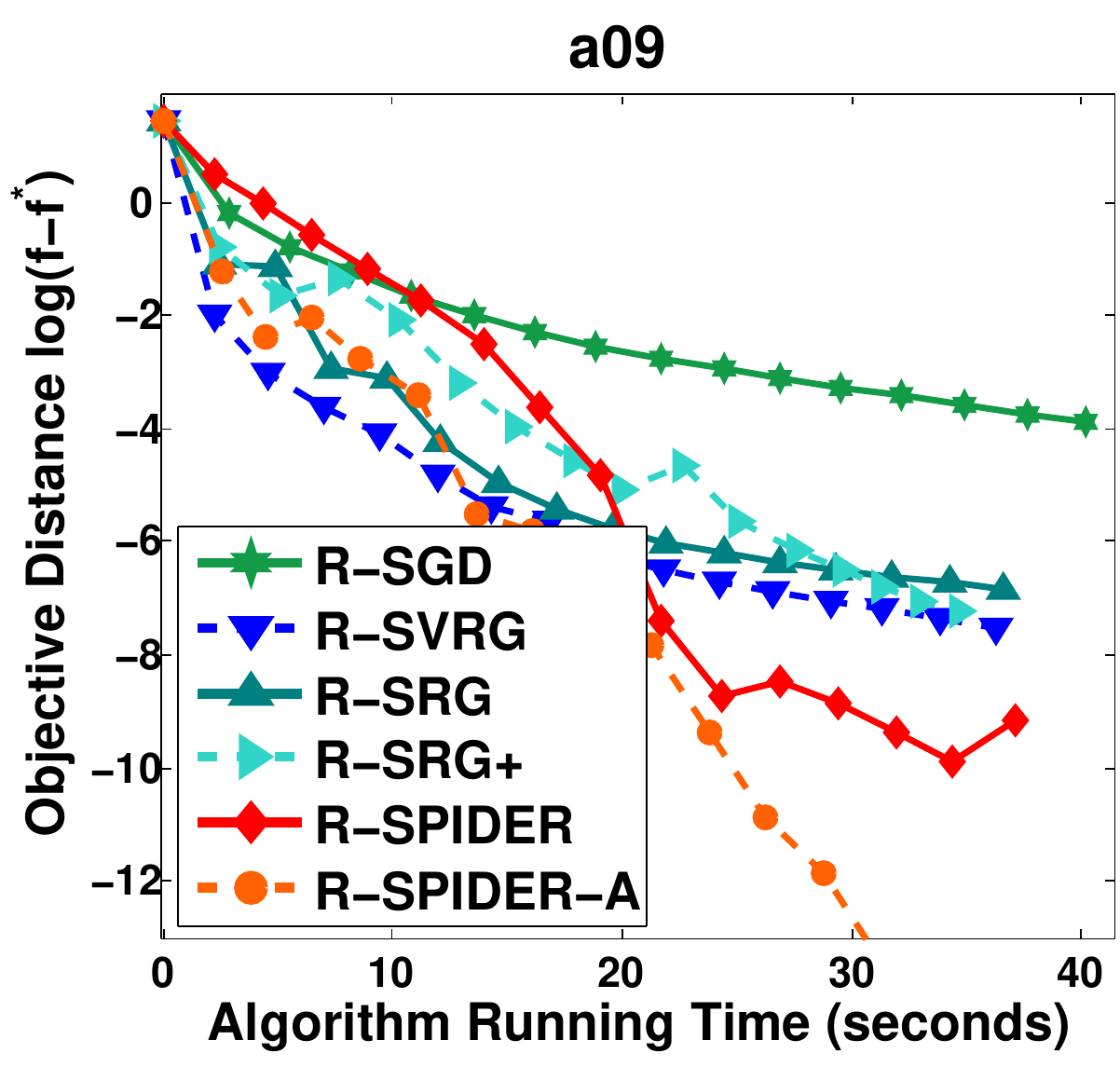}&
\includegraphics[width=0.2454\linewidth]{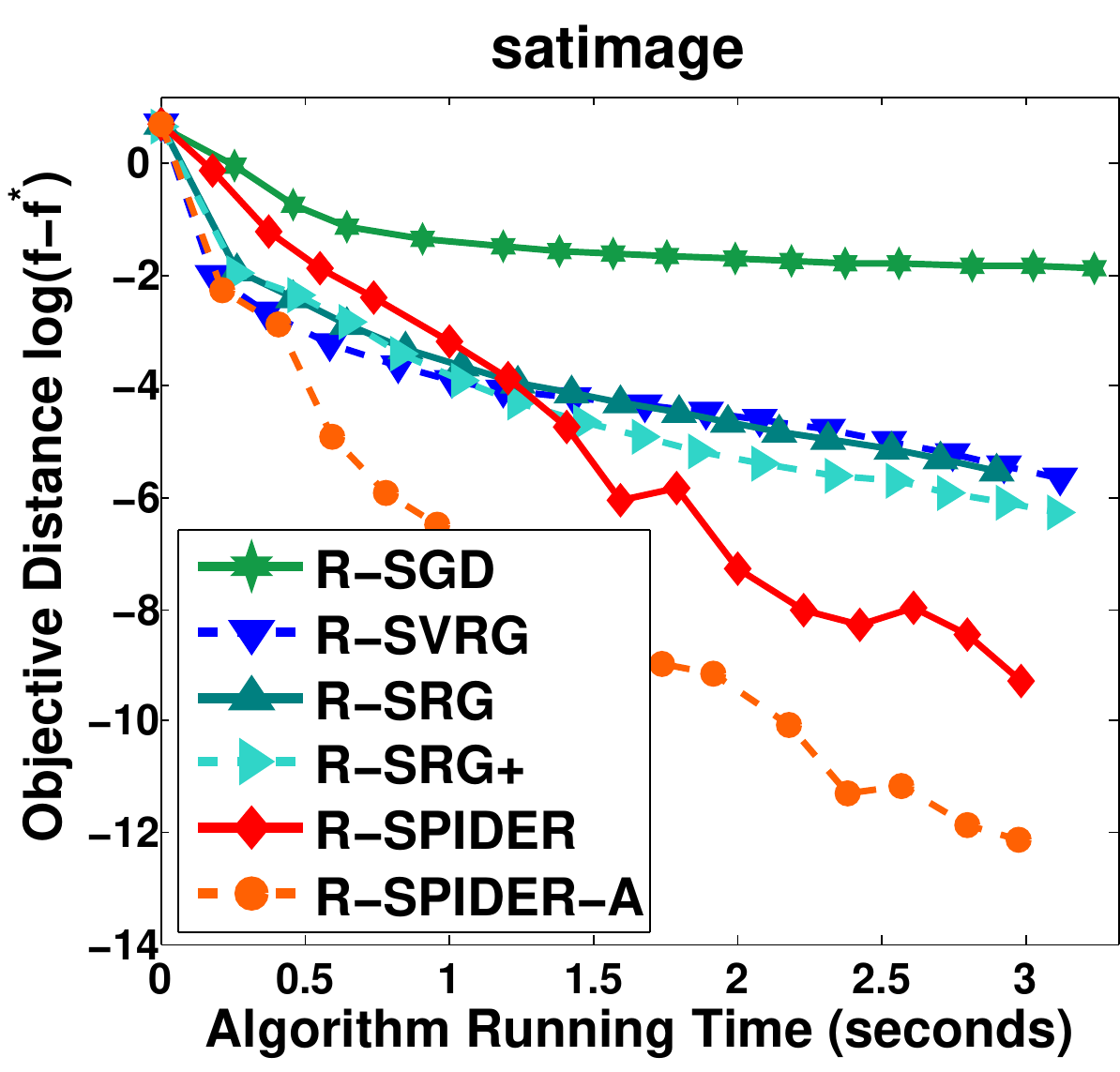}&
\includegraphics[width=0.2454\linewidth]{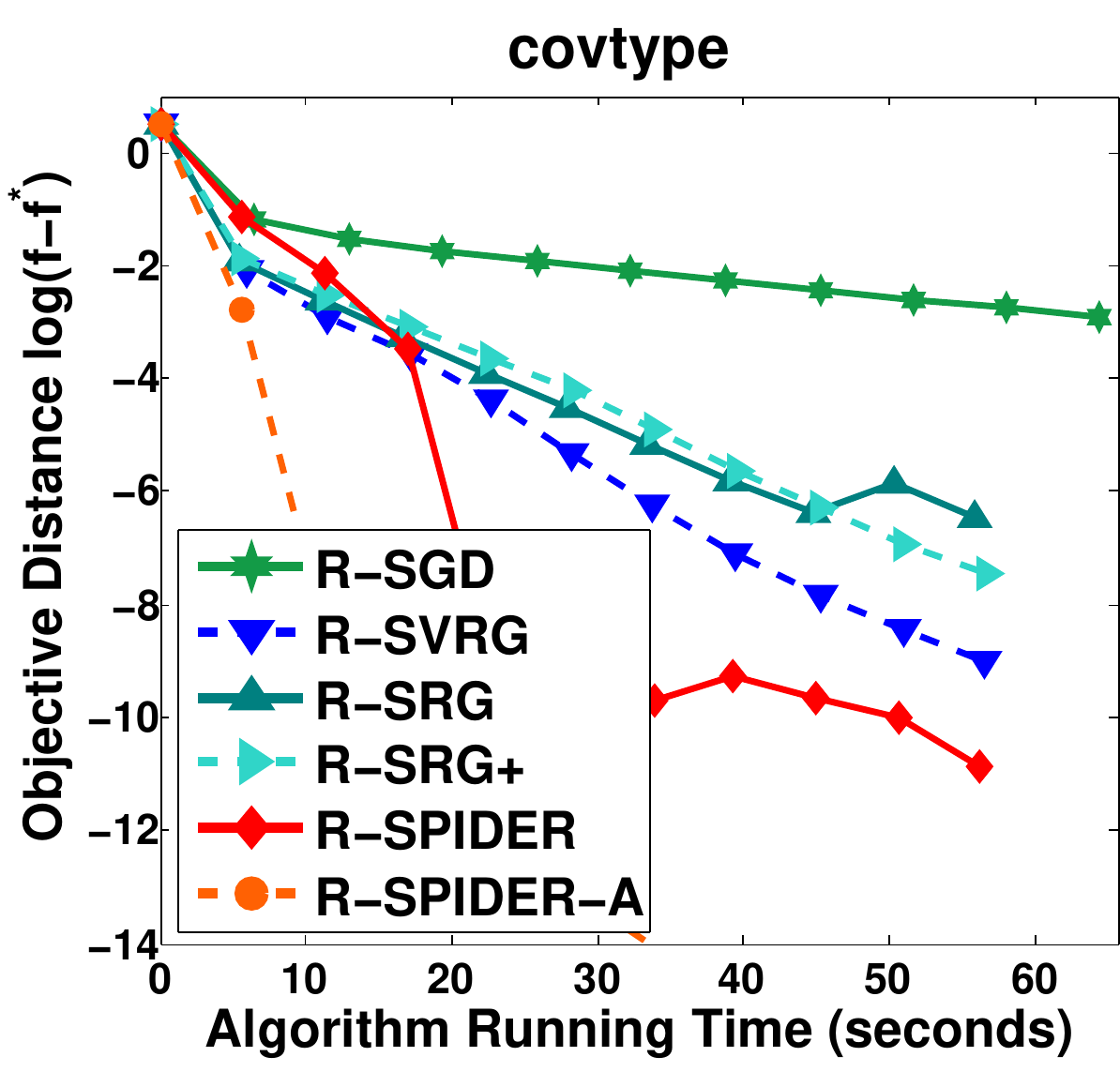}&
\includegraphics[width=0.2454\linewidth]{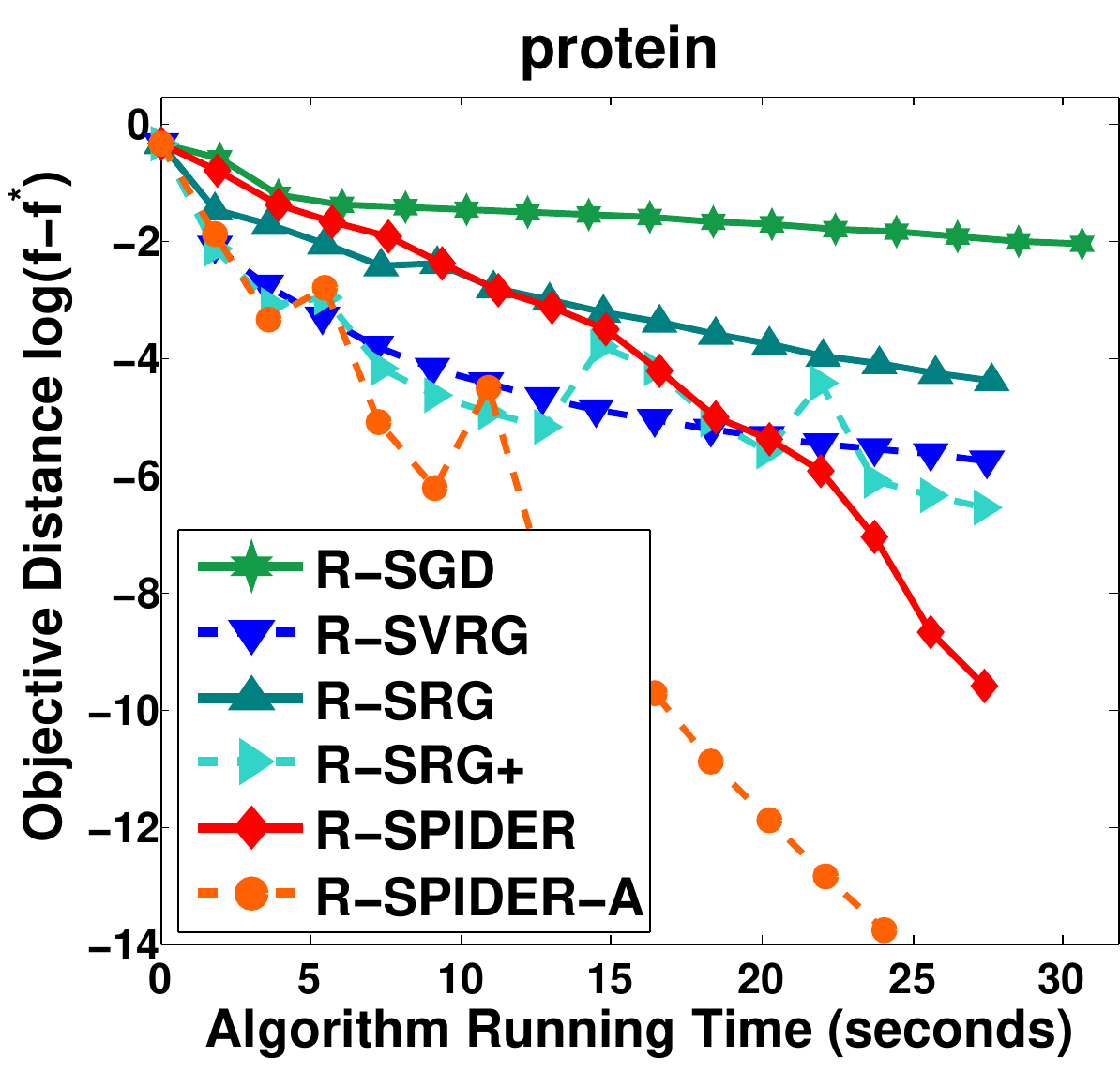}\\
\multicolumn{4}{c}{{(a) Comparison among Riemannian stochastic gradient algorithms on $k$-PCA problem.}}\\
\includegraphics[width=0.2454\linewidth]{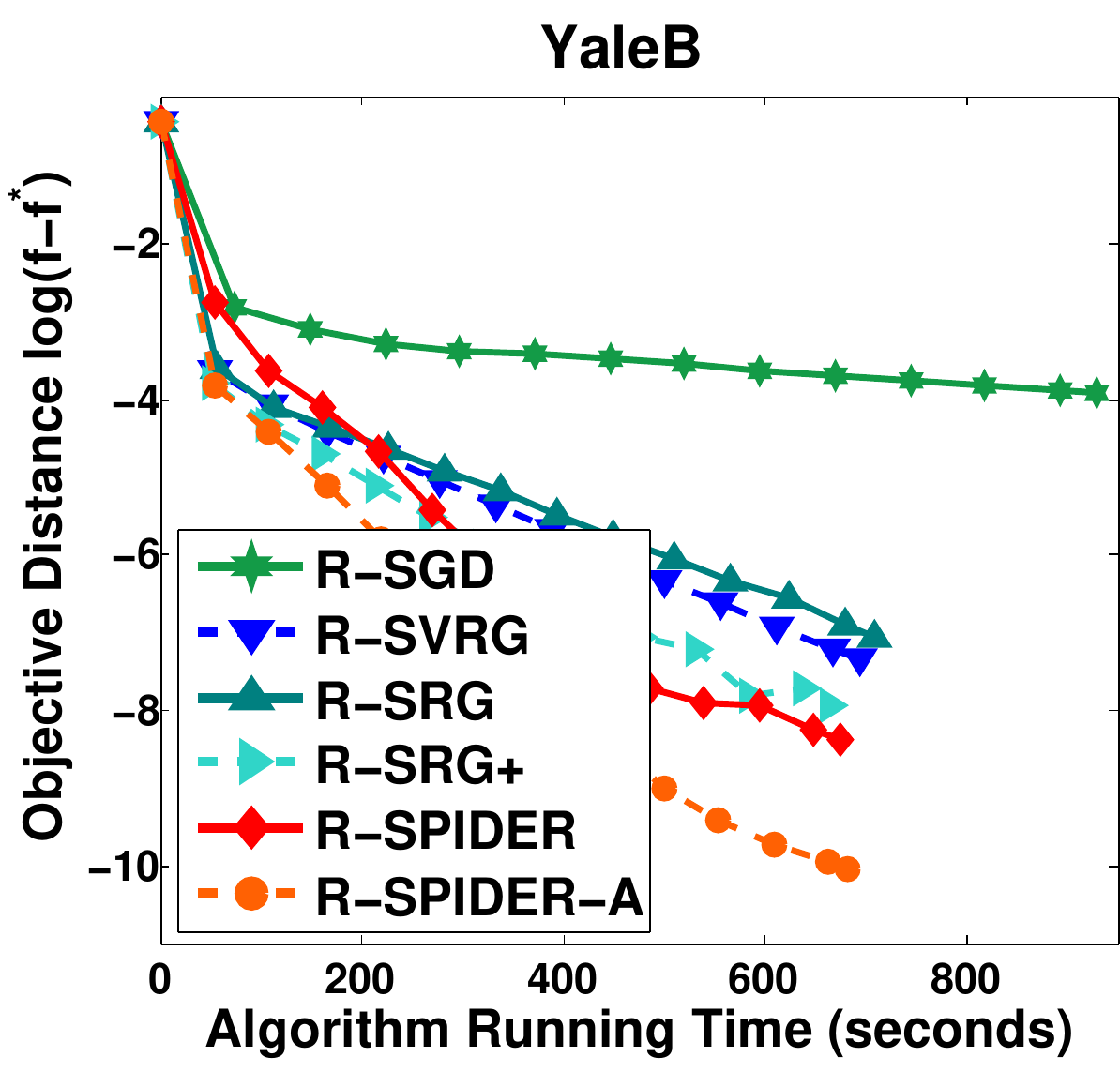}&
\includegraphics[width=0.2454\linewidth]{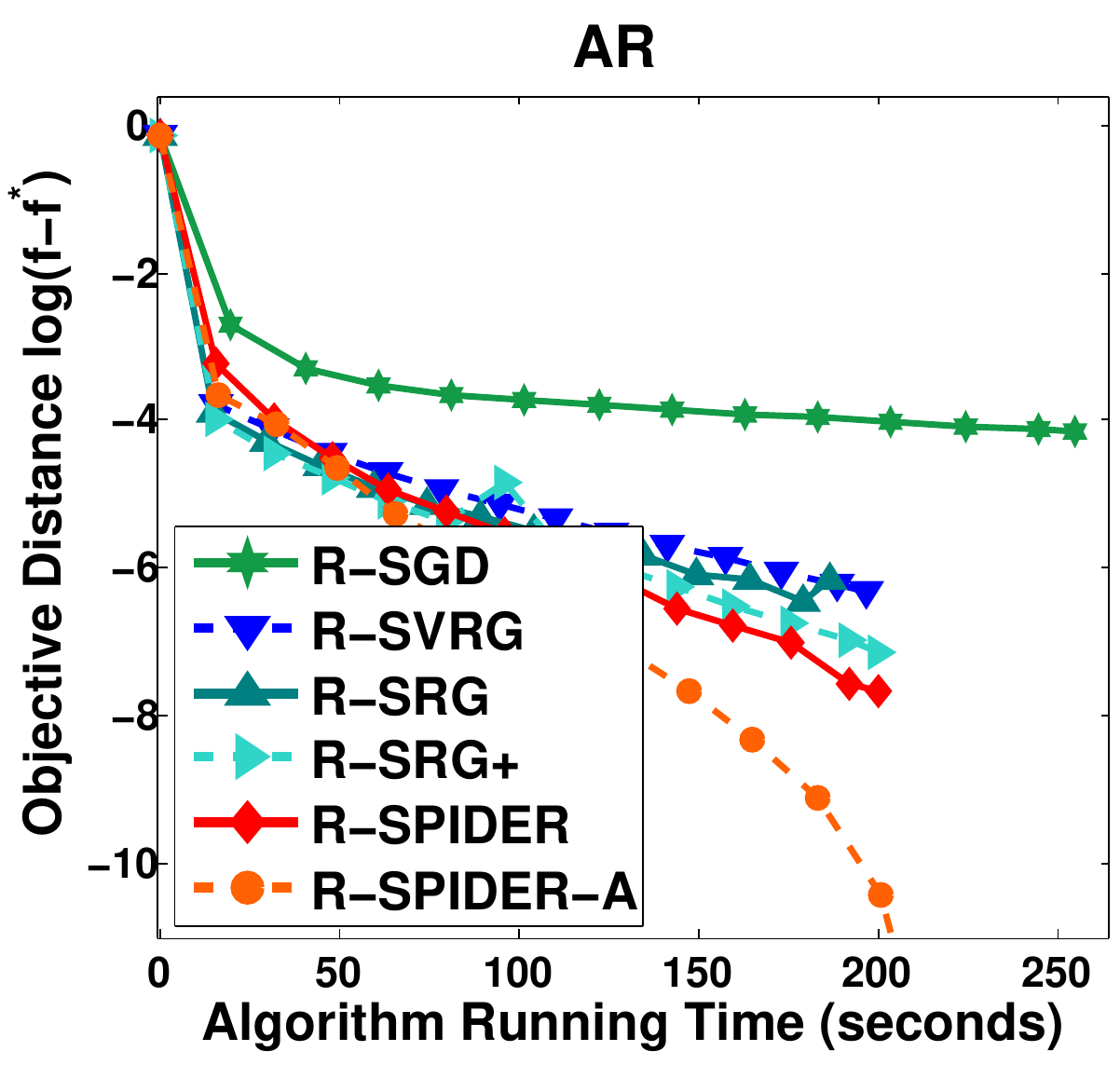}&
\includegraphics[width=0.2454\linewidth]{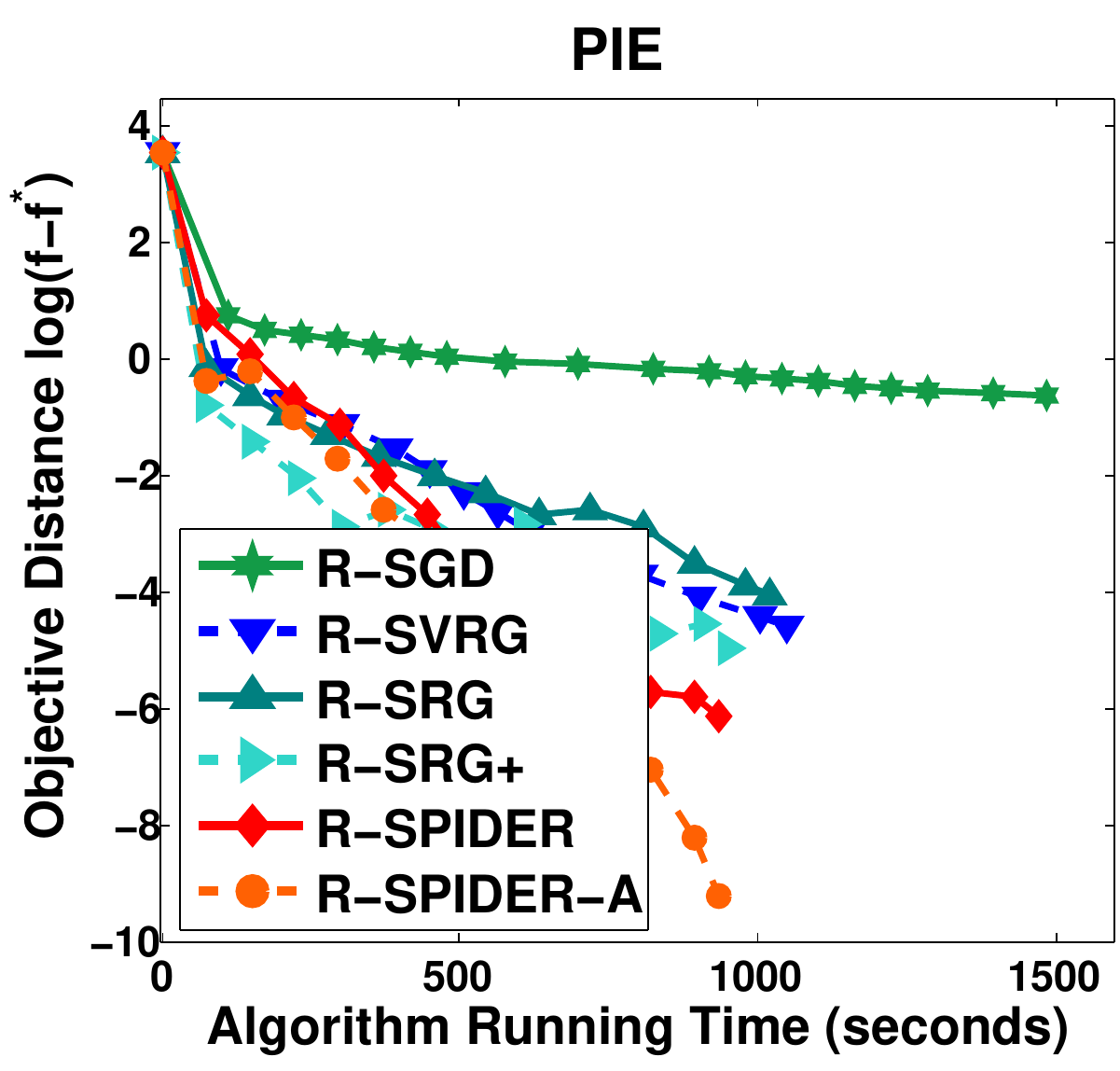}&
\includegraphics[width=0.2454\linewidth]{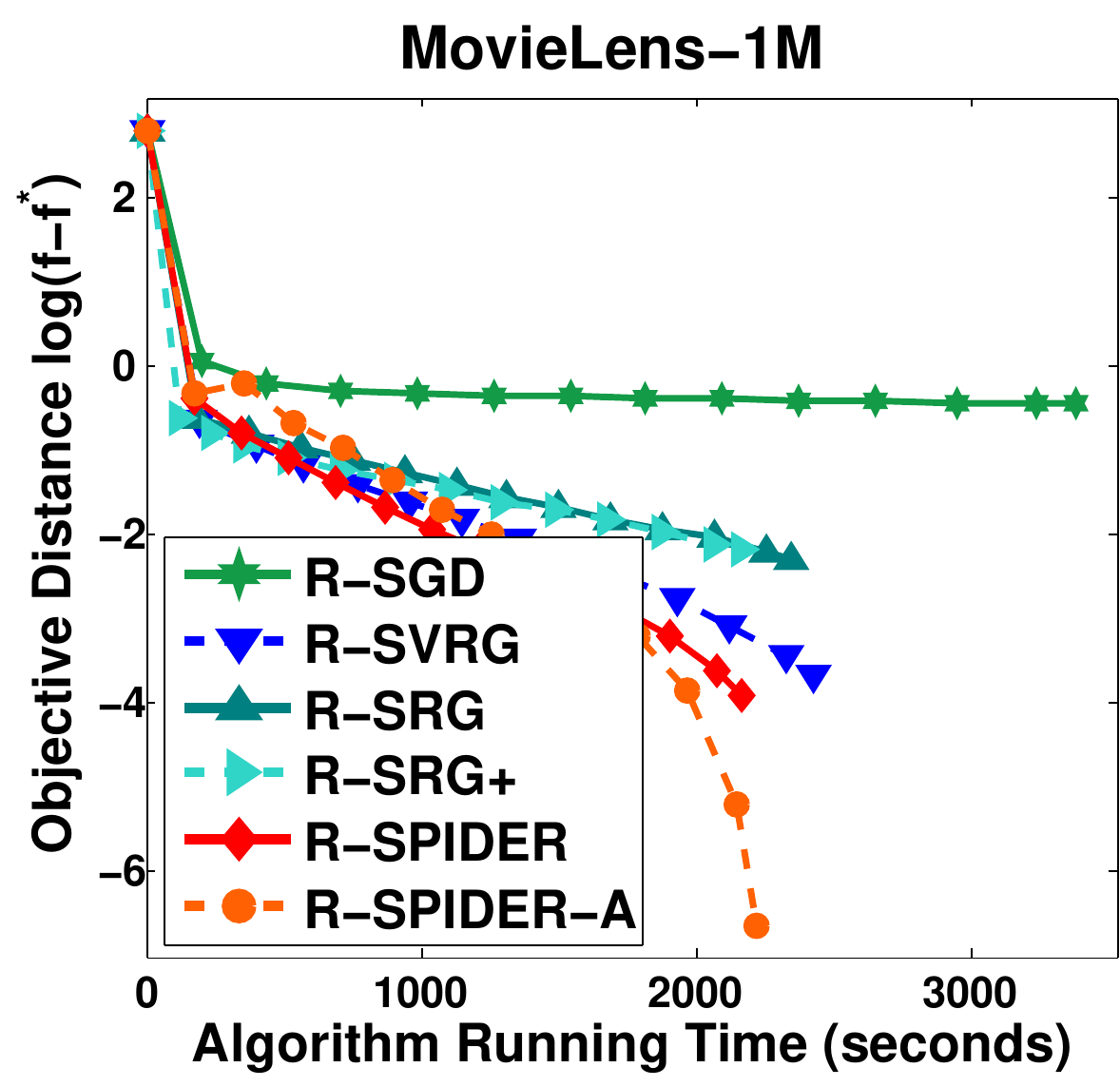}\\
\multicolumn{4}{c}{{(b) Comparison among Riemannian stochastic gradient algorithms on low-rank matrix completion problem.}}
\end{tabular}
\end{center}
\vspace{-1.2em}
\caption{Comparison of algorithm running time of Riemannian stochastic gradient algorithms. } \label{comparfsfas23safsafwr}
\vspace{-1.0em}
\end{figure*}

\subsection{Comparison between Riemannian Stochastic Gradient Algorithms with Adaptive Learning Rate}\label{moreexperimentada}
Here we provide more comparison among our proposed R-SPIDER-A, R-SRG-A and R-SRG+A.  R-SRG-A and R-SRG+A are respectively the counterparts of R-SRG and R-SRG+ with adaptive learning rate of formulation $\eta_k=\alpha(1+\alpha \lambda_{\alpha}\lfloor \frac{k}{p}\rfloor)$~\cite{kasai2018riemannian}. Notice, the reason that we do not compare all algorithms together is to avoid too many curves in one figure, leading to poor readability.

By observing Figure~\ref{comparsdfwe45twetgfasfwr}, we can find that the algorithm with adaptive learning rate usually outperforms the vanilla counterpart, which demonstrates the effectiveness of the strategy of adaptive learning rate. Moreover, R-SPIDER-A also consistently shows sharpest convergence behaviors compared with R-SRG-A and R-SRG+A. All these results are consistent with the experimental results in the manuscript. All results shows the advantages of our proposed R-SPIDER and R-SPIDER-A.

\begin{figure*}[h!p!b!t]
\begin{center}
\setlength{\tabcolsep}{0.8pt} 
\begin{tabular}{cccc}
 &
\includegraphics[width=0.2454\linewidth]{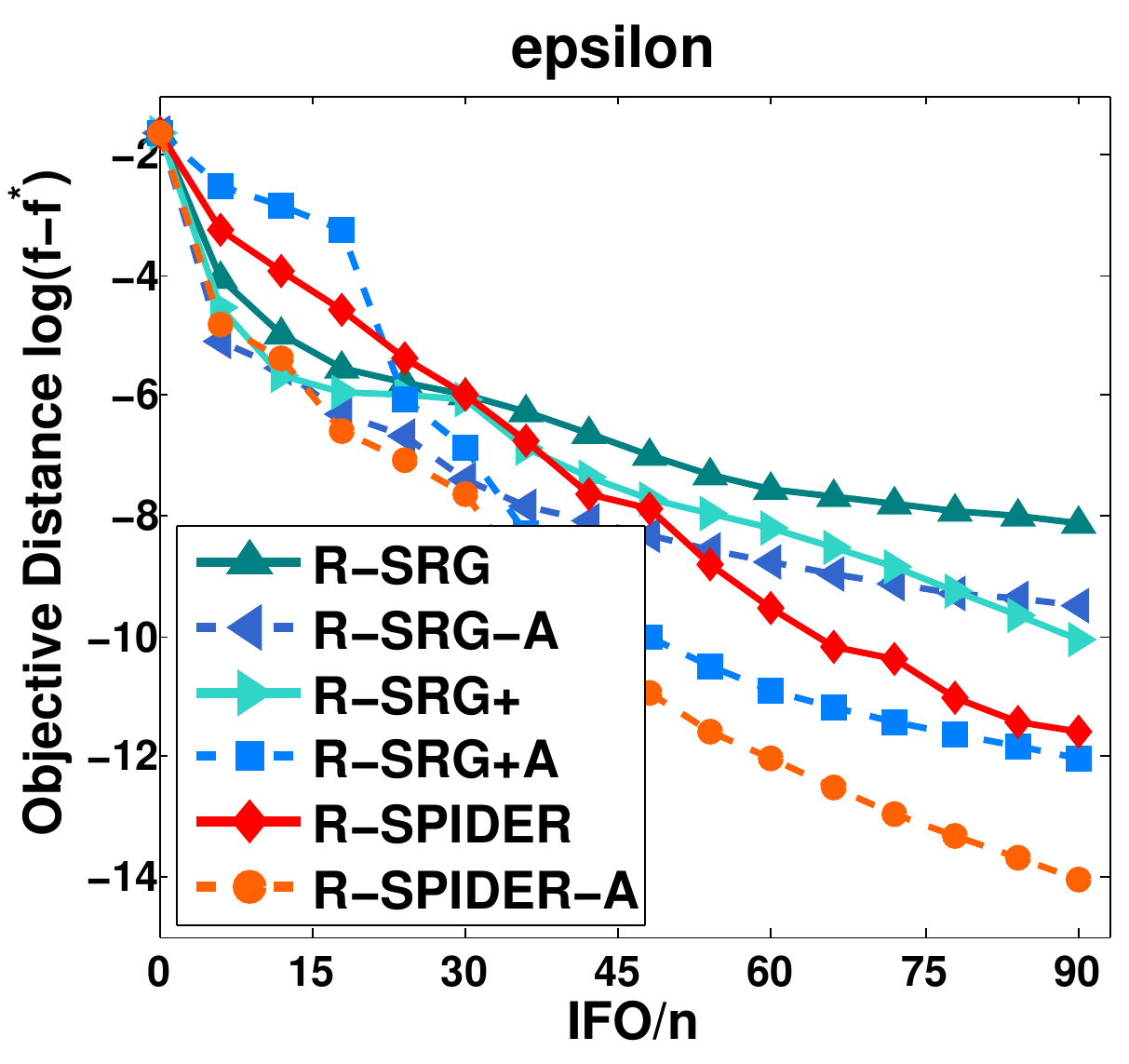}&
\includegraphics[width=0.2454\linewidth]{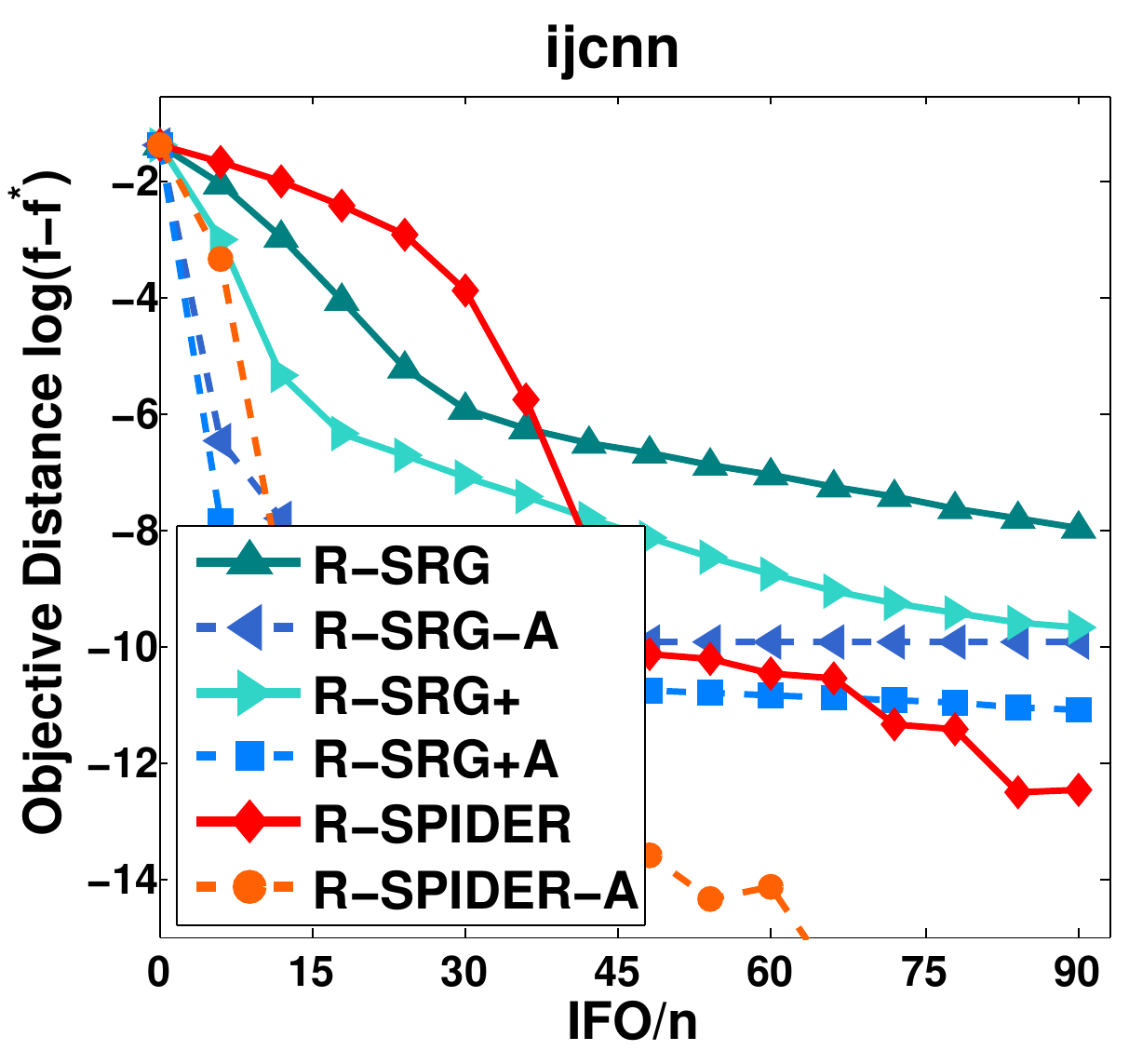}& \\
\multicolumn{4}{c}{{(a) Comparison among Riemannian stochastic gradient algorithms on $k$-PCA problem.}}\\
\includegraphics[width=0.2454\linewidth]{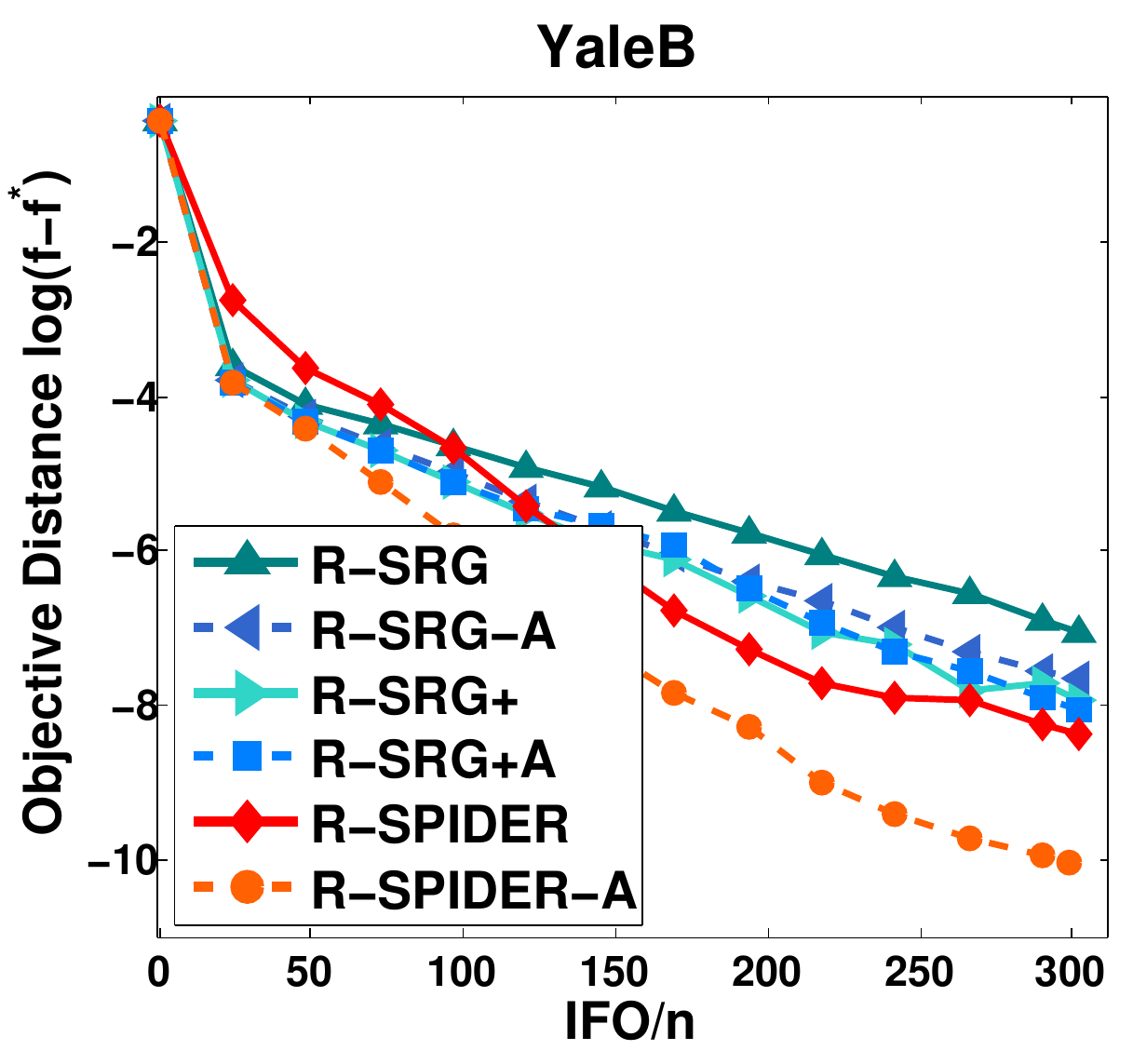}&
\includegraphics[width=0.2454\linewidth]{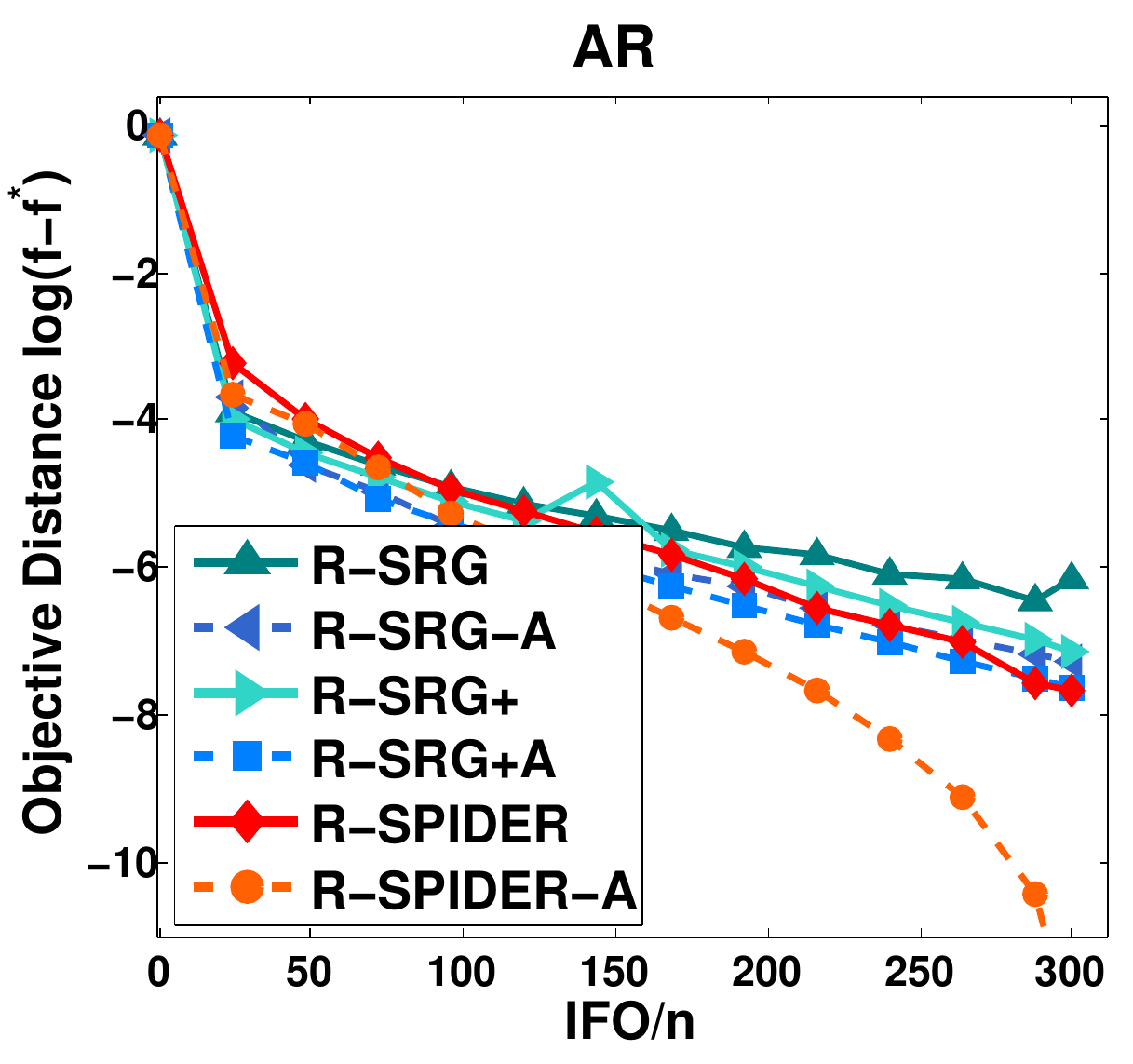}&
\includegraphics[width=0.2454\linewidth]{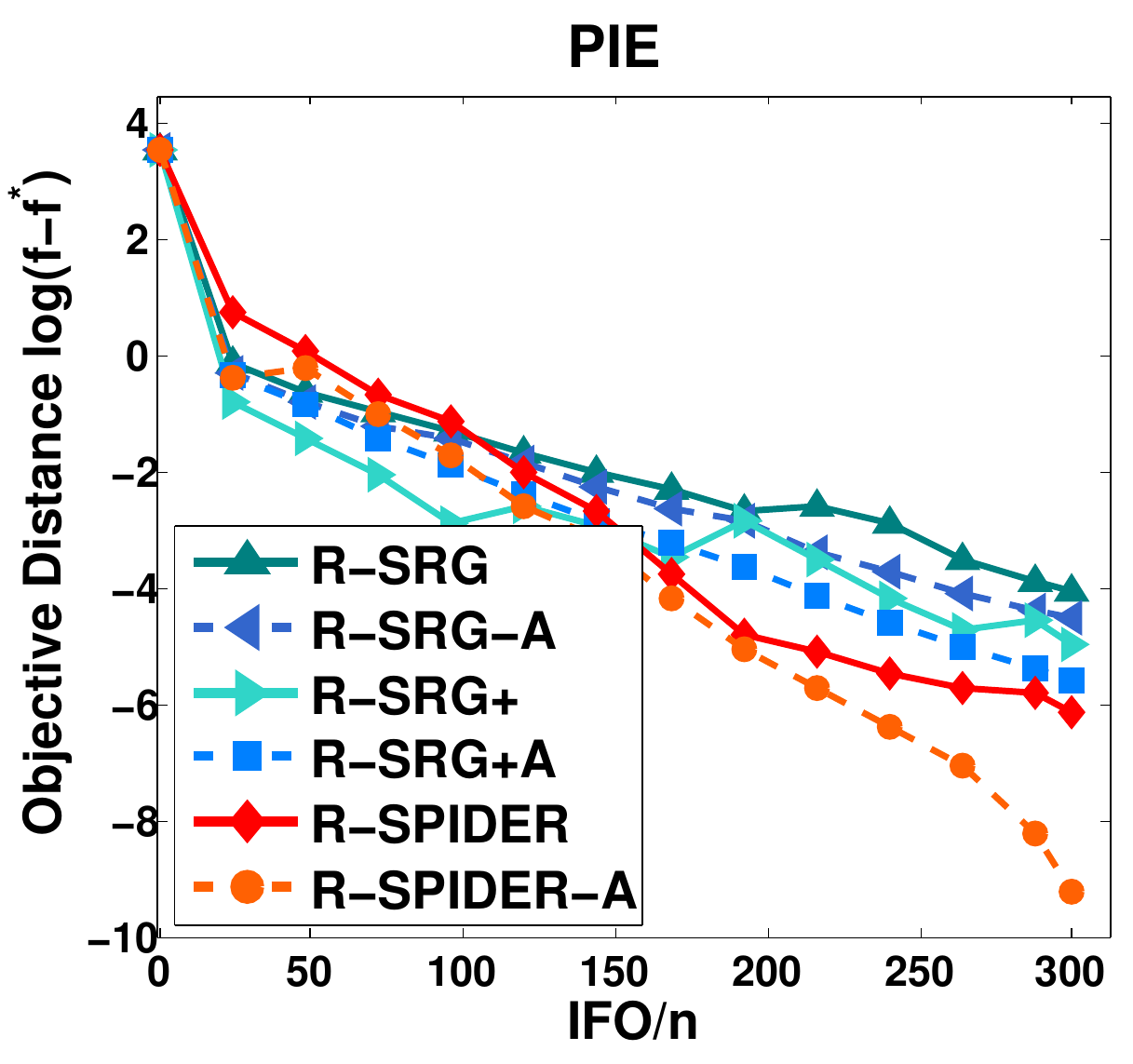}&
\includegraphics[width=0.2454\linewidth]{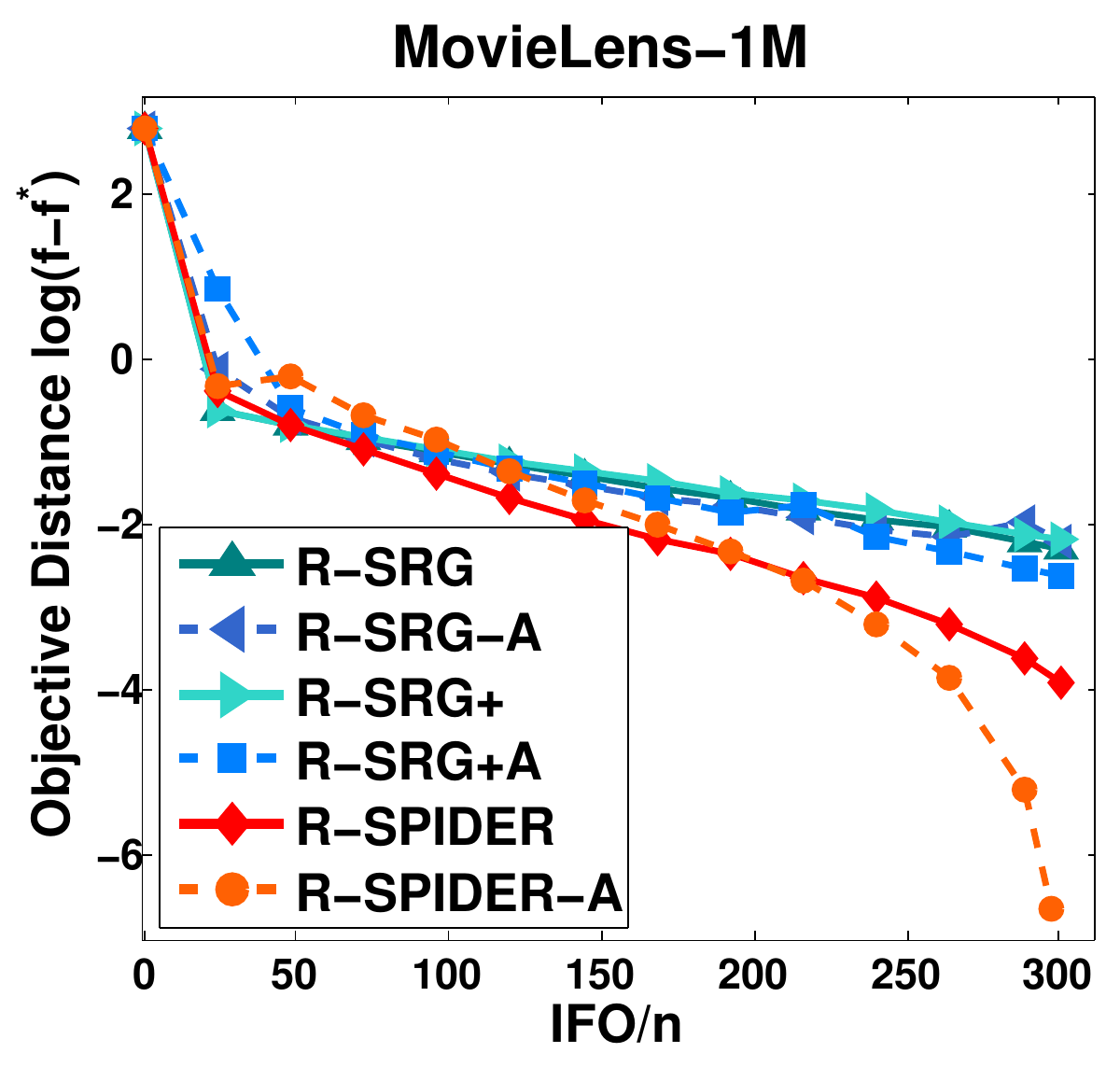}\\
\multicolumn{4}{c}{{(b) Comparison among Riemannian stochastic gradient algorithms on low-rank matrix completion problem.}}
\end{tabular}
\end{center}
\vspace{-1.2em}
\caption{More comparison between R-SPIDER and R-SRG with adaptive learning rates. } \label{comparsdfwe45twetgfasfwr}
\vspace{-1.0em}
\end{figure*}
\end{document}

%% file: MACPdef-ams_manu.tex
\newtheorem{thm}{Theorem}
\newtheorem{lem}{Lemma}
\newtheorem{cor}{Corollary}

\newtheorem{defn}{Definition}

\newtheorem{assum}{Assumption}


%% file: MACPdef_manu.tex
\DeclareMathOperator*{\argmin}{argmin}

\newcommand{\led}[1]{\overset{\text{\ding{#1}}}{\leq}}

\newcommand{\lee}[1]{\overset{\text{\ding{#1}}}{=}}
\newcommand{\ged}[1]{\overset{\text{\ding{#1}}}{\geq}}

\newcommand{\Rs}[1]{{\mathbb{R}^{#1}}}
\newcommand{\Oc}[1]{{\mathcal{O}\left(#1\right)}}

\newcommand{\lb}{\left(}
\newcommand{\rb}{\right)}
\newcommand{\la}{\left\langle}
\newcommand{\ra}{\right\rangle}

\newcommand{\g}{\mathfrak{g}}

\newcommand{\Tgi}[1]{{\mbox{T}_{#1}}}

\newcommand{\pii}[1]{ p^{{#1}}}
\newcommand{\fii}[2]{\nabla f_{{#1}_{#2}}}
\newcommand{\fsi}[1]{\nabla f_{\mathcal{S}_{#1}}}
\newcommand{\fsii}[2]{\nabla f_{\mathcal{S}_{#1}^{#2}}}
\newcommand{\dis}[2]{\mbox{d}\left(#1,#2\right)}
\newcommand{\diss}[2]{\mbox{d}^2\left(#1,#2\right)}

\newcommand{\Exp}[2]{\mbox{Exp}_{#1}\left(#2\right)}
\newcommand{\iExp}[2]{\mbox{Exp}^{-1}_{#1}\left(#2\right)}
\newcommand{\iExpp}{\mbox{Exp}^{-1}}
\newcommand{\Tran}[3]{\Gamma_{#1}^{#2}\left(#3\right)}

\newcommand{\Para}[3]{{\mbox{P}_{#1}^{#2}\left(#3\right)}}
\newcommand{\Paras}[2]{{\mbox{P}_{#1}^{#2}}}
\newcommand{\Parass}{{\mbox{P}}}

\newcommand{\0}{\bm{0}}

\newcommand{\xm}{\bm{x}}
\newcommand{\ym}{\bm{y}}

\newcommand{\zm}{\bm{z}}

\newcommand{\etai}[1]{\eta_{#1}}
\newcommand{\etati}[1]{\widetilde{\eta}_{#1}}
\newcommand{\ami}[1]{\bm{a}_{#1}}
\newcommand{\gmi}[1]{\bm{g}_{#1}}
\newcommand{\xmi}[1]{\bm{x}_{#1}}
\newcommand{\xms}{\widetilde{\bm{x}}}
\newcommand{\xmti}[1]{\widetilde{\bm{x}}_{#1}}

\newcommand{\vmi}[1]{\bm{v}_{#1}}
\newcommand{\vmti}[1]{{\bm{v}}_{#1}}
\newcommand{\vmhi}[1]{\widehat{\bm{v}}_{#1}}

\newcommand{\Am}{\bm{A}}

\newcommand{\Imm}{\bm{I}}
\newcommand{\Gm}{\bm{G}}
\newcommand{\Um}{\bm{U}}

\newcommand{\EE}{\mathbb{E}}

\newcommand{\A}{\mathcal{A}}

\newcommand{\Es}{\mathcal{E}}

\newcommand{\M}{\mathcal{M}}
\newcommand{\PP}{\mathcal{P}}

\newcommand{\V}{\mathcal{V}}

\newcommand{\SSm}{\mathcal{S}}

\newcommand{\Si}[2]{\mathcal{S}_{#1}^{#2}}

%% file: main.bbl
\begin{thebibliography}{10}

\bibitem{wold1987principal}
S.~Wold, K.~Esbensen, and P.~Geladi.
\newblock Principal component analysis.
\newblock {\em Chemometrics and intelligent laboratory systems}, 2(1-3):37--52,
  1987.

\bibitem{tan2014riemannian}
M.~Tan, I.~Tsang, L.~Wang, B.~Vandereycken, and S.~Pan.
\newblock {R}iemannian pursuit for big matrix recovery.
\newblock In {\em {Proc. Int'l Conf. Machine Learning}}, pages 1539--1547,
  2014.

\bibitem{vandereycken2013low}
B.~Vandereycken.
\newblock Low-rank matrix completion by {R}iemannian optimization.
\newblock {\em SIAM Journal on Optimization}, 23(2):1214--1236, 2013.

\bibitem{mishra2014r3mc}
B.~Mishra and R.~Sepulchre.
\newblock {R3MC}: A {R}iemannian three-factor algorithm for low-rank matrix
  completion.
\newblock In {\em {Proc. IEEE Conf. on Decision and Control}}, pages
  1137--1142, 2014.

\bibitem{kasai2016low}
H.~Kasai and B.~Mishra.
\newblock Low-rank tensor completion: a {R}iemannian manifold preconditioning
  approach.
\newblock In {\em {Proc. Int'l Conf. Machine Learning}}, pages 1012--1021,
  2016.

\bibitem{cherian2017riemannian}
A.~Cherian and S.~Sra.
\newblock Riemannian dictionary learning and sparse coding for positive
  definite matrices.
\newblock {\em {IEEE trans. on Neural Networks and Learning Systems}},
  28(12):2859--2871, 2017.

\bibitem{sun2017complete}
J.~Sun, Q.~Qu, and J.~Wright.
\newblock Complete dictionary recovery over the sphere ii: Recovery by
  {R}iemannian trust-region method.
\newblock {\em {IEEE Trans. on Information Theory}}, 63(2):885--914, 2017.

\bibitem{hosseini2015matrix}
R.~Hosseini and S.~Sra.
\newblock Matrix manifold optimization for {G}aussian mixtures.
\newblock In {\em {Proc. Conf. Neutral Information Processing Systems}}, pages
  910--918, 2015.

\bibitem{meyer2011linear}
G.~Meyer, S.~Bonnabel, and R.~Sepulchre.
\newblock Linear regression under fixed-rank constraints: a {R}iemannian
  approach.
\newblock In {\em {Proc. Int'l Conf. Machine Learning}}, 2011.

\bibitem{kasai2018riemannian}
H.~Kasai, H.~Sato, and B.~Mishra.
\newblock Riemannian stochastic recursive gradient algorithm with retraction
  and vector transport and its convergence analysis.
\newblock In {\em {Proc. Int'l Conf. Machine Learning}}, pages 2521--2529,
  2018.

\bibitem{zhang2016riemannian}
H.~Zhang, S.~Reddi, and S.~Sra.
\newblock Riemannian {SVRG}: Fast stochastic optimization on {R}iemannian
  manifolds.
\newblock In {\em {Proc. Conf. Neutral Information Processing Systems}}, pages
  4592--4600, 2016.

\bibitem{oja1992principal}
E.~Oja.
\newblock Principal components, minor components, and linear neural networks.
\newblock {\em Neural networks}, 5(6):927--935, 1992.

\bibitem{da1998geodesic}
J.~da~Cruz~Neto, L.~De Lima, and P.~Oliveira.
\newblock Geodesic algorithms in {R}iemannian geometry.
\newblock {\em Balkan Journal of Geometry and its Applications}, 3(2):89--100,
  1998.

\bibitem{badeau2005fast}
R.~Badeau, B.~David, and G.~Richard.
\newblock Fast approximated power iteration subspace tracking.
\newblock {\em {IEEE Trans. on Signal Processing}}, 53(8):2931--2941, 2005.

\bibitem{zhang2018estimate}
H.~Zhang and S.~Sra.
\newblock An estimate sequence for geodesically convex optimization.
\newblock In {\em {Proc. Conf. on Learning Theory}}, pages 1703--1723, 2018.

\bibitem{zhang2016first}
H.~Zhang and S.~Sra.
\newblock First-order methods for geodesically convex optimization.
\newblock In {\em {Proc. Conf. on Learning Theory}}, pages 1617--1638, 2016.

\bibitem{bonnabel2013stochastic}
S.~Bonnabel.
\newblock Stochastic gradient descent on {R}iemannian manifolds.
\newblock {\em IEEE Trans. Automatic Control}, 58(9):2217--2229, 2013.

\bibitem{kasai2016riemannian}
H.~Kasai, H.~Sato, and B.~Mishra.
\newblock {R}iemannian stochastic variance reduced gradient on {G}rassmann
  manifold.
\newblock {\em arXiv preprint arXiv:1605.07367}, 2016.

\bibitem{kasai2018riemanniana}
H.~Kasai, H.~Sato, and B.~Mishra.
\newblock {R}iemannian stochastic quasi-{N}ewton algorithm with variance
  reduction and its convergence analysis.
\newblock {\em {Prof. Int'l Conf. Artificial Intelligence and Statistics}},
  2018.

\bibitem{liu2017accelerated}
Y.~Liu, F.~Shang, J.~Cheng, H.~Cheng, and L.~Jiao.
\newblock Accelerated first-order methods for geodesically convex optimization
  on {R}iemannian manifolds.
\newblock In {\em {Proc. Conf. Neutral Information Processing Systems}}, pages
  4868--4877, 2017.

\bibitem{nesterov2013introductory}
Y.~Nesterov.
\newblock {\em Introductory lectures on convex optimization: A basic course}.
\newblock Springer Science \& Business Media, 2006.

\bibitem{SVRG}
R.~Johnson and T.~Zhang.
\newblock Accelerating stochastic gradient descent using predictive variance
  reduction.
\newblock In {\em {Proc. Conf. Neutral Information Processing Systems}}, pages
  315--323, 2013.

\bibitem{nguyen2017sarah}
L.~Nguyen, J.~Liu, K.~Scheinberg, and M.~Tak{\'a}{\v{c}}.
\newblock {SARAH}: A novel method for machine learning problems using
  stochastic recursive gradient.
\newblock {\em {Proc. Int'l Conf. Machine Learning}}, 2018.

\bibitem{nguyen2017stochastic}
L.~Nguyen, J.~Liu, K.~Scheinberg, and M.~Tak{\'a}{\v{c}}.
\newblock Stochastic recursive gradient algorithm for nonconvex optimization.
\newblock {\em arXiv preprint arXiv:1705.07261}, 2017.

\bibitem{fang2018spider}
C.~Fang, C.~Li, Z.~Lin, and T.~Zhang.
\newblock {SPIDER}: Near-optimal non-convex optimization via stochastic path
  integrated differential estimator.
\newblock {\em arXiv preprint arXiv:1807.01695}, 2018.

\bibitem{huang2015riemannian}
W.~Huang, P.~Absil, and K.~Gallivan.
\newblock A {R}iemannian symmetric rank-one trust-region method.
\newblock {\em Mathematical Programming}, 150(2):179--216, 2015.

\bibitem{huang2015broyden}
W.~Huang, K.~Gallivan, and P.~Absil.
\newblock A broyden class of quasi-{N}ewton methods for {R}iemannian
  optimization.
\newblock {\em SIAM Journal on Optimization}, 25(3):1660--1685, 2015.

\bibitem{polyak1963gradient}
B.~Polyak.
\newblock Gradient methods for the minimisation of functionals.
\newblock {\em USSR Computational Mathematics and Mathematical Physics},
  3(4):864--878, 1963.

\bibitem{nesterov2006cubic}
Y.~Nesterov and B.~Polyak.
\newblock Cubic regularization of {N}ewton method and its global performance.
\newblock {\em Mathematical Programming}, 108(1):177--205, 2006.

\bibitem{zhou2018stochastic}
D.~Zhou, P.~Xu, and Q.~Gu.
\newblock Stochastic nested variance reduction for nonconvex optimization.
\newblock {\em arXiv preprint arXiv:1806.07811}, 2018.

\bibitem{adler2002newton}
R.~Adler, J.~Dedieu, J.~Margulies, M.~Martens, and M.~Shub.
\newblock {N}ewton's method on {R}iemannian manifolds and a geometric model for
  the human spine.
\newblock {\em IMA Journal of Numerical Analysis}, 22(3):359--390, 2002.

\bibitem{absil2009optimization}
P.~Absil, R.~Mahony, and R.~Sepulchre.
\newblock {\em Optimization algorithms on matrix manifolds}.
\newblock Princeton University Press, 2009.

\bibitem{absil2012projection}
P.~Absil and J.~Malick.
\newblock Projection-like retractions on matrix manifolds.
\newblock {\em SIAM Journal on Optimization}, 22(1):135--158, 2012.

\bibitem{wen2013feasible}
Z.~Wen and W.~Yin.
\newblock A feasible method for optimization with orthogonality constraints.
\newblock {\em Mathematical Programming}, 142(1-2):397--434, 2013.

\bibitem{karimi2016linear}
H.~Karimi, J.~Nutini, and M.~Schmidt.
\newblock Linear convergence of gradient and proximal-gradient methods under
  the polyak-{\l}ojasiewicz condition.
\newblock In {\em Joint European Conference on Machine Learning and Knowledge
  Discovery in Databases}, pages 795--811. Springer, 2016.

\bibitem{candes2011robust}
E.~J. Cand{\`e}s, X.~Li, Y.~Ma, and J.~Wright.
\newblock Robust principal component analysis?
\newblock {\em Journal of the ACM}, 58(3):11, 2011.

\bibitem{georghiades2001few}
A.~Georghiades, P.~Belhumeur, and D.~Kriegman.
\newblock From few to many: Illumination cone models for face recognition under
  variable lighting and pose.
\newblock {\em IEEE Trans. on Pattern Analysis and Machine Intelligence},
  23:643--660, Jun. 2001.

\bibitem{AR}
A.~Martinez and R.~Benavente.
\newblock The {AR} face database.
\newblock {\em CVC Tech. Rep. 24}, Jun. 1998.

\bibitem{sim2003cmu}
T.~Sim, S.~Baker, and M.~Bsat.
\newblock The {CMU} pose, illumination, and expression database.
\newblock {\em IEEE Trans. on Pattern Analysis and Machine Intelligence},
  25:1615--1618, Dec. 2003.

\end{thebibliography}
